%% file: MLGM-I.tex
\documentclass[11pt]{amsart}

\usepackage{amsmath,amsfonts,amsthm,amssymb,amscd,mathabx,stmaryrd}
\usepackage{epsfig}
\usepackage{inputenc}
\usepackage{upref}
\usepackage{bm}
\usepackage{calc}
\usepackage{caption}
\usepackage{subcaption}
\usepackage{hyperref}
\usepackage[usenames,dvipsnames]{xcolor}
\usepackage[normalem]{ulem}
\usepackage{wasysym}
\usepackage{mathrsfs}
\usepackage{enumitem}
\usepackage{setspace}
\usepackage{array}
\usepackage{cancel}
\usepackage{tikz-cd}
\usepackage{colonequals}
\usepackage[titletoc]{appendix}
\usepackage{overpic}
\usepackage{float} 

\newtheorem{theorem}{Theorem}[section]
\newtheorem*{theorem*}{Theorem}	

\newtheorem{corollary}[theorem]{Corollary}
\newtheorem{lemma}[theorem]{Lemma}
\newtheorem{proposition}[theorem]{Proposition}

\theoremstyle{definition}

\newtheorem{defx}[theorem]{Definition}
\newenvironment{definition}
{\pushQED{\qed}\defx}
{\popQED\enddefx}

\newtheorem{examplex}[theorem]{Example}
\newenvironment{example}
{\pushQED{\qed}\examplex}
{\popQED\endexamplex}

\newtheorem{assumptx}[theorem]{Assumption}
\newenvironment{assumpt}
{\pushQED{\qed}\assumptx}
{\popQED\endassumptx}

\newtheorem{remarkx}[theorem]{Remark}
\newenvironment{remark}
{\pushQED{\qed}\remarkx}
{\popQED\endremarkx}

\numberwithin{equation}{section}

\DeclareMathOperator{\im}{Im}
\DeclareMathOperator{\re}{Re}

\DeclareMathOperator{\End}{End}

\DeclareMathOperator{\Id}{Id}
\DeclareMathOperator{\Hess}{Hess}

\DeclareMathOperator{\Hom}{Hom}
\DeclareMathOperator{\Sym}{Sym}

\DeclareMathOperator{\Ric}{Ric}

\DeclareMathOperator{\Map}{Map}

\DeclareMathOperator{\Ch}{CF}
\DeclareMathOperator{\Lie}{Lie}

\DeclareMathOperator{\Spin}{Spin}

\DeclareMathOperator{\map}{Map}

\DeclareMathOperator{\grad}{grad}
\DeclareMathOperator{\Crit}{Crit}

\DeclareMathOperator{\bd}{\mathbf{d}}

\DeclareMathOperator{\HFL}{HF}
\DeclareMathOperator{\Res}{Res}
\DeclareMathOperator{\HM}{\mathit{HM}}
\DeclareMathOperator{\KHM}{\mathit{KHM}}
\DeclareMathOperator{\TSigma}{\bf\Sigma}

\newcommand{\R}{\mathbb{R}}
\newcommand{\C}{\mathbb{C}}

\newcommand{\Z}{\mathbb{Z}}

\newcommand{\CP}{\mathbb{CP}}

\newcommand{\HH}{\mathbb{H}}
\newcommand{\T}{\mathbb{T}}

\newcommand{\BF}{\mathbb{F}}

\newcommand{\y}{\mathbb{Y}}

\newcommand{\Step}{\textit{Step }}

\newcommand{\half}{\frac{1}{2}}

\newcommand{\embed}{\hookrightarrow}
\newcommand{\pt}{\partial_t }
\newcommand{\ps}{\partial_s}

\newcommand{\dt}{\frac{d}{dt}}
\newcommand{\ds}{\frac{d}{ds}}

\newcommand{\CA}{\mathcal{A}}

\newcommand{\SC}{\mathcal{C}}

\newcommand{\D}{\mathcal{D}}
\newcommand{\E}{\mathcal{E}}

\newcommand{\CG}{\mathcal{G}}

\newcommand{\K}{\mathcal{K}}
\newcommand{\M}{\mathcal{M}}

\newcommand{\SO}{\mathcal{O}}

\newcommand{\SH}{\mathcal{H}}

\newcommand{\CL}{\mathcal{L}}

\newcommand{\CY}{\mathcal{Y}}

\newcommand{\fa}{\mathfrak{a}}

\newcommand{\F}{\mathfrak{F}}
\newcommand{\g}{\mathfrak{g}}

\newcommand{\s}{\mathfrak{s}}
\newcommand{\bs}{\widehat{\mathfrak{s}}}

\newcommand{\SL}{\mathscr{L}}

\newcommand{\bpartial}{\bar{\partial}}

\newcommand{\hy}{\widehat{Y}}

\newcommand{\cB}{\check{B}}
\newcommand{\cb}{\check{b}}
\newcommand{\cPsi}{\check{\Psi}}
\newcommand{\cvb}{\delta\check{b}}
\newcommand{\cvpsi}{\delta\check{\Psi}}

\newcommand{\Lm}{\zeta}

\newcommand{\spinc}{$spin^c\ $}

\newcommand{\dg}{\textbf{d}_\gamma}

\newcommand{\tX}{\tilde{\xi}}
\newcommand{\tY}{\tilde{\eta}}

\newcommand{\vdelta}{\vec{\delta}}
\newcommand{\Dt}{\nabla^A_{\partial_t}}
\newcommand{\Ds}{\nabla^A_{\partial_s}}

\newcommand{\EHess}{\widehat{\Hess}}

\title{Monopoles And Landau-Ginzburg Models I}
\author{Donghao Wang}

\date{\today}

\address{Department of Mathematics, Massachusetts Institute of Technology, Cambridge, MA 02139, USA}
\email{donghaow@mit.edu}

\begin{document}
	
	\begin{abstract} The endpoint of this series of papers is to construct the monopole Floer homology for any pair $(Y,\omega)$, where $Y$ is a compact oriented  3-manifold with toroidal boundary and $\omega$ is a suitable closed 2-form. In the first paper, we exploit the framework of the gauged Landau-Ginzburg models to address two model problems for the (perturbed) Seiberg-Witten moduli spaces on either $\C\times\Sigma$ or $\HH^2_+\times\Sigma$, where $\Sigma$ is any compact Riemann surface of genus $\geq 1$. Our first result states that finite energy solutions to the perturbed equations on $\C\times\Sigma$ are necessarily trivial. The second states that small energy solutions on $\HH^2_+\times\Sigma$ necessarily have energy decaying exponentially in the spatial direction. These results will lead eventually to the compactness theorem in the second paper \cite{Wang20}. 
	\end{abstract}
	
	\maketitle
	\tableofcontents

\include{MLGM-I-Introduction_4.0}
\include{LGmodels_5.0}

\input{SWEQ_6.0}

\newpage
\appendix

\input{maximumprinciple}

\input{BFormula}
\input{higher_1.0}

\bibliographystyle{alpha}
\bibliography{sample}

\end{document}

%% file: MLGM-I-Introduction_4.0.tex
\part{Introduction}
\section{Introduction}
\subsection{Motivations in Floer Homology} The Seiberg-Witten Floer homology of closed oriented 3-manifolds as introduced by Kronheimer-Mrowka \cite{Bible} has greatly influenced the study of 3-manifold topology since its inception. The underlying idea is an infinite-dimensional Morse theory: solutions to the Seiberg-Witten equations on a closed 3-manifold $Z$ are critical points of the Chern-Simons-Dirac functional $\CL$, while solutions on the 4-manifold $\R_t\times Z$ are downward gradient flowlines of $\CL$. 

\smallskip

The purpose of this series of papers is to generalize their construction and  define the Seiberg-Witten Floer homology for any pair $(Y,\omega)$, where $Y$ is a compact oriented 3-manifold  with toroidal boundary and $\omega$ is a suitable closed 2-form. This analytic construction will recover the monopole link Floer homology for any link $L\subset S^3$ and produce potentially new invariants for the link complement. In the third paper \cite[Theorem 1.10]{Wang203},  we will establish a (3+1) TQFT structure comparable to the one for link Floer homology \cite{Ju16,Zem19, L18}. 

\smallskip

The long range goal of this line of work is to develop a bordered version of monopole Floer homology, analogous to the bordered Heegaard Floer theory of Lipshitz-Ozsv\'{a}th-Thurston \cite{BHF}, but via analytic methods.  In the process of attacking this problem, the author realized that the geometric framework needed to understand the Seiberg-Witten equations (SWEQ) on the product 4-manifold $\C_z\times\Sigma$, where $\Sigma$ is any compact Riemann surface of genus $\geq 1$, is gauged Landau-Ginzburg models \cite{Witten93}. This interesting dictionary, which links gauge theory with symplectic topology, however, does not seem to be well-known in the literature. In short, this dictionary says that any \spinc structure $\s$ on $\Sigma$ specifies an infinite dimensional K\"{a}hler manifold $M(\Sigma,\s)$ together with a holomorphic function 
\[
W: M(\Sigma, \s)\to \C,
\]
then the Seiberg-Witten equations on $\C_z\times\Sigma$ can be interpreted as the complex gradient flow equation \eqref{E0.1} of $W$ defined on $\C_z$. 

\smallskip

This paper focuses on the already interesting problem of the analysis of these equations: we exploit this dictionary to establish two basic analytic results that are essential to all forthcoming development. The first one states that finite energy solutions to the Seiberg-Witten equations on $\C_z\times\Sigma$, once perturbed by a suitable 2-form, must be trivial, while the second states that small energy solutions on $\R_t\times [0,\infty)_s\times\Sigma$ must have energy decaying exponentially as $s\to\infty$. These results are completely general with regards to the genus of $\Sigma$ and will play an important role in the proof of the compactness theorem in the second paper \cite{Wang20}. This 2-form is to perturb $W: M(\Sigma,\s)\to \C$ into a holomorphic Morse function (interpreted suitably in the infinite dimensional case), a crucial property in order for these results to hold.

\subsection{Summary of Results}\label{Subsec1.2} To state our main theorems, we first outline a monopole Floer theory for 3-manifolds with cylindrical ends. Given a compact oriented 3-manifold $Y$ with boundary $\partial Y=\Sigma$, let $g_Y$ be a metric cylindrical near $\Sigma$. Set $g_\Sigma\colonequals g_Y|_\Sigma$. In \cite{NguyenI, NguyenII},  Nguyen studied the Seiberg-Witten equations directly on the 3-manifold $Y$ and laid the analytic foundation for a Floer theory with Lagrangian boundary conditions along $\partial Y$. We will work instead with a 3-manifold with cylindrical ends:
\[
\hy=Y\ \bigcup\ [0,\infty)_s\times \Sigma.
\]
 This approach is also adopted in the PhD thesis of Yang \cite{Y99} and in an unpublished manuscript by Mrowka-Ozsv\'{a}th-Yu. In the latter case, they investigated Seifert-fibered spaces with some regular fibers removed in attempt to generalize their earlier work \cite{MOY97}.
 
 \smallskip
 
  As we shall focus on the geometry of the cylindrical end of $\hy$ in this paper, it is harmless to assume that $\Sigma$ is connected and $g(\Sigma)\geq 1$ from now on. A solution to the Seiberg-Witten equations on $[0,\infty)_s\times\Sigma$ is a flowline of $-\nabla \re W$ on the K\"{a}hler manifold $M(\Sigma,\s)$. The key observation is that the functional $\re W$ is in fact the real part of a holomorphic function $W:M(\Sigma,\s)\to \C$. Without any perturbation, the critical set of $\re W$ modulo gauge is a symmetric product of the surface $\Sigma$. While it is tempting to develop a Floer theory with respect to a Lagrangian submanifold of this symmetric product, we follow a different approach in this paper and perturb $W$ into a holomorphic Morse function using a harmonic 1-form $\lambda\in \Omega^1_h(\Sigma, i\R)$; the $(1,0)$-part $\lambda^{1,0}\in \Omega^{1,0}(\Sigma, \C)$ is required to have $(2g-2)$ simple zeros. 
  
  Returning to the 3-manifold $\hy$, we choose a closed 2-form $\omega\in \Omega^2(\hy, i\R)$ to perturb the 3-dimensional Seiberg-Witten equations such that $\omega=\nu+ds\wedge\lambda$ on $[0,\infty)_s\times \Sigma$ for some 2-form $\nu\in \Omega^2(\Sigma, i\R)$ on $\Sigma$.
  
  \smallskip

This setup is inspired by the work of Meng-Taubes \cite{MT96}, and this quadruple $\TSigma=(\Sigma,g_\Sigma,\lambda,\nu)$ will be called an $H$-surface in this paper. We are interested in \spinc structures $\s$ on $Y$ with
\[
c_1(\s)[\Sigma]=2(d+g(\Sigma)-1) \text{ for some integer } 0\leq d\leq 2 g(\Sigma)-2.
\] 

 Our first observation concerns these perturbed equations on the standard cylinder $\R_s\times\Sigma$.

\begin{lemma}[Proposition \ref{stableW}]\label{L1.1} For any $H$-surface $\TSigma=(\Sigma, g_\Sigma,\lambda,\nu)$, consider the perturbed Seiberg-Witten equations \eqref{3DDSWEQ} on the 3-manifold $\R_s\times\Sigma$ with $\omega=\nu+ ds\wedge\lambda$. Then for the \spinc structure $\s$ with $c_1(\s)[\Sigma]=2(d-g(\Sigma)+1)$, there are precisely $
	\binom{2g-2}{d}$ $\R_s$-invariant solutions up to gauge. These solutions are irreducible.  
\end{lemma}

\begin{remark} When $g(\Sigma)=1$, Lemma \ref{L1.1} is used by Meng-Taubes \cite{MT96} to define the 3-dimensional Seiberg-Witten invariants of $Y$ (even though this property is not stated explicitly in their paper); see \cite[Section 2.1-2.3]{MT96}. The 4-dimensional analogue of Lemma \ref{L1.1} is \cite[Lemma 3.1]{Taubes01}. Lemma \ref{L1.1} is also used to compute the Seiberg-Witten invariants of the product manifold $S^1\times S^1\times \Sigma_g$, a well-known result since the very beginning of this subject. See \cite[(4.11)--(4.16)]{Witten94}.
\end{remark}

Lemma \ref{L1.1} says that the critical set of $W$ modulo gauge is (at least) discrete after this perturbation. Let $\fa_j, 1\leq j\leq k=\binom{2g-2}{d}$ be a collection of the special solutions to the Seiberg-Witten equations obtained in Lemma \ref{L1.1}, one for each gauge-equivalence class. For any $H$-surface $\TSigma$, the goal is to define the Floer homology of $Y$ relative to the pair $(\TSigma,\fa_j)$:
\[
\HM_*(Y,\fa_j). 
\]


To this end, consider finite energy solutions on $\hy$ that approximate $\fa_j$ along the cylindrical end $[0,\infty)_s\times \Sigma$. As critical points of the perturbed Chern-Simons-Dirac functional $\CL_\omega$ (cf. \cite[Definition 3.8]{Wang20}), they become non-degenerate after a further perturbation and form a compact moduli space of dimension $0$.

\smallskip

Now we have to analyze the 4-dimensional equations on the product manifold $\R_t\times \hy$ carrying a planar end $\HH^2_+\times \Sigma$, where
$
\HH^2_+\colonequals \R_t\times [0,\infty)_s
$
is furnished with the Euclidean metric. Our convention is to use $t$ for the time coordinate and $s$ for the spatial coordinate on the cylindrical end of $\hy$. The Seiberg-Witten moduli space on $\R_t\times \hy$ has another possible source of non-compactness coming from this planar end: for a sequence of Seiberg-Witten monopoles on $\R_t\times \hy$, some amount of energy might slide off along the cylindrical end of $\hy$ and converges to a monopole on $\R_t\times\R_s\times\Sigma$ with finite positive energy in the limit.

This is the sort of phenomena we must prevent in order to define a Floer homology, which is addressed successfully in Theorem \ref{T1.2} below. In some sense, the 4-manifold $\R_t\times \hy$ is non-compact in two directions, but the energy can slide off only in the time direction. 

\begin{theorem}\label{T1.2} For any $H$-surface $\TSigma=(\Sigma, g_{\Sigma},\lambda, \nu)$, consider the \spinc structure $\s$ on $\Sigma$ with $c_1(\s)[\Sigma]=2(d-g(\Sigma)+1)$. Then any finite energy solution to the perturbed Seiberg-Witten equations \eqref{SWEQ} on $\C_z\times \Sigma$ with $\omega=\nu+ds\wedge\lambda$, which is also called point-like, is gauge equivalent to the trivial solution obtained by extending $\fa_j$ constantly over $\C_z$ with $ z=t+is$, for some $1\leq j\leq \binom{2g-2}{d}$. 
\end{theorem}

In this sense, we say that point-like solutions on $\C_z\times \Sigma$ are trivial. In contrast, for the unperturbed Seiberg-Witten equations (with $\omega=0$), interesting point-like solutions do exist and were classified completely in an earlier work of the author \cite{W18}. 

\smallskip

The second problem that we address is the decay rate in the spatial direction of $\R_t\times \hy$. We state the result for the planar end $\HH^2_+\times\Sigma$. For any $n\in \Z$ and $R\in [1,\infty)$, consider the rectangle $ \Omega_{n,R}\colonequals [n-1,n+1]\times [R-1,R+1]\subset \HH^2_+$. Denote by
\begin{equation}\label{E0.4}
\E_{an}(\gamma;  \Omega_{n,R}) 
\end{equation}
the analytic energy of the configuration $\gamma$ on the 4-manifold $ \Omega_{n,R}\times\Sigma$, called the local energy functional (see Lemma \ref{D7.3}). $\E_{an}(\gamma;  \Omega_{n,R})$ is the integral of the energy density function of $\gamma$ over $\Omega_{n,R}$ and therefore is non-negative and gauge-invariant; it bounds the $L^2_1$-norm of $\gamma$ (with a gauge fixing condition) on $ \Omega_{n,R}\times\Sigma$ and also the $L^2_k$-norms of $\gamma$ in the interior of $ \Omega_{n,R}\times\Sigma$ for any $k\geq 2$ when $\gamma$ is a solution. 

\begin{theorem}\label{T1.3} For any $H$-surface $\TSigma=(\Sigma, g_\Sigma,\lambda,\nu)$, there exist constants $\epsilon(\TSigma),\zeta(\TSigma)>0$ with the following significance. Suppose $\gamma$ solves the perturbed Seiberg-Witten equations \eqref{SWEQ} on $\HH^2_+\times\Sigma$ with $\omega=\nu+ds\wedge\lambda$ and $\E_{an}(\gamma,  \Omega_{n,R})<\epsilon$ for any $n\in \Z$ and $R\geq 1$, then 
	\[
	\E_{an}(\gamma; \Omega_{n,R})<e^{-\zeta R}. 
	\] 
\end{theorem}

From Theorem \ref{T1.3}, one can easily deduce the exponential decay of $L^2_k$-norms of a solution $\gamma$ on $\HH^2_+\times\Sigma$. Since the spatial direction ($s\to \infty$) is not the one for the downward gradient flow of $\CL_\omega$, Theorem \ref{T1.3} is not a consequence of the standard theory, e.g. \cite[Section 13]{Bible}.

\begin{remark}\label{R1.3} In the second paper \cite{Wang20}, we will carry out the construction of $\HM_*(Y,\fa_j)$ in detail for the special case that $(\Sigma,g_\Sigma)$ is a union of flat 2-tori and $\nu$ is harmonic. While the general case is almost identical, the invariance of this Floer homology with $g(\Sigma)\geq 2$ is not well-understood and is left as a topic for future research.
\end{remark}

\subsection{Gauged Landau-Ginzburg Models}\label{Subsec1.3} The proofs of Theorem \ref{T1.2} and \ref{T1.3} exploit the fundamental relation between the Seiberg-Witten equations and the gauged Witten equations, which links gauge theory with symplectic topology, and reducing Theorem \ref{T1.2} and \ref{T1.3} to some finite dimensional problems.

The gauged Witten equations were first introduced by Witten in his formulation of gauged
linear sigma models \cite{Witten93} to explain the so-called
Landau-Ginzburg/Calabi-Yau correspondence. Its mathematical foundation has been recently developed by Tian-Xu \cite{TX18,TX18-2}, to which readers are referred for the background. Since our focus is slightly different, we give a short discussion below with emphasis on Picard-Lefschetz theory.

When the dimension is finite and the structure group $G$ is trivial, a Landau-Ginzburg model is a pair $(M, W)$ where
\begin{itemize}
	\item $M$ is a  complete non-compact K\"{a}hler manifold, and
	\item $W=L+iH: M\to \C$ is a holomorphic function, called the superpotential.
\end{itemize}

The Landau-Ginzburg model $(M,W)$ is called \textit{Morse} if $L\colonequals \re W: M\to \R$ Morse function, so $(M,W)$ defines a Lefschetz fibration over $\C$. From a symplectic viewpoint, one may define its Fukaya-Seidel category $\CA$ using Lagrangian Floer theory, following \cite{S08}, and associate to each compact Lagrangian submanifold $\CL_0\subset M$ an $A_\infty$-module over $\CA$. This construction is based on the Cauchy-Riemann equation, or more conveniently, the complex gradient flow equation,
\begin{equation}\label{E0.1}
\pt P+J\ps P+\nabla H=0,
\end{equation}
where $H\colonequals \im W$ and $P: \R_t\times [0,1]_s\to M$ is subject to Lagrangian boundary conditions. 

\smallskip

We wish to generalize this picture in two directions, starting with
\begin{enumerate}
\item $M$ is acted on by an abelian Lie group $G$ with a moment map $\mu:M\to \g$, and the superpotential $W$ is $G$-invariant. The triple $(M,W,G)$ is then called a gauged Landau-Ginzburg model; see Section \ref{1.2} for the precise definition.
\end{enumerate}

The replacement of \eqref{E0.1} is the gauged Witten equations: 
\begin{equation}\label{E0.2}
\left\{ \begin{array}{rl}
-*_2F_A+\mu&=\vdelta,\\
\nabla^A_{\partial_t}P +J\nabla^A_{\partial_s}P+\nabla H&=0,
\end{array}
\right.
\end{equation}
where $A$ is a connection on the trivial $G$-bundle $Q$ over $\R_t\times[0,1]_s$ and $\vdelta\in \g$ is a perturbation of the moment map $\mu$. The map $P$ is now regarded as a section of the trivial bundle $Q\times_G M$. The replacement of the \textbf{Morse} condition is a notion of \textbf{stability}, cf. Definition \ref{stable}. In this context, the local energy functional \eqref{E0.4} is defined as 
\begin{equation}\label{E0.5}
\E_{an}(A,P; \Omega_{n,R})=\int_{ \Omega_{n,R}} |\nabla^A P|^2+|\nabla H|^2+|F_A|^2+|\mu-\vdelta|^2.
\end{equation}
for any $ \Omega_{n,R}\subset \HH^2_+$. In particular, $\E_{an}(A, P;  \Omega_{n,R})=0$ implies that up to gauge, $A$ is the trivial connection and $P$ is a constant map with values in $\mu^{-1}(\vdelta)\cap \Crit(H)$. 

\medskip

Here comes the second generalization:
\begin{enumerate}[label=(2)]
	\item the gauged Landau-Ginzburg model $(M,W,G)$ can be infinite-dimensional.
\end{enumerate}

The proofs of Theorem \ref{T1.2} and \ref{T1.3} start with their counterparts for finite-dimensional Landau-Ginzburg models (as toy problems) and are concluded by the following observation. 

\begin{proposition}[Proposition \ref{P6.2} \& \ref{stableW}]\label{P1.4} For any $H$-surface $\TSigma=(\Sigma,g_\Sigma,\lambda,\nu)$, consider the \spinc structure $\s$ on $\Sigma$ with $c_1(\s)[\Sigma]=2(d+g(\Sigma)-1)$. Then we can construct an infinite-dimensional gauged Landau-Ginzburg model $(M(\Sigma,\s), W_{\lambda}, \CG(\Sigma))$ such that the associated gauged Witten equations on $\C_z$ recover the perturbed Seiberg-Witten equations \eqref{SWEQ} on $\C_z\times \Sigma$ with $\omega=\nu+ds\wedge\lambda$. Moreover, this Landau-Ginzburg model is stable in the sense of Proposition \ref{stableW}. The critical locus of  the superpotential $W_\lambda$ are precisely given by the free $\CG_\C(\Sigma)$-orbits of $\fa_1,\cdots, \fa_k$ with $k=\binom{2g-2}{d}$. 
\end{proposition}

\begin{remark} From the standpoint of the gauged Witten equations \eqref{E0.2}, the roles played by the perturbation $\lambda$ and $\nu$ are quite different: the 1-form $\lambda$ is to perturb the superpotential $W_\lambda$, while the 2-form $\nu$ is to perturb the moment map equation in \eqref{E0.2} by changing $\vdelta\in \g$.
\end{remark}

 Roughly speaking, Theorem \ref{T1.2} and \ref{T1.3} hold in general for any \textbf{stable} gauged Landau-Ginzburg models. The only difference in the infinite-dimensional case is that the metric of $M$ depends on a Sobolev completion, so one has to specify the correct norms involved in the estimates. The plotline of proofs are summarized in the table below:

\begin{center}
	\begin{tabular}{ | c | c |c| } 
		\hline
	\multicolumn{2}{|c|}{$\dim<\infty$} & $\dim=\infty$ \\ 
		\hline
	 $G=\{e\}$ & $G\neq \{e\}$& (SWEQ) on $\C\times\Sigma$ or $\HH^2_+\times\Sigma$\\ 
	 \hline
		Lemma  \ref{solutionsonC}& Theorem \ref{point-like-solutions} & Theorem \ref{pointlike} (Theorem \ref{T1.2}) \\
		\hline 
			Lemma \ref{Lemma-exponentialdecay} & Theorem \ref{exponential-decay} & Theorem \ref{T9.1} (Theorem \ref{T1.3}) \\ 
		\hline
	\end{tabular}
\end{center}

\smallskip

For instance, in the case that $M=\C, G=S^1$ and $W\equiv 0$, the gauged Witten equations $\eqref{E0.2}$ defined on $\C_z=\R_t\times\R_s, z=t+is$ come down to the vortex equations of Taubes \cite{Taubes} (with $\vdelta=\frac{i}{2}$):
\begin{equation}\label{E0.3}
\left\{ \begin{array}{rl}
*_2iF_A+\half |P|^2&=\half,\\
\bpartial_AP &=0,
\end{array}
\right.
\end{equation}
where $A$ is a $U(1)$-connection on $\C_z$ and $P:\C_z\to \C$ is a complex valued function. This example is not stable in the sense of Definition \ref{stable}, and Theorem \ref{point-like-solutions} fails by \cite[Theorem 1]{Taubes}: the vortex moduli space on $\C_z$ is $\coprod_{k\geq 0}\Sym^k \C_z$. However, Theorem \ref{exponential-decay} still holds, which says that the local energy functional $\E_{an}(A,P; \Omega_{n,R})$ defined by \eqref{E0.5} has exponential decay as $R\to \infty$. In fact, this decay is pointwise, so it recovers a theorem of Taubes \cite{Taubes-thesis}:
\begin{theorem}[{\cite[P.59, Theorem 1.4]{JT}}]\label{JT} Let $(A,P)$ be a smooth finite energy solution to the vortex equations $\eqref{E0.3}$. Given any $\epsilon>0$, there exists $C=C(\epsilon,A,P)>0$ such that
	\[
	0\leq *_2 i F_A=\half(1-|P|^2)<Ce^{-(1-\epsilon)\sqrt{t^2+s^2}}.
	\]
\end{theorem} 

Our proof of Theorem \ref{exponential-decay} is an adaptation of Taubes' argument \cite[Proposition 6.1]{JT}: we first establish a Bochner-type formula (Lemma \ref{Bochner}) for the energy density functional of $(A,P)$ and then apply the maximum principle; see Remark \ref{B.10} for such a comparison. 

\subsection{Future Directions: Fukaya-Seidel Categories}\label{Subsec1.4} The complex gradient flow equation \eqref{E0.1} is the Cauchy-Riemann equation perturbed by the special Hamiltonian function $H=\im W$ and therefore exhibits interesting analytical properties due to the holomorphicity of $W$. Meanwhile, its solutions admit concrete geometric interpretation, suggesting a direct construction of Fukaya-Seidel categories even for the infinite dimensional gauged Landau-Ginzburg model in Proposition \ref{P1.4}. This may lead eventually to a bordered monopole theory. To develop the Picard-Lefschetz theory in this context, one may work directly with infinite-dimensional Lagrangian submanifolds as in the work of Nguyen \cite{NguyenI, NguyenII}. But there might be a shortcut: \textit{can we define the Floer cohomology of Lagrangian submanifolds without using boundary conditions?}

\smallskip

When $G$ is trivial and $\dim M<\infty$, this idea can be partly realized when the Lagrangian submanifold $\CL_0$ is a Lefschetz thimble, i.e. the stable (or unstable) submanifold of a critical point $q\in \Crit(\re W)$. Instead of a strip $\R_t\times [0,1]_s$, we look at solutions to the complex gradient flow equation \eqref{E0.1} on the upper half plane:
\[
P: \HH^2_+=\R_t\times [0,\infty)_s\to M
\]
with a Lagrangian boundary condition at $s=0$ and with an asymptotic condition as $s\to \infty$: we require the solution $P(t, s)$ to converge to some $q\in \Crit(\re W)$ uniformly as $s\to\infty$. The study of Fukaya-Seidel categories via this approach has been pioneered by the work of Haydys \cite{Haydys15} and Gaiotto-Moore-Witten \cite{GMW15}. See also \cite{FJY18,GMW17,KKS16}. We give a brief sketch of their proposal in Section \ref{1.1}. The primary application in their cases is when 
\begin{align*}
M&=\text{the space of }SL(2,\C)\text{ connections on a closed oriented 3-manifold } Y,\\
W&=\text{the complex valued Chern-Simons functional},
\end{align*}
and the gauged Witten equations \eqref{E0.2} come down to the Haydys-Witten equations on the 5-manifold  $\C_z\times Y$. This idea goes back at least to the seminar paper \cite{DT96} by Donaldson-Thomas: if one takes 
\begin{align*}
M&=\text{the space of $\bpartial$-operators on a complex vector bundle }E\to \CY,\\
W&=\text{the holomorphic Chern-Simons functional}.
\end{align*}
where $\CY$ is a compact Calabi-Yau 3-fold, then the gauged Witten equations \eqref{E0.2} recover the $\Spin(7)$-instanton equation on $\C_z\times \CY$.

\smallskip

The Seiberg-Witten equations (SWEQ) should serve as the test for their programs in higher dimensions. To pursue a bordered monopole Floer theory along this line, we have to construct 
\[
\begin{array}{ccccl}
\TSigma&\rightsquigarrow& (M(\Sigma,\s), W_\lambda, \CG(\Sigma)) &\rightsquigarrow &\text{ an $A_\infty$ category }\CA;\\
Y&\rightsquigarrow&\text{(SWEQ) on }\hy &\rightsquigarrow &\text{ an $A_\infty$ right module $\M(Y)$ over } \CA;\\
Y'&\rightsquigarrow&\text{ (SWEQ) on }\hy' &\rightsquigarrow &\text{ an $A_\infty$ left module $\M(Y')$ over } \CA, 
\end{array}
\]
where $\TSigma=(\Sigma, g_\Sigma,\lambda,\nu)$ is any $H$-surface, and $(M(\Sigma,\s), W_\lambda, \CG(\Sigma)) $ is the gauged Landau-Ginzburg model in Proposition \ref{P1.4}. Here $Y,\ Y'$ are oriented 3-manifolds with $\partial Y\cong \Sigma\cong (-Y')$. Denote by $Y\circ_h Y'$ the closed 3-manifold obtained by gluing $Y$ and $Y'$ along their common boundary. The ultimate goal of this project is to establish a gluing theorem relating $(\M(Y),\M(Y'))$ with a suitable version of monopole Floer homology of $Y\circ_h Y'$ perturbed by a closed 2-form. The monopole Floer homology group
\[
\HM_*(Y,\TSigma)\colonequals \bigoplus \HM_*(Y,\fa_j),
\]
outlined as in Section \ref{Subsec1.2}, should in principle compute the homology of the underlying chain complexes of $\M(Y)$. The group $\HM_*(Y,\TSigma)$ may not be a topological invariant of $Y$: the full $A_\infty$ structure might be required to formulate this invariance. A more detailed discussion on this bordered picture will continue in Section \ref{Subsec2.3} below. 


\subsection{Organization} Section \ref{1.1} is a continuation of this introduction. The construction of the monopole Floer homology group $\HM_*(Y,\fa_j)$ is based on a variant of Lagrangian Floer cohomology as a finite dimensional model. Following the work of Haydys \cite{Haydys15} and Gaiotto-Moore-Witten \cite{GMW15}, we work instead with ``holomorphic upper half planes" with a boundary condition at $s=0$ and an asymptotic condition as $s\to\infty$. Also, we introduce Lemma  \ref{solutionsonC} \& \ref{Lemma-exponentialdecay} as the toy model of Theorem \ref{T1.2} and \ref{T1.3}, and explain their significance for the compactness theorem. The discussion in Section \ref{1.1} is largely informal; no proofs will be presented. 

In Part \ref{Part1}, we introduce the notion of gauged Landau-Ginzburg models. The focus is on the geometric insights that motivate definitions and proofs in the infinite-dimensional setting. In Section \ref{1.4}, we prove that point-like solutions must be trivial provided that the superpotential $W$ is stable (Theorem \ref{point-like-solutions}). In Section \ref{1.5}, we prove the exponential decay result, Theorem \ref{exponential-decay}, using a Bochner-type formula for the energy density functional. 

In the last part of this paper, we construct the fundamental Landau-Ginzburg Model associated to an $H$-surface $\TSigma=(\Sigma, g_\Sigma,\lambda,\nu)$ and prove Theorem \ref{T1.2}
 and \ref{T1.3} by generalizing Theorem \ref{point-like-solutions} and \ref{exponential-decay} from Part \ref{Part1}. \\
 
 \textbf{Acknowledgments.} The author is extremely grateful to his advisor, Tom Mrowka, for his patient help and constant encouragement throughout this project.  The author would like to thank Chris Gerig and Jianfeng Lin for their comments on a preliminary version of this paper. The author would like to thank Tim Large and Paul Seidel for several discussions and for providing the critical references, and also Ao Sun for teaching him the proof of maximum principle. Finally, the author would like to thank the anonymous referees for many valuable suggestions to improve the overall presentation of this paper. This material is based upon work supported by the National Science Foundation under Grant No.1808794. 
 
\section{A General Overview}\label{1.1}

In this section, we explain a variant of Lagrangian Floer cohomology defined by counting holomorphic upper half planes, following \cite{Haydys15} and \cite{GMW15}. This serves as the finite dimensional model of the monopole Floer theory for 3-manifolds with boundary, proposed in the introduction. This analogy is only used as an inspiration or a guideline for future research; it is not our intention to relate these two theories in a precise way. To make it consistent with the existing literature, we will work with \textbf{cohomology} instead of homology in this section.

 Analytically, the perturbed Seiberg-Witten equations on $\HH^2_+\times\Sigma$ and the complex gradient flow equation \eqref{P-holo} on $\HH^2_+$ have many features in common. Lemma \ref{solutionsonC} and Lemma \ref{Lemma-exponentialdecay} below are the counterparts of Theorem \ref{T1.2} and \ref{T1.3} in this toy model. Their proofs are postponed to the next part, where the corresponding results, Theorem \ref{point-like-solutions} and \ref{exponential-decay}, are stated and proved for the gauged Witten equations.

\subsection{A Variant of Lagrangian Floer Cohomology}
Recall that a Landau-Ginzburg Model is a pair $(M,W)$ where
\begin{itemize}
	\item $(M, \omega, J, g)$ is a non-compact complete  K\"{a}hler manifold with complex structure $J$ and K\"{a}hler metric $h\colonequals g-i\omega_M$. The underlying Riemannian metric is $g$, and $\omega_M$ is the symplectic form. 
	\item $W: M\to \C$ is a holomorphic function, called the superpotential.
\end{itemize}

Since $M$ is K\"{a}hler, $J$ is parallel. Write $W=L+i H$ with $L=\re W$ and $H=\im W$. Then the Cauchy-Riemann equation $(dW)^{0,1}=0$ comes down to
\begin{equation}\label{CRequation}
\nabla L+J\nabla H=0,
\end{equation}
which says the gradient $\nabla L$ is the Hamiltonian vector field of $H$. 

A Landau-Ginzburg model $(M,W)$ is called \textit{Morse} if all critical points of $L$ are non-degenerate. We always assume $(M,W)$ is Morse in this section. Denote $\Crit(L)$ by the set of critical points of $L$. Taking the covariant derivative of $(\ref{CRequation})$ yields:
\begin{equation}\label{Hess}
\Hess L+ J\circ \Hess H=0.
\end{equation}

Since $\Hess H$ is a symmetric operator and $J$ is skew-symmetric, (\ref{Hess}) implies that
\begin{equation}\label{anti-commutivity}
J\circ \Hess L+\Hess L\circ J=0. 
\end{equation}

For any $q\in \Crit(L)$, denote by $H_q^\pm\subset T_q M$ the positive (negative) spectral subspace of $\Hess_q L$. Then (\ref{anti-commutivity}) implies $J(H^\pm_q)= H^\mp_q$. In particular, the index of $q$ is $(n,n)$ if $\dim_\R M=2n$. Let $U_q$ and $S_q$ be the unstable and stable submanifolds of $q$ respectively, which are also called Lefschetz thimbles: 
\begin{align*}
U_q&=\{x\in M: \exists p:(-\infty,0]_s\to M,  \ps p+\nabla L=0, p(0)=x, \lim_{s\to-\infty} p=q\},\\
S_q&=\{x\in M: \exists p:[0,\infty)_s\to M,  \ps p+\nabla L=0, p(0)=x, \lim_{s\to\infty} p=q\}.
\end{align*}

\begin{lemma}[{\cite[Lemma 2.5]{Eli97}\cite[Lemma 1.13]{S03}}]\label{L2.1} $U_q$ and $S_q$ are Lagrangian submanifolds of $(M,\omega_M)$.
\end{lemma}

Lemma \ref{L2.1} has been well-known for a long time, but the Floer theory of Lefschetz fibrations was only explored after the works of Donaldson \cite{D99} and Seidel \cite{S03}; see \cite[Remark 16.10]{S08}. 
\begin{assumpt}\label{assumption} To simplify our exposition in this section, we make the following assumptions on $(M,W)$.
	\begin{itemize}
\item  $(M,\omega_M)$ is an exact symplectic manifold, i.e. $\omega_M=d\theta_M$ for a smooth 1-form $\theta_M\in \Omega^1(M)$ called the primitive;
\item The superpotential $W: M\to \C$ need not to be proper, but the function $|\nabla H|^2:M\to [0,\infty)$ is;
\item $L$ has finitely many critical points $\{q_1,\cdots, q_k\}$ on $M$ with $H(q_i)$ mutually distinct. \qedhere
	\end{itemize}
\end{assumpt}

\begin{example}\label{example0}
	Let $M=\C^n$ and $W=z_1^2+\cdots+z_n^2,\ (z_1,\cdots,z_n)\in \C^n$. Then the unique critical point $q$ is the origin.
\end{example}

Take a pair of \textbf{compact} Lagrangian submanifolds $\CL_0, \CL_1\subset M$ intersecting transversely. Let $\Ch^*(\CL_0,\CL_1)$ be the $\BF_2$-vector space freely generated by the intersection $\CL_0\cap \CL_1$:
\[
\Ch^*(\CL_0,\CL_1)=\bigoplus_{y\in \CL_0\cap \CL_1} \BF_2\cdot y.
\]

The differential $\partial $ on $\Ch^*(\CL_0,\CL_1)$ is defined by counting $J$-holomorphic strips of Maslov index $1$ and with Lagrangian boundary conditions. They are smooth maps
\[
P: \R_t\times [0,1]_s\to M 
\]
satisfying the equation
\begin{equation}\label{J-holo}
\pt P+ J\ps P=0,
\end{equation}
along with the boundary conditions $P(\cdot, 0)\in \CL_0$ and $P(\cdot, 1)\in \CL_1$. At this point, some assumptions on $M$ and $(\CL_0, \CL_1)$ are required to show that $\partial$ is well-defined after suitable perturbations, but let us skip these technical steps here.

\smallskip

The equation (\ref{J-holo}) can be perturbed by a Hamiltonian function. In our case, we use the imaginary part of $W$:
\begin{equation}\label{P-holo}
\pt P+ J\ps P+\nabla H=0. 
\end{equation}
The cochain complex $\Ch^*(\CL_0,\CL_1)$ is then generated by Hamiltonian chords, which are smooth maps $p: [0,1]_s\to M$ satisfying the relations
\[
p(0)\in \CL_0, p(1)\in \CL_1, J\ps p+\nabla H=0.
\]
By (\ref{CRequation}), the last condition is equivalent to
\begin{equation}\label{DGFlow}
0=J(\ps p+\nabla L), s\in [0,1]_s. 
\end{equation}
i.e. $p$ is a downward gradient flowline of $L$.

\smallskip

The Lagrangian Floer cohomology $\HFL^*(\CL_0,\CL_1)$ is the homology of the cochain complex $(\Ch^*(\CL_0,\CL_1),\partial)$. The underlying idea is an infinite-dimensional Morse theory. The configuration space is the path space
\[
C^\infty([0,1], M; \CL_0,\CL_1)\colonequals \{p:[0,1]_s\to M: p\text{ smooth, } p(0)\in \CL_0, p(1)\in\CL_1 \},
\]
and the Morse function defined on $\SC^\infty([0,1], M; \CL_0,\CL_1)$ is the perturbed symplectic action functional:
\[
\CA_H(p)=\CA(p)+\int_{[0,1]_s}H\circ p(s)ds.
\]
A path $p$ is a critical point of $\CA_H$ if and only if $p$ is a Hamiltonian chord. An $\R_t$-family of paths $\{p_t\}_{t\in \R}\subset C^\infty([0,1], M; L_0, L_1)$ forms a downward gradient flowline of $\CA_H$ precisely when $P(t,s)=p_t(s)$ solves the equation (\ref{P-holo}) on $\R_t\times [0,1]_s$.

\smallskip

We would like to generalize this setup for certain \textbf{non-compact} Lagrangian submanifolds, in particular for the unstable and stable thimbles $U=U_q$ and $S=S_q$ at any $q\in \Crit(L)$. This can not be done without further assumptions on the submanifolds $\CL_j, j=0,1$ in order to control their geometry at infinity. Instead of making them precise, we give an (incomplete) list of axiomatic properties for these submanifolds:
\begin{enumerate}
	\item there are two classes of non-compact Lagrangian submanifolds, $\SC_{un}$ and $\SC_{st}$, which are of the unstable type and the stable type respectively;
	\item for any $q\in \Crit(L)$, $U_q\in \SC_{un}$ and $S_q\in \SC_{st}$, where $U_q$ and $S_q$ are Lefschetz thimbles at $q$;
	\item for any $\CL_U\in \SC_{un}$ and $\CL_S\in \SC_{st}$, $L=\re W$ is bounded above on $\CL_U$ and below on $\CL_S$. $\CL_U$ only intersect $\CL_S$ within a compact region of $M$;
	\item the Lagrangian Floer cohomology $\HFL^*(\CL_U,\CL_S)$ is well-defined, assuming transversality, by counting Hamiltonian chords and solutions of (\ref{P-holo}).
\end{enumerate}

Our goal is to give an alternative construction of $\HFL^*(\CL_U, S_q)$ and $\HFL^*(U_q,\CL_S)$ for all $q\in \Crit(L)$, and we focus on the first case for simplicity. Suppose $\CL_U$ is exact, so the primitive $\theta_M|_{\CL_U}=dh$ for some function $h: \CL_U\to \R$. The cochain group $\Ch^*(\CL_U, S_q)$ is generated by the finite set $ \CL_U\cap S_q$. Each $x\in \CL_U\cap S_q$ corresponds to a path $p: [0,\infty)_s\to M$ such that
\[
\ps p+\nabla L=0,\ p(0)=x,\ \lim_{s\to\infty} p(s)=q. 
\] 

The analogy with $(\ref{DGFlow})$ suggest that we look at the Morse cohomology of the perturbed action functional:
\[
\CA_H(p)=-h\big(p(0)\big)+\int_{[0,\infty)_s} -p^*\theta_M+\big(H\circ p(s)-H(q)\big)ds,
\]
defined on the space
\[
\SC^\infty([0,\infty), M; \CL_U)\colonequals \{p:[0,\infty)_s\to M: p\text{ smooth, } p(0)\in \CL_U, \lim_{s\to\infty}p(s)=q \}.
\]

The differential $\widehat{\partial}$ on $\Ch^*(\CL_U, S_q)$ is defined by counting solutions to (\ref{P-holo}) on the upper half plane $\HH^2_+=\R_t\times [0,\infty)_s$ with the boundary condition:
\begin{equation}\label{boundarycondition}
P(t, 0)\in \CL_U,\ \lim_{s\to\infty} P(t, s)=q,\ \forall t\in \R_t. 
\end{equation}

To prove $\widehat{\partial}^2=0$, it is important to know a compactness property. The proof of the next lemma will be carried out in a forthcoming paper \cite{Wang22}. It is included here only to motivate our work:
\begin{lemma}[{\cite[Lemma 3.7]{Wang22}}]\label{Lemma-uniformdecay} With some additional assumptions on $(M,W)$, for any unstable type Lagrangian submanifold $\CL_U$, one can find a function $\eta: [0,\infty)_s\to [0,\infty)$ such that $\lim_{s\to\infty}\eta(s)=0$, and for any solution $P: \HH^2_+\to M$ of $(\ref{P-holo})$ subject to the boundary condition $(\ref{boundarycondition})$, we have
	\[
	\sup_{t\in \R} d\big(P(t,s),q\big)\leq \eta(s),
	\]
	where $d$ is the distance function of the Riemannian metric $g$. 
\end{lemma}

This means that the convergence in the boundary condition (\ref{boundarycondition}) is also uniform for all possible solutions $P$. In fact, this decay is exponential. The next lemma is the toy model of Theorem \ref{T1.3} when $G$ is trivial and $\dim M<\infty$. 

\begin{lemma}\label{Lemma-exponentialdecay} There exist constants $\epsilon(M,W), \zeta(M,W)>0$ with the following significance. Let $P_1:\HH^2_+\to M$ be any solution of $(\ref{P-holo})$ such that for some $q\in \Crit(L)$, 
$ d\big(P_1(t,s),q\big)< \epsilon, \forall (t,s)\in \HH^2_+$, then
	\[
	d\big(P_1(t,s),q\big)<e^{-\zeta s}, \forall s\geq 0.
	\]
\end{lemma}

The exponent $\zeta(M,W)$ is determined by the first positive eigenvalue of $\Hess_q L$. To derive the exponential decay of the distance function from Lemma \ref{Lemma-exponentialdecay}, set $P_1(t,s)=P(t, s+R)$ for some $R\gg0 $ in Lemma \ref{Lemma-uniformdecay}. 
\begin{remark}
Apparently, Lemma \ref{Lemma-exponentialdecay} holds when $P$ is time-independent, since in this case $\{P(t,s)\}_{s\in[0,\infty)_s}$ is a downward gradient flowline of $L=\re W$ for any fixed $t\in \R_t$, and $L$ is a Morse function. It is not clear to the author whether $\{P(\cdot, s)\}_{s\in [0,\infty)}$ forms a downward gradient flowline (in the spatial direction) of some functional in general.
\end{remark}

The proof of uniform decay of the distance function in Lemma \ref{Lemma-uniformdecay} will rely on the following fact:
\begin{lemma}\label{solutionsonC} Let $P: \C_z\to M$ be a solution of $(\ref{P-holo})$  on the complex plane with finite energy. If $P(z)\to q$ as $z=t+is\to \infty$, then $P\equiv q$. 
\end{lemma}

\begin{remark}\label{R1.7}
	As we will see later, for a gauged Landau-Ginzburg model, Lemma \ref{solutionsonC} is not true in general. An appropriate condition that ensures Lemma \ref{solutionsonC} is a notion of stability, cf. Definition \ref{stable}.
\end{remark}

\begin{remark} We will prove Lemma \ref{solutionsonC} and Lemma \ref{Lemma-exponentialdecay} in the context of gauged Landau-Ginzburg models in Part \ref{Part1}, cf. Theorem \ref{point-like-solutions} and Theorem \ref{exponential-decay}. Under Assumption \ref{assumption}, the conclusion of Lemma \ref{Lemma-uniformdecay} will hold, if one can verify an additional local compactness property; see \cite[Section 3]{Wang22}. The analogue of Lemma \ref{Lemma-uniformdecay} in the context of the Seiberg-Witten equations is \cite[Theorem 6.3]{Wang20}. The main results of this paper, however, do not relies on this lemma.
\end{remark}

\subsection{A Gluing Theorem} The Fukaya-Seidel category associated to the Landau-Ginzburg model $(M,W)$ is a directed $A_\infty$-category $\CA$ with objects given by the thimbles $\{U_q\}_{q\in \Crit(L)}$ or $\{S_q\}_{q\in \Crit(L)}$. Moreover, for each $\CL_U\in \SC_{un}$ and $\CL_S\in \SC_{st}$, we assign:
\begin{align*}
\CL_U&\rightsquigarrow \text{ an }A_\infty \text{-right module over }\CA,\\
 \CL_S&\rightsquigarrow \text{ an }A_\infty \text{-left module over }\CA,
\end{align*}
whose underlying cochain complexes are given respectively by 
\begin{equation}\label{E2.8}
\bigoplus_{q\in \Crit(L)} \Ch^*(\CL_U,S_q) \text{ and } \bigoplus_{q\in \Crit(L)} \Ch^*(U_q,\CL_S).
\end{equation}

A theorem of Seidel \cite[Corollary 18.27]{S08} then suggests a spectral sequence whose $E_1$-page is 
\begin{equation}\label{E2.9}
\bigoplus_{q\in \Crit(L)}\HFL^*(\CL_U, S_q)\otimes \HFL^*(U_q, \CL_S),
\end{equation}
and which abuts to $\HFL^*(\CL_U, \CL_S)$ in the $E_\infty$-page. The underlying geometric picture was first observed by Donaldson and elaborated later in the monograph \cite{GMW15} by Gaiotto-Moore-Witten. The Lagrangian Floer cohomology $\HFL^*(\CL_U, \CL_S)$ is defined by counting holomorphic strips of width $1$ with Lagrangian boundary conditions, but one can instead work with strips of width $R$ for any $R>0$, and let $R\to \infty$. 

\smallskip

\textbf{This neck-stretching picture} makes more sense with the complex gradient flow equation \eqref{P-holo}, instead of \eqref{J-holo}. In \cite{GMW15}, a general framework is proposed to understand the limit of solutions of \eqref{P-holo} on $\R_t\times [0,R]_s$ in $C^\infty_{loc}$-topology as $R\to\infty$, and also to rebuild the differential on the cochain complex $\Ch^*(\CL_U, \CL_S)$ with $R\gg 1$ by gluing all limiting pieces. For instance, the differential on the $E_0$-page of \eqref{E2.9} is obtained by gluing a holomorphic upper half plane contributing to $\Ch^*(\CL_U,S_q)$ with a generator of $\Ch^*(U_q, \CL_S)$, or vice versa. This explains why the $E_1$-page of \eqref{E2.9} is a tensor product. We refer interested readers to \cite{GMW17} for a nice exposition of their proposal. 

\smallskip

If the $A_\infty$-category $\CA$ turns out to be trivial, then the spectral sequence collapses after the $E_1$-page. There is one simple geometric condition that yields this triviality:
\begin{lemma}\label{L2.9} If for any $e^{i\theta}\in S^1$ the downward gradient flowline equation 
	\[
	\ps p(s)+\nabla\big(\re(e^{-i\theta}W)\big)=0,\ p:\R_s\to M
	\]
	has only constant solutions, then the Fukaya-Seidel category of $(M,W)$ is trivial. 
\end{lemma}

As we shall see in Corollary \ref{C9.4}, this condition can be verified for the Seiberg-Witten equations for the special case in Remark \ref{R1.3} when $\langle \nu, [\Sigma]\rangle\neq 0$.

Although this neck-stretching picture is enlightening to keep in mind, the analytic foundation of the web-based formalism \cite{GMW15, GMW17} is still missing. To implement their proposal for the Seiberg-Witten equations remains a challenging problem.

\subsection{Relations with Gauge Theory}\label{Subsec2.3} As noted in Section \ref{Subsec1.4}, our goal is to define the Floer cohomology for Lagrangian submanifolds without using boundary conditions, and we have achieved this goal partly for thimbles by considering holomorphic upper half planes. (The boundary condition on $\R_t\times \{1\}\subset \R_t\times [0,1]$ is now removed). To deal with the remaining boundary component $\R_t\times \{0\}\subset \HH^2_+$, we explain the geometric origin of $\CL_U$ and $\CL_S$ in our primary applications.

 Suppose a closed oriented 3-manifold $Z=Y_L\ \bigcup_\Sigma Y_R$ is separated by a closed connected surface $\Sigma$ with $g(\Sigma)\geq 1$, and
\[
\im \big(H_1(\Sigma; \R)\to H_1(Z; \R)\big)\neq 0. 
 \]
 This means that $Y_L, Y_R$ are oriented 3-manifolds with $\partial Y_L\cong \Sigma\cong \partial(-Y_R)$ and $Z$ is obtained by gluing $Y_L$ with $Y_R$ along their common boundary. Let $M(\Sigma,\s_0)$ be the infinite-dimensional K\"{a}hler manifold associated to $\Sigma$ in Proposition \ref{P1.4} with $c_1(\s_0)[\Sigma]=0$. As shown in the work of Nguyen \cite{NguyenI}, the solution space of 3-dimensional Seiberg-Witten equations on $Y_L$ (resp. $Y_R$), when restricted to $\Sigma$, gives rise to an infinite-dimensional Lagrangian submanifold of $M(\Sigma,\s_0)$, which we denote by $\CL_U$ (resp. $\CL_S$). The monopole Floer cohomology 
 \[
 \HM^*(Y_L,\TSigma)\colonequals \bigoplus_{j} \HM^*(Y_L,\fa_j)
 \]
  of $Y_L$, as we sketched in Section \ref{Subsec1.2}, is the analogue of 
\[
\bigoplus_{q\in \Crit{L}} \HFL^*(\CL_U, S_q)
\]
defined using holomorphic upper half planes. By working with the extended 3-manifold 
\[
\hy_L=Y_L\ \bigcup\  [0,\infty)_s\times \Sigma
\]
and the 4-manifold $\R_t\times \hy_L$, the remaining boundary condition along $\R_t\times \{0\} \subset \HH^2_+$ is also removed.

\smallskip

To see the relation of this monopole Floer theory with the monopole knot Floer cohomology, recall the construction from \cite{KS}. For any knot $K\subset S^3$, take a meridian $m\subset S^3\setminus K$. Then the link complement 
\[
Y_K\colonequals S^3\setminus N(m\cup K). 
\]
is a 3-manifold with boundary $\partial Y_K=\Sigma\colonequals\Sigma_1\cup \Sigma_2$. In this case, $\Sigma$ is disconnected, with each $\Sigma_i\cong \T^2$. Using a suitable orientation reversing diffeomorphism $\varphi: \Sigma_1\to\Sigma_2$, we close up the boundary of $Y_K$ and obtain a closed oriented 3-manifold $Z(S^3, K; \varphi)$. Then define
\[
\KHM^*(S^3, K)\colonequals \HM^*(Z(S^3, K; \varphi)),
\]
where $\HM^*(Z(S^3, K; \varphi))$ is the monopole Floer cohomology of $Z(S^3, K; \varphi)$ defined using a suitable local coefficient system. It is shown in \cite{KS} that $\KHM^*(S^3, K)$ is independent of the isotopy class of $\varphi$ up to isomorphisms. 

On the other hand, one may equip the boundary $\Sigma=\partial Y_K$ with suitable differential forms to obtain a (disconnected) $H$-surface $\TSigma$, and take $\HM^*(Y_K, \TSigma)$ as a candidate of the knot Floer cohomology of $K$. To see its relation with $\KHM^*(S^3, K)$, regard $\varphi$ as gluing two pieces:
\[
Y_K \text{ and } [0,R]_s\times \Sigma_1 
\] 
As $R\to\infty$, we are stretching the metric in a neighborhood of $\Sigma_1$ in $Z(S^3, K;\varphi)$, which is similar to the neck stretching picture underlying Seidel's spectral sequence \eqref{E2.9}. The condition of Lemma \ref{L2.9} can be verified in this case (Corollary \ref{C9.4}), so one may recover $\KHM^*(S^3, K)$ from $\HM^*(Y_K, \TSigma)$ using an internal gluing theorem. This isomorphism is established in \cite[Theorem 1.9]{Wang203} using a different argument. 

\smallskip

As an ending remark for this expository section, the monopole Floer homology of 3-manifolds with toroidal boundary to be defined in the second paper \cite{Wang20} only gives the underlying cochain complexes \eqref{E2.8}. The construction of $A_\infty$-structures is left as an interesting future project; see \cite{Wang22} for details. 

%% file: LGmodels_5.0.tex
\part{Gauged Landau-Ginzburg Models}\label{Part1}

In this part, we generalize the setup from the previous section by allowing an abelian Lie group $G$ act on the K\"{a}hler manifold $M$. In this case, we obtain the gauged Witten equations \eqref{generalized-vortex-equation} as the generalization of the Floer equation \eqref{P-holo}. Theorem \ref{point-like-solutions} and Theorem \ref{exponential-decay} are the analogue of Theorem \ref{T1.2} and Theorem \ref{T1.3} in the finite-dimensional case; their proofs are presented in Section \ref{1.4} and Section \ref{1.5} respectively.

\section{Definitions and Examples}\label{1.2}

We generalize the setup from the previous section and introduce the notion of gauged Landau-Ginzburg Models.
\begin{definition}\label{LGmodels}
The quadruple $(M,W,G,\rho)$ is called an (abelian) gauged Landau-Ginzburg model if 
\begin{enumerate}
\item $(M,\omega, J, g)$ is a complete non-compact K\"{a}hler manifold with complex structure $J$ and K\"{a}hler metric $h\colonequals g-i\omega_M$;  $g$ is the underlying Riemannian metric, and $\omega_M$ the symplectic form. 
\item $(G,\rho)$ is a compact abelian Lie group acting on $M$ preserving the K\"{a}hler structure,  i.e. for any $g\in G$, the action $\rho(g): M\to M$ is a holomorphic isometry;
\item $(G,\rho)$ is a Hamiltonian group action, and it admits a moment map:
\[
\mu: M\to\g,
\]
 where $\g$ is the Lie algebra of $G$. Since $G$ is abelian, $\mu$ is $G$-invariant;
 \item The action of $(G,\rho)$ extends to an action of the complex group $(G_\C,\rho_\C)$.  $\rho_\C: G_\C\times M\to M$ is holomorphic. $\rho_\C$ does not preserve the Riemannian metric $g$ in general.
 \item $W: M\to \C$ is a $G_\C$-invariant holomorphic function called the superpotential. Write $W=L+i H$ with $L=\re W$ and $H=\im W$. \qedhere
\end{enumerate}
\end{definition}

Again, we assume $(M,\omega)$ is an exact symplectic manifold, i.e. $\omega_M=d\theta_M$ for some $\theta_M\in \Omega^1(M)$.  For any $\xi\in \g$, let $\tilde{\xi}$ be the vector field on $M$ induced from the group action $(G,\rho)$:
\[
\tilde{\xi}(p)=\frac{d}{dt} \rho(e^{t\xi})p\bigg|_{t=0}. 
\]

We adopt a non-standard (sign) convention of the moment map in this paper:
\begin{equation}\label{2.1}
\iota(\tilde{\xi})\omega=-d\langle \mu, \xi\rangle_\g.
\end{equation}
Since $\omega(\cdot,\cdot)=g(J\cdot, \cdot)$, (\ref{2.1}) is equivalent to
\begin{equation}\label{convention}
\langle \nabla\mu, \xi\rangle_\g=\nabla \langle \mu, \xi\rangle _\g=-J\tilde{\xi},
\end{equation}
where $\nabla\mu\in \Gamma(M, TM\otimes \g)$ is a $\g$-valued vector field on $M$ and $\langle \cdot,\cdot\rangle_\g$ denotes a bi-invariant metric of $\g$. 
\begin{example}\label{example1}
	Let $G=S^1, G_\C=\C^*, M=\C$ and $W\equiv 0$. The group action is the standard complex multiplication. Using our sign convention (\ref{convention}), the moment map  $\mu(x)=\frac{i}{2}|x|^2$ for all $x\in \C$. 
\end{example}
\begin{example}
	Let $G=S^1, G_\C=\C^*, M=S^2$ and $W\equiv 0$. Identify $M$ with $\CP^1=\C\cup \{\infty\}$. The action $\rho_\C$ is the same as in Example \ref{example1}. $\{0, \infty\}$ is the fixed point set of $\rho_\C$.  
\end{example}
\begin{example}\label{example2}
	Let $G=S^1, G_\C=\C^*, M=\C^2=\{(x,y)|x,y\in \C\}$ and $W(x,y)=xy$. $W$ becomes $G_\C$-invariant if we set the action $\rho_\C$ as \[
	\rho_\C(u)(x,y)=(ux,u^{-1}y),
	\]
	for any $u\in \C^*$. The moment map is $\mu(x,y)=\frac{i}{2}(|x|^2-|y|^2)$. 
\end{example}

Just as in Assumption \ref{assumption}, we wish $(W,\mu)$ to satisfy some good properties. The replacement of the Morse condition for gauged Landau-Ginzburg models is a notion of stability. There are two possible candidates; the second one turns out to be more useful.
\begin{definition}\label{W-stable}
	A vector $\vec{\delta}\in \g$ is called $W$-stable if $\mu^{-1}(\vec{\delta})\cap \Crit(L)$ consists of finitely many free $G$-orbits and $L|_{\mu^{-1}(\vec{\delta})}$ is Morse-Bott. In particular, the function induced by $L=\re W$ on the symplectic quotient
	\[
	\mu^{-1}(\vec{\delta})/G
	\]
	is a Morse function with finitely many critical points. 
\end{definition}


However, having a $W$-stable vector $\vdelta$ is not sufficient to deduce the triviality of point-like  solutions in Theorem \ref{point-like-solutions} (see Example \ref{vortex} below). Note that the critical set $\Crit(L)=\{x\in M: \nabla L(x)=0\}$ is closed and $G_\C$-invariant.
\begin{definition}\label{stable} The superpotential $W$ is called stable if $\Crit(L)$ consists of finitely many closed free $G_\C$-orbit and $L$ is Morse-Bott. In particular, $\ker\Hess_x L=T_x(G_\C\cdot x)$ for all $x\in \Crit(L)$.
\end{definition}

In fact, any $\vdelta\in \g$ is $W$-stable if $W$ is stable. In Example \ref{example1}, any $\delta\in i(0,\infty)$ is $W$-stable, but $W$ itself is not a stable superpotential. Indeed, $\Crit(L)=M$ consists of two $G_\C$-orbits, one of which is not free. In Example \ref{example2}, $L$ has a unique critical point $q=(0,0)\in \C^2$ with trivial $G_\C$-action. If we set instead $G=\{e\}$, then $W$ is stable and $\vdelta=0$ is $W$-stable. Now we provide a more interesting example. 
\begin{example}[The Fundamental Toy Model]\label{fundamental_example}
	Let $G=S^1, G_\C=\C^*, M=\C^3=\{(x,y,b): x,y,b\in \C\}$ and $W_\lambda(x,y,b)=(xy-\lambda)b$, where $\lambda\in \C$ is a fixed parameter. The $G_\C$-action $\rho_\C$ is defined as \[
	\rho_\C(u)(x,y,b)=(ux,u^{-1}y,b),\ \forall u\in \C^*.
	\]
The moment map is $\mu(x,y,b)=\frac{i}{2}(|x|^2-|y|^2)$ and $\nabla L=(\bar{y}\bar{b},\bar{x}\bar{b},\bar{x}\bar{y}-\bar{\lambda})$. If  $\lambda\neq 0$, then $\Crit(L)=\{b=0, xy=\lambda\}$ consists of a single $G_\C$-orbit, and the superpotential $W$ is stable. If $\lambda=0$, then $\Crit(L)=A_{xy}\cup A_{xb}\cup A_{yb}$ where
	\[
	A_{xy}\colonequals \{x=0, y=0\}, etc. 
	\]
	So $W$ is not stable. $\vdelta\in i\R$ is $W$-stable if and only if $\vdelta\neq 0$. For instance, take $\vdelta\in i\cdot (0,\infty)$. If $(x,y,b)\in \mu^{-1}(\vdelta)$, then $x\neq 0$ and so 
	\[
	\Crit(L)\cap \mu^{-1}(\vdelta) =\{(x,0,0): \frac{i}{2}|x|^2=\vdelta\}\subset A_{yb},
	\]
which consists of a single free $G$-orbit. 
	Moreover, we compute $\Hess L$ at $(x,y,b)\in M$:
	\[
	\Hess L(v)=\begin{pmatrix}
0 & \bar{b} & \bar{y}\\
\bar{b} & 0 & \bar{x}\\
\bar{y} & \bar{x} & 0
	\end{pmatrix}\begin{pmatrix}
\bar{x}'\\\bar{y}'\\\bar{b}'
	\end{pmatrix} \text{ with }v=\begin{pmatrix}
	x'\\y'\\b'
	\end{pmatrix}\in T_{(x,y,b)}M,
	\]
	so $L$ is Morse-Bott away from the origin. Note that the $G_\C$-orbit of $(x,0,0)$ is not closed. Its closure contains the origin. 
\end{example}
\section{The Gauged Witten equations}\label{1.3}

In this section, we introduce gauged Witten equations and explain its relation with downward gradient flow of the gauged action functional $\CA_H$. This serves as a toy model for the Floer theory to be developed in the second paper \cite{Wang20} of this series. Some results are stated and proved only for inspiration; they are not closely related to the proof of Theorem \ref{T1.2} and \ref{T1.3} in the end. 

\subsection{The Gauged Action Functional}\label{Sec4.1} Let $(M, W, G, \rho)$ be a gauged Landau-Ginzburg model and take $\vdelta\in \g$ to be $W$-stable (in the sense of Definition \ref{W-stable}). Take $\CL_U\subset \mu^{-1}(\vdelta)$ to be a $G$-invariant Lagrangian submanifold of $M$ such that $L=\re W$ is bounded above on $\CL_U$. Since $\vdelta$ is $W$-stable, $\mu^{-1}(\vdelta)\cap \Crit(L)$ consists of finitely many free $G$-orbits; let them be $O_1,\cdots, O_k$. Choose a reference point $q\in O_j$ for some $1\leq j\leq k$.  
\begin{assumpt}\label{A4.1} We first summarize the assumptions we need in order to formally set up a Floer theory:
\begin{itemize}
	\item the K\"{a}hler form $\omega_M$ is exact, i.e. $\omega_M=d\theta_M$ for some $\theta_M\in \Omega^1(M)$;
	\item the Lagrangian submanifold $\CL_U$ is exact, i.e. the primitive $1$-form $\theta_M|_{\CL_U}=dh$ for some smooth function $h: \CL_U\to \R$;
		\item the function $|\nabla H|^2+|\mu|^2_\g: M\to \R$ is proper;
	\item the superpotential $W$ is stable and $\vdelta\in \g$ is $W$-stable.
	\qedhere
\end{itemize}
\end{assumpt}

Let $Y=[0,\infty)_s$ and $X=\R_t\times Y=\HH^2_+$. Consider a smooth map $P: X\to M$ and a connection $A=d+a$ of the trivial principal $G$-bundle $Q$ over $X$:
\[
Q=X\times G.
\]
 Write the connection 1-form $a$ as $a_tdt+a_sds$ with $a_t,a_s\in \Gamma(\HH^2_+, \g)$. The smooth map $P$ can be differentiated covariantly  with respect to $A$:
\[
\nabla^A_{V}P\colonequals V\cdot P+\tilde{a}(V)
\]
for any tangent vector $V\in TX$. Here $\tilde{a}(V)$ is the induced tangent vector of $a(V)\in \g$.

We are interested in the gauged Witten equations on $X=\HH^2_+$ with boundary values in $\CL_U$:
\begin{equation}\label{generalized-vortex-equation}
\left\{ \begin{array}{rl}
-*_2F_A+\mu&=\vdelta,\\
\nabla^A_{\partial_t}P +J\nabla^A_{\partial_s}P+\nabla H&=0,\\
P(t,0)&\in \CL_U. 
\end{array}
\right.
\end{equation}

The first equation is a moment map constraint. The second one is the Cauchy-Riemann equation perturbed by the Hamiltonian $H=\im W$.  When $H\equiv 0$, the first two equations of \eqref{generalized-vortex-equation} reduce to the symplectic vortex equation \cite{CGMS02,Mun03} discovered by Cieliebak-Gaio-Salamon \cite{CGS00} and independently by Ignasi Munde i Rierra \cite{Mun00}. The gauged Witten equations (\ref{generalized-vortex-equation}) can be viewed as a formal downward gradient flow equation in an infinite-dimensional space, as we explain now. 

For either $Z=Y$ or $X$, let $\CA(Z)=d+\Gamma_0(Z, T^*Z\otimes \g)$ be the space of smooth connections with decay in the spatial direction with
\begin{equation}\label{E4.2}
\Gamma_0(Z, T^*Z\otimes \g)=\{a\in \SC^\infty(Z, T^*Z\otimes \g): \lim_{s\to\infty} a=0 \text{ and } \langle a,ds\rangle=0 \text{ at } s=0\}.
\end{equation}

A smooth map $p: Z\to M$ can be viewed as a section of the trivial $M$ bundle over $Z$:
\[
\widetilde{M}=Z\times M=(Z\times G)\times_G M. 
\]
 Consider the space of smooth sections of $\widetilde{M}\to Z$ subject to the Lagrangian boundary condition at $s=0$ and a asymptotic condition as $s\to\infty$:
 \begin{equation}\label{E5.7}
\Gamma_0(Z, \widetilde{M};\CL_U)=\{p: Z\to \widetilde{M}: \lim_{s\to\infty} p(s)=q \text{ and }p(0)\in \CL_U \}.
 \end{equation}
A gauge transformation must converge to the identity element $e$ of $G$ as $s\to\infty$:
\begin{equation}\label{E5.8}
\CG(Z)\colonequals \Map_0(Z, G)=\{u:Z\to G: \lim_{s\to\infty} u=e\in G\}. 
\end{equation}

The configuration space is $\SC(Z)=\CA(Z)\times \Gamma_0(Z, \widetilde{M};\CL_U)$ with $\CG(Z)$ acting on by the formula:
\[
u(A,p)=(A-u^{-1}du,u\cdot p). 
\]
\begin{remark} When $Z=X$, the 1-form $a$ in \eqref{E4.2} is required to decay smoothly to zero as $s\to\infty$ and uniformly for all $t\in \R_t$. More concretely, this means 
	\[
	\|\pt^k\ps^l a(\cdot, n+\cdot )\|_{L^\infty(\R_t\times [0,1]_s)}\to 0
	\]
	as $n\to\infty $ for all $k,l$. The same holds for the convergence in \eqref{E5.7} and \eqref{E5.8}. 
\end{remark}
\begin{definition}The gauged action functional $\CA_H$ is defined on $\SC(Y)$ with $Y=[0,+\infty)_s$ as:
\begin{equation}\label{areafunctional}
\CA_H(d+a, p)=-h(p(0))-\int_Y p^*\theta_M+\int_Y H\circ p(s)ds+\langle a, \vdelta-\mu\circ p\rangle_\g,
\end{equation}
where $\langle a,\vdelta-\mu\circ p\rangle_\g=\langle a_s, \vdelta-\mu\circ p\rangle_\g ds$ is understood as an 1-form on $Y$ and $\theta_M=dh$ on $\CL_U$. 
\end{definition}

For any $\gamma=(A,p)\in \SC(Y)$, a tangent vector $(\delta a,\delta p)$ in $T_\gamma\SC(Y)$ consists of a smooth form $\delta a\in \Gamma_0(Y,T^*Y\otimes \g)$ and a vector field $\delta p$ along the image $p(Y)$:
\[
\delta p \in \Gamma_0(Y, p^*TM; \CL_U).
\]

The tangent space $T_\gamma\SC(Y)$ inherits a $\CG$-invariant $L^2$-inner product from the Riemannian metric $g$ of $M$. We compute the $L^2$-formal gradient of $\CA_H$:
\begin{lemma}\label{P3.2}$\grad \CA_H (d+a, p)=(\vdelta-\mu\circ p,J\nabla^A_{\partial s}p+\nabla H)$.
\end{lemma}
\begin{proof} Let $P:[0,1]_t\times Y\to M$ be a smooth map such that $P(0, s)=p(s)$, $\partial_t P(0,s)=\delta p(s)$ and $\lim_{s\to\infty} P(\cdot, s)=q$. Then $\gamma_t=(d+a+t\delta a, P(t,\cdot))$ is a smooth variation of $\gamma_0=\gamma$. Note that 
	\begin{equation}\label{IBP}
	\int_{[0,t]\times Y} P^*\omega_M=\int_{[0,t]\times Y} dP^*\theta_M=h\big(P(t,0)\big)-h\big(P(0,0)\big)+\int_{\{t\}\times Y} P^*\theta_M-\int_{\{0\}\times Y} P^*\theta_M,
	\end{equation}
which is computed also as 
\begin{equation}\label{IBP2}
\int_{[0,t]\times Y} P^*\omega_M=\int_{[0,t]\times Y} \omega_M(\pt P, \ps P)dt'ds=-\int_{[0,t]}dt'\int_Y g(\pt P,J \ps P)ds.
\end{equation}
Now compute the variation of $(\ref{areafunctional})$ along a tangent vector $(\delta p,\delta a)$ using \eqref{IBP} and \eqref{IBP2}:
\begin{align*}
\frac{d}{dt} \CA_H(\gamma_t)|_{t=0}&=\int_Y g(\delta p, J\partial_s P+\nabla H)ds+\langle \delta a, \vdelta-\mu\circ p\rangle_\g-\langle \nabla\mu, \delta p\otimes a\rangle_\g, \\
&=\int_Y g(\delta p, J\Ds P+\nabla H)ds+\langle \delta a,\vdelta-\mu\circ p\rangle_\g,
\end{align*}
where we used the relation $\Ds P=\ps P+J\langle \nabla \mu, a_s\rangle_\g$. 
\end{proof}

\begin{lemma}\label{gaugeinvariant}
$\CA_H$ is $\CG(Y)$-invariant. 
\end{lemma}
\begin{proof} Since elements of $\CG(Y)$ are subject to the asymptotic condition $\lim_{s\to\infty} u=e$, $\CG(Y)$ is contractible. It suffices to consider the infinitesimal action. The Lie algebra of $\CG$ is 
	\[
	\Lie (\CG)=\Gamma_0(Y, \g)=\{\xi: Y\to \g: \lim_{s\to\infty} \xi(s)=0\}.
	\]

For $\xi\in \Lie(\CG)$, the tangent vector generated at $\gamma\in \SC(Y)$ is 
\begin{equation}\label{E3.4}
\bd_\gamma(\xi)\colonequals (-\partial_s \xi,\tilde{\xi})=(-\partial_s \xi,J\langle \nabla \mu, \xi\rangle_\g ).
\end{equation}

It suffices to verify this vector is $L^2$-orthogonal to $\grad \CA_H$. For any path $p\in \Gamma_0(Y,\widetilde{M};\CL_U)$, $\vdelta-\mu\circ p(s)=0$ for $s=0$ and $\infty$. Hence, the boundary terms do not contribute to the integration by parts below:
\begin{align*}
\int_Y\langle \vdelta-\mu\circ p,-\partial_s \xi\rangle_\g &=-\int_Y\langle \partial_s(\mu\circ p), \xi\rangle_\g=-\int_Y\langle \nabla \mu, \partial_s p\otimes \xi\rangle_{TM\otimes\g}.
\end{align*}

On the other hand, we use Lemma \ref{identities} (\ref{5})(\ref{6}) from Appendix \ref{B} to compute:
\[
\int_Y\big\langle J\nabla^A_{\partial s}p+\nabla H, J\langle \nabla \mu, \xi\rangle_\g\big\rangle =\int_Y\langle \nabla \mu, \partial_s p\otimes \xi\rangle_{TM\otimes\g}.\qedhere
\]
\end{proof}
\begin{remark} In the expression (\ref{areafunctional}), the first two terms come from the usual action functional, motivated by the integration by parts (\ref{IBP}). The third part comes from the Hamiltonian perturbation. The last one is added by requiring $\CA_H$ to be gauge-invariant. 
\end{remark}

Hence, the gauged Witten equations (\ref{generalized-vortex-equation}) is cast into the form
\[
\pt \gamma_t+\grad \CA_H(\gamma_t)=0
\]
 with $\gamma_t=(d+a_s(t,\cdot)ds, P(t,\cdot))\in \SC(Y)$ when $a_t\equiv 0$. There is a classical notion of analytic energy associated to any downward gradient flow equation: 
\begin{align}\label{E3.6}
\infty>	\E_{an}(\{\gamma_t\})&=-\lim_{t\to\infty}\CA_H(\gamma_t)+\lim_{t\to-\infty}\CA_H(\gamma_t)=\int_{\R_t} \langle -\pt \gamma_t, \grad \CA_H(\gamma_t)\rangle\\
	&=\half \int_{\R_t} |\pt\gamma_t|^2+|\grad \CA_H(\gamma_t)|^2\geq 0.\nonumber
\end{align}

This formula is only valid when $A$ is in the temporal gauge, i.e. when $a_t\equiv 0$. On the contrary, the gauged Witten equations (\ref{generalized-vortex-equation}) are invariant under a larger gauge group $\CG(X)$ (recall that $X=\R_t\times Y$). In fact, the left hand side of (\ref{generalized-vortex-equation}) defines a $\CG(X)$-equivariant map:
\[
\F: \SC(X)\to \Gamma_0(X, Q\times_G (TM\oplus \g)),
\]
called the gauged Witten map. 

\begin{definition}\label{D3.5} Let $X=\HH^2_+=\R_t\times [0,\infty)_s$. For any $(A, P)\in \SC(X)$, define
	\begin{align*}
T&\colonequals \nabla^A_{\partial_t}P\in \Gamma(X, P^*TM),\\
S&\colonequals \nabla^A_{\partial_s}P\in \Gamma(X, P^*TM),\\
F&\colonequals -*_2F_A\circ P\in \Gamma(X,\g).
	\end{align*}
	Then define the analytic energy of $(A,P)$ as 
	\begin{equation}\label{analytic-energy}
	\E_{an}(A,P)=\int_X |T|^2+|JS+\nabla H|^2+|F|^2+|\vdelta-\mu|^2,
	\end{equation}
	which is invariant under $\CG(X)$ and reduces to $\E_{an}(\{\gamma_t\})$ in \eqref{E3.6} in the temporal gauge.
\end{definition}

In terms of the shorthands introduced in Definition \ref{D3.5}, the equation (\ref{generalized-vortex-equation}) takes a more compact form:
\begin{subequations}
	\begin{align}
		F+\mu&=\vdelta, \label{secondeq}\\
	T+JS+\nabla H&=0,\label{firsteq}\\
	P(t,0)&\in \CL_U.\label{thirdeq}
	\end{align}
\end{subequations}

We are interested in the moduli space of solutions of  (\ref{generalized-vortex-equation}) with finite analytic energy. It is possible to impose a gauge-fixing condition, produce an elliptic theory and finally construct a Morse complex in this context. However, the full construction is carried out in \cite{Wang20} only for an infinite-dimensional case in which the gauged Witten equations \eqref{generalized-vortex-equation} are identified with the Seiberg-Witten equations. 

\subsection{The Extended Hessian}\label{Subsec4.2} Although we will only get into linear analysis in the second paper \cite{Wang20}, it is enlightening to first work out the extended Hessian of the gauged action functional $\CA_H$. The discussion below will be used in \cite[Section 11]{Wang20} when the essential spectrum of the extended Hessian is computed for the perturbed Chern-Simons-Dirac functional on a 3-manifold with cylindrical ends. 

At any $\gamma=(A,p)\in \SC(Y)$, the linearized gauge action 
\begin{align*}
\bd_\gamma: \Lie (\CG)=\Gamma_0(Y, \g)&\to T_\gamma \SC(Y)
\end{align*}
defined by the formula \eqref{E3.4} has a formal $L^2$-adjoint:
\begin{align*}
\bd_\gamma^*:  T_\gamma \SC(Y)&\to \Gamma_0(Y, \g),\\
\big((\delta a_s) ds,\delta p\big)&\mapsto \ps (\delta a_s)+\langle J\nabla \mu,\delta p\rangle. 
\end{align*}

By linearizing the expression in Proposition \ref{P3.2}, we obtain the Hessian of $\CA_H$ at $\gamma$:
\begin{align*}
\D_\gamma \CA_H: T_\gamma\SC(Y)&\to T_\gamma\SC(Y),\\
\big((\delta a_s) ds,\delta p\big)&\mapsto  \big(-\langle \nabla\mu, \delta p\rangle, J(\ps \delta p)-\langle \nabla \mu, a_s\rangle_\g+\Hess H(\delta p)\big).
\end{align*}

The key observation is that these operators can be combined into a larger operator, the extended Hessian of $\CA_H$, which is formally $L^2$-self-adjoint:
\[
\EHess_\gamma\colonequals \begin{pmatrix}
0& \dg^*\\
\dg & \D_\gamma \CA_H
\end{pmatrix}: \SH^2_1\to L^2\big(Y, \g\oplus (T^*Y\otimes \g)\oplus p^*TM\big),
\]
where $\SH^2_1$ is a subspace of $L^2_1$-sections:
\[
\SH^2_1\colonequals\{(f, (\delta a_s)ds, \delta p)\in L^2_1\big(Y, \g\oplus (T^*Y\otimes \g)\oplus p^*TM\big): \delta a_s(0)=0, \delta p(0)\in p^*T\CL_U\}. 
\]

The operator $\EHess_\gamma$ is cast into the form $
 \sigma(\ps+\widehat{D}_{p(s)}) $ with 
\begin{equation}\label{E4.8}
\sigma=\begin{pmatrix}
0 & 1 & 0\\
-1 & 0 & 0\\
0 & 0 & J
\end{pmatrix} \text{ and }
\widehat{D}_{p(s)}=
\begin{pmatrix}
0 & 0 & \langle \nabla\mu, \cdot \rangle_{T_{p(s)} M}\\
0 & 0 &  \langle J\nabla\mu, \cdot \rangle_{T_{p(s)} M}\\
\langle \nabla\mu, \cdot \rangle_\g &  \langle J\nabla\mu, \cdot \rangle_\g & \Hess L
\end{pmatrix},
\end{equation}
where we identify $T^*Y\otimes \g$ with the trivial $\g$-bundle over $Y$ (still denoted by $\g$) by omitting the 1-form $ds$. The operator $\widehat{D}$ is a self-adjoint endomorphism on the bundle
\[
\widehat{TM}\colonequals\g\oplus \g\oplus TM\to M.
\]
Moreover, $\sigma$ defines an almost complex structure on $\widehat{TM}$, which anti-commutes with $\widehat{D}$, i.e.
\[
\sigma^2=-\Id,\ \sigma \widehat{D}+\widehat{D}\sigma=0.
\]

The operator $\widehat{D}$ is tied to the stability of $W$ by the following observation:
\begin{lemma}\label{L4.6} The superpotential $W$ is stable in the sense of Definition \ref{stable} if and only if $\widehat{D}_q$ is invertible at all $q\in \Crit(L)$ and $\Crit(L)$ consists of finitely many closed free $G_\C$-orbits. 
\end{lemma}

The linear analysis in \cite[Section 11]{Wang20} relies essentially on this special structure of the extended Hessian $\EHess_\gamma$. As a preview, the essential spectrum of $\EHess_\gamma$ is
\[
(-\infty,-\lambda_1]\cup [\lambda_1,+\infty)
\]
where $\lambda_1$ is the first non-negative eigenvalue of $\widehat{D}_q$. In particular, $\EHess_\gamma$ is Fredholm if and only if $\lambda_1>0$.

We end this section with a remark on the domain $\SH^2_1$ of $\EHess_\gamma$. A section $(f, (\delta a_s)ds, \delta p)\in \SH^2_1$ is subject to the boundary condition:
\[
\big(f(0), \delta a_s(0), \delta p(0)\big)\in \g\oplus \{0\}\oplus p^*T\CL_U \text{ at } s=0,
\]
which is a Lagrangian subspace with respect to $\sigma$; otherwise, $\EHess_\gamma$ is not formally $L^2$-self-adjoint. This is why we imposed the boundary condition 
\[
\langle a,ds\rangle=0 \text{ at } s=0,
\]
in the definition \eqref{E4.2} of $\CA(Z)$. Otherwise, $\bd_\gamma^*$ is not the $L^2$-formal adjoint of $\bd_\gamma$.

\section{Point-Like Solutions}\label{1.4}

In this section, we study finite energy solutions of (\ref{generalized-vortex-equation}) on the complex plane $\C$, the so-called point-like solutions in terms of \cite[Section 14.1]{GMW15}. Assuming $W$ is a stable superpotential, we will prove that all point-like solutions are trivial, i.e. they are gauge equivalent to the constant solutions. Interesting solutions may occur if $W$ is not stable, cf. Example \ref{vortex}. 

Let $P: \C\to M$ be a smooth map and $A$ be a smooth connection in the trivial principal $G$-bundle $\C\times G\to \C$. Write $z=t+is$ for the coordinate function on $\C$. In the sequel, we shall frequently use the shorthands from Definition \ref{D3.5}. The main result of this section is the following:
\begin{theorem}\label{point-like-solutions} Suppose $(M,W,G,\rho)$ is a gauged Landau-Ginzburg model and $W$ is stable in the sense of Definition \ref{stable}. Take any $\vdelta\in \im \mu\subset \g$. Suppose $(A,P)$ is a solution of the gauged Witten equations
		\begin{equation}\label{equation on C}
	\left\{ \begin{array}{r}
		-*F_A+\mu=\vdelta,\\
	\nabla^A_{\partial_t}P +J\nabla^A_{\partial_s}P+\nabla H=0,
	\end{array}
	\right.
	\end{equation}	
on $\C$ with $\E_{an}(A, P)<\infty$, and  $(A,P)$ is subject to the asymptotic condition 
\begin{equation}\label{boundarycondition2}
\lim_{z\to\infty} P(z)\to q, 
\end{equation}
for some base point $q\in \mu^{-1}(\vdelta)\cap \Crit(L)$,
then $(A,P)$ is gauge equivalent to the constant solution $(A_0=d, P\equiv q)$.
\end{theorem}

The proof of Theorem \ref{point-like-solutions} is based on an interesting observation. Since $W$ is holomorphic and $P$ is $J$-holomorphic up to Hamiltonian perturbations, it is reasonable to ask if the composition:
\[
W\circ P: \C\xrightarrow{P} M\xrightarrow{W}\C,
\]
is still holomorphic. In fact, we have 
\begin{lemma}\label{dbarG} If $(A,P)$ is a solution to the gauged Witten equations $ (\ref{equation on C}) $ on $\C$, then 
	\[
	\bpartial (W\circ P)\colonequals (\partial_t+i\partial_s)(W\circ P)=-i|\nabla H|^2.
	\]
\end{lemma}
\begin{proof} By the Cauchy-Riemann equation $\nabla L=-J\nabla H$ and $(\ref{equation on C})$, we have 
	\begin{align*}
	\bpartial (W\circ P)&=\langle \nabla L+i\nabla H, \Dt P+J\Ds P\rangle =-i|\nabla H|^2.\qedhere
	\end{align*}
\end{proof}
\begin{remark}When $A=d$ is the trivial connection and $\Dt P\equiv 0$, $P(t,\cdot)$ is a downward gradient flowline of $L$. In this case, this lemma recovers the usual identity:\[
	\partial_s (L\circ P)=-|\nabla L|^2.
	\]
$P(t,\cdot)$ is also a Hamiltonian flow, so $ 	\partial_s (H\circ P)=0.$
\end{remark}

We also need a more useful notion of energy:
\begin{lemma}\label{L5.4} Under the conditions of Theorem \ref{point-like-solutions}, define
	\begin{equation}\label{analytic energy C}
\E_{an}(A,P; \C)\colonequals \int_{\C} |\nabla_A P|^2+|\nabla H|^2+|F_A|^2+|\vdelta-\mu|^2.
\end{equation}
Then $\E_{an}(A,P; \C)=\E_{an}(A,P)<\infty$. 
\end{lemma}
\begin{proof} Using the shorthands from Definition \ref{D3.5} and the Cauchy-Riemann equation (\ref{CRequation}), we deduce that:
	\begin{align*}
\int_\C |JS+\nabla H|^2&=\int_\C |S+\nabla L|^2= \int_\C |S|^2+|\nabla L|^2+\lim_{t'\to\infty}\lim_{s'\to\infty}\int_{[-t',t']\times [-s',s']}2\langle S,\nabla L\rangle\\
&=\int_\C |S|^2+|\nabla L|^2+\lim_{t'\to\infty}\lim_{s'\to\infty} \int_{[-t',t']} 2(L\circ P(t, s')-L\circ P(t, -s')).
	\end{align*}
	
	 By (\ref{boundarycondition2}), the boundary term tends to zero as $s'\to\infty$, so $\E_{an}(A,P;\C)=\E_{an}(A,P)$. 
\end{proof}

\begin{lemma}\label{vanishing-lemma} Under the assumption of Theorem \ref{point-like-solutions}, $\nabla L\equiv 0$, so $P(z)\in \Crit(L)$ for all $z=t+is\in \C$ and $W\circ P$ is a constant function on $\C$. 
\end{lemma}
\begin{proof}  Since $W$ is Morse-Bott, for some $G$-invariant neighborhood $\Omega$ of $G\cdot q\subset M$ and $C>0$, the estimate 
	\begin{equation}\label{morse-bott}
|W(x)-W(q)|\leq C|\nabla H(x)|^2,
\end{equation}
holds for any $x\in \Omega$. By the asymptotic condition $(\ref{boundarycondition2})$, for a large constant $R(\Omega)>0$, $P(z)\in \Omega$ for all $|z|>R$. As a result,  
\begin{equation}\label{tameness2}
|W\circ P(z)-W(q)|\leq C\big|\nabla H\big(P(z)\big)\big|^2, 
\end{equation}
when $|z|>R$. Write $(W\circ P)(z)-W(q)=U+iV$ with $U,V$ real. Then Lemma \ref{dbarG} implies that
\[
\partial_t U-\partial_sV= 0, \partial_t V+\partial_s U=-|\nabla H|^2\leq 0. 
\]
Let $K(z)\colonequals \int_0^z Vdt+Uds$. By the first equation above, this integral is independent of the path we choose. Therefore,
\[
U=\partial_sK, V=\partial_t K \text{ and } 	\Delta_\C K=(-\partial_s^2-\partial_t^2)K=|\nabla H|^2\geq 0. 
\]

Then the Morse-Bott inequality (\ref{tameness2}) is equivalent to $|\nabla K|=|W\circ P-W(q)|\leq C |\Delta K|.$

Our goal is to show $K\equiv 0$. Let $Z(r)\colonequals \int_{\partial B(0,r)} \Delta K\geq 0$. Take $r>R(\Omega)$ and integrate by parts:
\begin{align*}
0\leq E(r)&\colonequals\int_0^r Z(r')dr'=\int_{B(0,r)} \Delta K= \bigg|\int_{\partial B(0,r)} \vec{n}\cdot \nabla K\bigg|\\
&\leq C \bigg(\int_{\partial B(0,r)}\Delta K\bigg)\leq CE(r)' .
\end{align*}

Therefore,  for any $r>r_0>R(\Omega)$, 
\begin{equation}\label{4.6}
0\leq E(r_0)\leq E(r)e^{\frac{r_0-r}{C}}.
\end{equation}

Let $r\to\infty$. Note that $\lim_{r\to\infty} E(r)=\int_\C |\nabla H|^2\leq \E_{an}(A,P;\C)<\infty$ by Lemma \ref{L5.4}. Hence, $E(r_0)\equiv 0$, and
\[
\Delta K=|\nabla H|^2\equiv 0\Rightarrow W\circ P(z)\equiv W(q).\qedhere
\]

\end{proof}

\begin{remark}\label{R5.6} The proof of Lemma \ref{vanishing-lemma} does not require $W$ to be stable. It suffices to assume that $W$ is Morse-Bott near the critical $G$-orbit $G\cdot q$. 
\end{remark}

\begin{proof}[Proof of Theorem \ref{point-like-solutions}]  Since $W$ is stable, the multiplication $g\mapsto g\cdot q$ defines a closed embedding $\iota$ of $G_\C$ into $M$. By Lemma \ref{vanishing-lemma}, $\im P\subset \im \iota$, so $P(z)=g(z)\cdot q$ for a unique element $g(z)\in G_\C$. 
	
	We first deal with the case when $G=S^1$ and $G_\C=\C^*$. Since we are interested in solutions modulo gauge equivalences, $g(z)$ may be assumed to be real. Suppose $g(z)=e^{\alpha(z)}$ for some $\alpha:\C\to \R$. The boundary condition (\ref{boundarycondition2}) then says
	\begin{equation}\label{E4.7}
	\lim_{z\to\infty} \alpha(z)=0.
	\end{equation}

	Moreover, the second equation of (\ref{equation on C}) implies $A=d+i*_2d\alpha.$ Plug this into the first equation of (\ref{equation on C}) to obtain that
	\begin{equation}\label{maximum}
	i \Delta_\C\alpha+\big(\mu(e^{\alpha(z)}\cdot q)-\mu(q)\big)=0.
	\end{equation}
	
By \eqref{E4.7}, $\sup_{z\in \C}|\alpha(z)|$ is attained at some $z_0\in \C$ and let $\beta\colonequals \alpha(z_0)\neq 0$. Then 
	\[
	\langle \mu(e^{\beta}\cdot q)-\mu(q), i\beta\rangle_\g=-\langle \Delta_{\C} \alpha (z_0),\beta\rangle\leq 0.
	\]
	
	We claim that the inner product $\langle \mu(e^{\beta}\cdot q)-\mu(q),i\beta\rangle_\g$ is always non-negative. Indeed,
	\begin{align*}
\langle \mu(e^{\beta}\cdot q)-\mu(q),i\beta\rangle_\g&=\int_0^1 \langle \pt \mu(e^{t\beta}\cdot q), i\beta\rangle_\g dt=\int^1_0 \big\langle \langle\nabla\mu(e^{t\beta}\cdot q),\tilde{\beta}\rangle_{TM}, i\beta\big\rangle_\g.
	\end{align*}
Here $\tilde{\beta}_x\in T_xM$ denotes the tangent vector $\dt (e^{\beta}\cdot x)\big|_{t=0}$ at $x\in M$. Since the $G_\C$-action is holomorphic (by Definition \ref{LGmodels}), we set $\xi=i\beta$ in \eqref{convention} to obtain $\tilde{\beta}_x=\langle \nabla\mu(x), i\beta\rangle_\g$. Hence,
\[
	\langle \mu(e^{\beta}\cdot q)-\mu(q),i\beta\rangle_\g=\int^1_0 |\langle \nabla \mu(e^{t\beta}\cdot q),i\beta\rangle_\g|^2dt\geq 0. 
\]	
	Since the base point $q$ generates a free $G_\C$-orbit, when $\beta\neq 0$, this integrand is positive. So $\beta= 0$. When $G$ is a general abelian Lie group, the proof can proceed in a similar fashion. 
\end{proof}

We end this section with a few examples.

\begin{example} In our Fundamental Example \ref{fundamental_example}, suppose $\lambda=r_+r_-,\ q=(r_+,r_-,0)$ and $\vdelta=\frac{i}{2}(r_+^2-r_-^2)$ for some $r_\pm\in \R$. Then the equation (\ref{maximum}) becomes 
	\[
	\Delta_\C\alpha+\half \big(r_+^2(e^{2\alpha}-1)+r_-^2(1-e^{-2\alpha})\big)=0.\qedhere
	\]
\end{example}

\begin{example}\label{vortex} For Example \ref{example1}, the gauged Witten equations come down to the vortex equation on $\C$ (with $\delta=\frac{i}{2}$):
	\begin{equation}
\left\{ \begin{array}{r}
\bpartial_A P=0,\\
i*F_A+\frac{1}{2}(|P|^2-1)=0.
\end{array}
\right.
\end{equation}	

By \cite{Taubes}, the moduli space $\M_n$ with $\E_{an}=2\pi n$ is $\Sym^n\C$ for any $n\geq 1$, so Theorem \ref{point-like-solutions} fails. $W$ is not stable in this case, even though $\delta=\frac{i}{2}$ is $W$-stable. Note that $\M_n$ is regular; its dimension agrees with the prediction of the index theory.
\end{example}

\begin{example}\label{EX5.9} In Example \ref{fundamental_example}, let $\lambda=0$ and $\vdelta=\frac{i}{2}$. For a solution $(A,P)$ of (\ref{equation on C}), write $P(z)=(x(z), y(z), b(z))$. Then Lemma \ref{vanishing-lemma} and Remark \ref{R5.6} imply $y(z)\equiv b(z)\equiv 0$. The equations are reduced to the previous example. However, in this case, the moduli space $\M_n'$ is not regular. Its formal dimension is always zero for any $n\geq 0$. 
\end{example}

\section{Exponential Decay in the Spatial Direction}\label{1.5}

In this section, we generalize Lemma \ref{Lemma-exponentialdecay} in the context of gauged Landau-Ginzburg models, as the analogue of Theorem \ref{T1.3} in the finite-dimensional case. We state and prove the theorem for the energy density function. 

\begin{theorem}\label{exponential-decay} For any stable gauged Landau-Ginzburg model $(M, W,G,\rho)$, there exist $\epsilon(M,W),\zeta(M,W)>0$ with following significance. Given a solution $\gamma=(A,P)\in \SC(X)$ to the gauged Witten equations $(\ref{generalized-vortex-equation})$ on the upper half plane $X=\R_t\times [0,\infty)_s$, suppose the point-wise estimate
	\begin{equation}\label{E6.1}
	U_\gamma(t,s)\colonequals |\nabla^AP|^2+|\nabla H|^2+|F_A|^2+|\vdelta-\mu|^2<\epsilon
	\end{equation} 
	holds for all $(t,s)\in X$. Then 
	\[
	U_\gamma(t,s)<e^{-\zeta s},\ \forall s\geq 0.
	\]
The function $U_\gamma:X\to [0,\infty)$ is called the energy density function. 
\end{theorem}

Recall from Section \ref{Sec4.1} that any configuration $(A,P)\in \SC(X)$ is subject to the asymptotic condition 
\[
\lim_{s\to\infty} P(\cdot ,s)\to q,
\]
for some base point $q\in\mu^{-1}(\vdelta)\cap \Crit(L)$. Under the assumption of Theorem \ref{exponential-decay}, the energy density $U_\gamma$ then provides an upper bound for the distance function:
\begin{equation}\label{distance}
U_\gamma(t,s)\geq |\nabla H|^2+|\vdelta-\mu|^2\geq \epsilon' \cdot \big(d(P(t,s), O_*)\big)^2,
\end{equation}
when $0<\epsilon\ll 1$. Indeed, the tangent space $T_x M$ can be decomposed orthogonally as $V_0\oplus JV_\g\oplus V_\g$  with $V_\g=T_x(G\cdot q)$ and $JV_{\g}\oplus V_\g=T_x(G_\C\cdot x)$ for all $x\in M$. At $x=q$, identify a normal neighborhood of $q$ with a ball $\Omega\subset T_q M$ centered at the origin and impose the gauge fixing condition $P(t,s)\in \Omega\cap (V_0\oplus JV_\g)_q\subset T_qM$ when $\epsilon\ll 1$. This can be done since the function $|\nabla H|^2+|\vdelta-\mu|^2:M\to \R$ is proper by Assumption \ref{A4.1}. Within this gauge slice, the map
\[
x\mapsto (\nabla H,\vdelta-\mu)
\]
can be treated as a section of $\Gamma(\Omega\cap (V_0\oplus JV_\g)_q, V_0\oplus \g)$ with a transverse zero at the origin by our stability condition. This implies the distance estimate \eqref{distance}. In particular, Theorem \ref{exponential-decay} implies Lemma \ref{Lemma-exponentialdecay} when $G=\{e\}$ is trivial. 

\begin{remark} The analogue of Lemma \ref{Lemma-uniformdecay} (the uniform $L^\infty$ decay) continues to hold for the gauged Witten equation if $H(q_j)\neq H(q_k)$ for any $q_j, q_k$ in different critical $G_\C$-orbits and if an additional local compactness property can be verified. This uniform decay can be further improved into a uniform exponential decay using Theorem \ref{exponential-decay}. The proof of such a uniform decay result will be carried out for the Seiberg-Witten equations in \cite[Theorem 6.3]{Wang20}.
\end{remark}

\begin{proof}[Proof of Theorem \ref{exponential-decay}]

By the gauged Witten equations (\ref{generalized-vortex-equation}), it suffices to show the exponential decay for the quantity
\begin{equation}\label{E6.2}
u(t,s)\colonequals |\nabla^A P|^2+|F_A|^2_\g.
\end{equation}

We use a lemma from Appendix \ref{max} and verify its conditions:
\begin{lemma}[Corollary \ref{coro-max}]\label{maximum2}
Take $\Lm>0$. Suppose $w: \HH^2_+=\R_t\times [0,\infty)_s\to \R$ is  a bounded $\SC^2$-function on the upper half plane $\HH^2_+$ such that 
\begin{enumerate}[label=(U\arabic*)]
	\item\label{U1} $(\Delta_{\HH^2_+}+\Lm^2)w\leq 0$, and
	\item\label{U2} for some $K>0$, $w(t,0)\leq K$ for all $t\in \R_t$. 
\end{enumerate}
Then $w(t,s)\leq Ke^{-\Lm s}$ for all $(t,s)\in \HH^2_+$. 
\end{lemma}

The second condition \ref{U2} holds for the function $u$ in \eqref{E6.2} by our assumption on $U_\gamma$ in \eqref{E6.1}:
\[
u(t,s)\leq U_\gamma(t,s)<\epsilon.
\]

 To verify \ref{U1}, we find an explicit formula of $\Delta_{\HH^2_+}u(s,t)$. It is convenient to define a bundle map $D: TM\to \g\oplus \g\oplus TM$ such that 
\[
D_x(v)\colonequals\big( \langle \nabla \mu, v\rangle_{TM}, \langle J\nabla\mu, v\rangle_{TM},\Hess_x L(v)\big), \forall x\in M, v\in T_x M. 
\]

For the next lemma, we follow the shorthands from Definition \ref{D3.5}.

\begin{lemma}[Corollary \ref{B.8}]\label{Bochner} We have the following Bochner-type formula for $\Delta_{\HH^2_+}u(t,s)$:
	\[
	0=\half \Delta_{\HH^2_+}\big(|\nabla^A P|^2+|F|^2_\g\big)+I_1+I_2+I_3+I_4+I_5
	\]
	where
\begin{align*}
I_1&=|\Hess_A P|^2+|\nabla F|^2_\g,\qquad I_2=|D_P(\nabla^A P)|^2+|\langle \nabla\mu, F\rangle_\g |^2,\qquad I_3=2\langle R_M(S,T)S,T\rangle, \\
I_4&=\langle(\nabla_T\Hess H)(\nabla H), T\rangle+\langle(\nabla_S\Hess H)(\nabla H), S\rangle,\\
I_5&=6\langle \Hess \mu(JS),T\otimes F\rangle-\langle \Hess \mu(T),T\otimes F\rangle-\langle \Hess \mu(S),S\otimes F\rangle,
\end{align*}
and $R_M$ is the Riemannian curvature tensor of $M$. 
\end{lemma}

	\begin{remark} This identity was first proved by Taubes in \cite[Proposition 6.1]{JT} for the vortex equation on $\C$, in which case $M=\C$ is furnished with the flat metric,  $W\equiv 0$ and $\mu=\frac{i}{2}|x|^2$ for all $x\in \C$, cf. Example \ref{example1}. For more details, see Remark \ref{B.10}. 
	\end{remark}

Let us digest the consequence of Lemma \ref{Bochner}. $I_1\geq 0$. $I_4$ and $I_5$ involve only trilinear tensors: 
\begin{align*}
\langle\nabla_\cdot \Hess H(\cdot),\cdot\rangle&:   TM\otimes TM\otimes TM\to \R,\\
\langle \Hess\mu(\cdot), \cdot\otimes \cdot\rangle&:TM\otimes TM\otimes \g\to \R. 
\end{align*}

By \eqref{E6.1} and \eqref{distance}, the image of $P$ lies in a $G$-invariant neighborhood $\Omega$ of $G\cdot q$ with compact closure, where these trilinear maps have uniformly bounded operator norms. Hence, 
\[
|I_4|+|I_5|\leq Cu^{3/2},
\]
for some $C>0$. Indeed, each entry that appears in $I_4$ and $I_5$ takes the form $T,S,\nabla H, F$ or $JS$, with norm bounded above by $C'u^{1/2}$.  The same argument leads to an estimate for $I_3$ with a different exponent of $u$:
\[
|I_3|\leq Cu^2. 
\]
Since the critical orbit $G\cdot q$ is free,  for some small $\Lm_1>0$, we have
\begin{equation}\label{E6.4}
|\langle \nabla \mu, F\rangle_\g|^2> \Lm^2_1 |F|^2
\end{equation}
 for all $x\in G\cdot q$ and $F\in\g$. Moreover, since $W$ is a stable superpotential, $D_x$ is injective at all $x\in G\cdot q$ (by the Morse-Bott condition) and so for some $\zeta_2>0$,
\begin{equation}\label{E6.5}
|D_x(v)|^2>\Lm_2^2|v|^2
\end{equation}
for all $x\in G\cdot q$ and $v\in T_x M$. The estimates \eqref{E6.4} and \eqref{E6.5} continue to hold for all $x\in \Omega$, if the neighborhood $\Omega$ had been chosen sufficiently small. Hence, for $\zeta=\min\{ \Lm_1,\Lm_2\}$,
\begin{equation}
|I_2|\geq \Lm^2 u
\end{equation}
whenever  $P(s,t)\in \Omega$. By taking the constant $\epsilon$ in \eqref{E6.1} to satisfy  $\epsilon+\epsilon^{1/2}<\Lm^2/2C$, Lemma \ref{Bochner} implies
\[
0\geq \half \Delta_{\HH^2_+} u+\Lm^2 u-C(u^2+u^{3/2})\geq \half (\Delta_{\HH^2_+}+\Lm^2)u. 
\]

Now apply Lemma \ref{maximum2}  with $K=\epsilon$ and $w=u$. 
\end{proof}

\begin{remark} The bundle maps $D$ and $\langle \nabla\mu,\cdot\rangle_\g$ involved in $I_2$ appeared also in the extended operator $\widehat{D}$ in \eqref{E4.8}. The invertibility of $\widehat{D}$ is essential to this proof. 
\end{remark}

\begin{remark} To better understand the mysterious Bochner-type formula in Lemma \ref{Bochner}, consider the baby case in Example \ref{example0} for which the structure group $G=\{e\}$ is trivial. Then $\Hess H$ is a constant self-adjoint $\R$-linear operator on $\C^n$, so
	\[
	(\nabla H)_x=\Hess H(x),\ x\in \C^n. 
	\]
	Applying the operator $(\pt-J\ps)$ to \eqref{firsteq}, we obtain that 
\begin{align*}
0&=\pt T+\ps S+\Hess H(T+JS)=-\Delta P-(\Hess H)^2(P),
\end{align*}
from which one can easily deduce that the map $P: \HH^2_+\to \C^n$ along with its all higher derivatives has exponential decay as $s\to \infty$. Example \ref{example1} is the other extreme where $W\equiv 0$ and $\mu$ is quadratic. The proof of Lemma \ref{Bochner} is a tedious exercise in Riemannian geometry and is deferred to Appendix \ref{B}. 
\end{remark}

%% file: SWEQ_6.0.tex
\part{The Seiberg-Witten Equations on $\C\times\Sigma$}\label{Part2}

In the third part of this paper, we associate an infinite-dimensional Landau-Ginzburg model to an $H$-surface $\TSigma$ so that the associated gauged Witten equations on $\C$ recover the Seiberg-Witten equations on $\C\times \Sigma$. Recall that an $H$-surface $\TSigma=(\Sigma, g_\Sigma,\lambda, \nu)$ is a quadruple, where
\begin{itemize}
	\item $\Sigma$ is a connected closed oriented surface of genus $g(\Sigma)\geq 1$ and $g_\Sigma$ is a metric;
	\item $\lambda\neq 0\in \Omega^1_h(\Sigma, i\R)$ is a harmonic 1-form on $\Sigma$ such that $\lambda^{1,0}\in \Omega^1(\Sigma,\C)$ has precisely $(2g(\Sigma)-2)$ simple zeros; $\nu\in \Omega^2(\Sigma, i\R)$ is an imaginary-valued closed 2-form on $\Sigma$. 
\end{itemize}

We generalize Theorem \ref{point-like-solutions} and \ref{exponential-decay} from the previous part to this infinite-dimensional case. The main difference is that the topology of $M$ depends on a Sobolev completion of smooth sections, and we need to specify the correct norms used in the estimates. 

The main obstacle in defining a Floer homology for a 3-manifold $\hy$ with cylindrical ends is a compactness problem, and its resolution relies on three inputs:

\begin{enumerate}[label=(K\arabic*)]
\item\label{K1} a uniform upper bound on the analytic energy;
\item\label{K2} finite energy solutions are trivial on $\C\times \Sigma$, namely, they have to be $\C$-invariant up to gauge;
\item\label{K3}  finite energy solutions on $\R_s\times \Sigma$ are trivial, namely, they have to be $\R_s$-invariant up to gauge. 
\end{enumerate} 

One way to attain these properties is to perturb the Seiberg-Witten equations on either $\C\times\Sigma$ or $\R_s\times \Sigma$ using the closed 2-form $\omega\colonequals\nu+ds\wedge\lambda$.  While $\lambda$ is used to perturb the superpotential $W_\lambda$, $\nu$ is to perturb the moment map equation in \eqref{generalized-vortex-equation}.

The first property \ref{K1} will be addressed in the second paper \cite{Wang20} when we set up the cobordism category properly. The second property \ref{K2} is established by Theorem \ref{T1.2}.  In Section \ref{Sec8}, we address the last property $\ref{K3}$ using a theorem of Taubes.

\section{The Fundamental Gauged Landau-Ginzburg Model}\label{Sec6}

In this section, we construct the fundamental gauged Landau-Ginzburg model 
\[
\big(M(\Sigma,\s), W_\lambda, \CG(\Sigma)\big)
\]
associated to an $H$-surface $\TSigma=(\Sigma,g_\Sigma,\lambda,\nu)$. In Proposition \ref{stableW}, we verify that the superpotential $W_\lambda$ is stable in the sense of Definition \ref{stable} and that any $\vdelta\in \g$ is $W_\lambda$-stable in the sense of Definition \ref{W-stable} (see Remark \ref{R7.5}).

\subsection{Review} 
Recall that a \spinc structure $\s$ on a smooth 4-manifold $X$ is a pair $(S_X,\rho_4)$ where $S_X=S^+\oplus S^-$ is the spin bundle, and the bundle map $\rho_4: T^*X\to \Hom(S_X,S_X)$ defines the Clifford multiplication. A configuration $\gamma=(A,\Phi)\in \SC(X,\s)$ consists of a smooth \spinc connection $A$ and a smooth section $\Phi$ of $S^+$. Take $A^t$ to be the induced connection on $\Lambda^2 S^+$. Let $\omega$ be a closed 2-form on $X$ and $\omega^+$ denote its self-dual part. The Seiberg-Witten equations perturbed by $\omega$ are defined on $\SC(X,\s)$ by the formula:
\begin{equation}\label{SWEQ}
\left\{
\begin{array}{r}
\half \rho_4(F_{A^t}^+)-(\Phi\Phi^*)_0-\rho_4(\omega^+)=0,\\
D_A^+\Phi=0,
\end{array}
\right.
\end{equation}
where $D_A^+: \Gamma(S^+)\to \Gamma(S^-)$ is the Dirac operator and $(\Phi\Phi^*)_0=\Phi\Phi^*-\half |\Phi|^2\otimes\Id_{S^+}$ denotes the traceless part of the endomorphism $\Phi\Phi^*:S^+\to S^+$. 

The gauge group $\CG(X)=\map (X,S^1)$ acts naturally on $\SC(X,\s)$ by the formula:
\[
\CG(x)\ni u: \SC(X,\s)\to \SC(X,\s),\ (A,\Phi)\mapsto (A-u^{-1}du, u\Phi).
\]
The monopole equations (\ref{SWEQ}) are invariant under gauge transformations. Let $\TSigma=(\Sigma, g_\Sigma,\lambda,\nu)$ be any $H$-surface. In the special case that $X=\C\times \Sigma$ is equipped with the product metric and the complex orientation, the equations (\ref{SWEQ}) can be understood more concretely as follows.

Let $L^+\to \Sigma$ be a line bundle of degree $0\leq d\leq 2g(\Sigma)-2$ and set $L^-\colonequals L^+\otimes \Lambda^{0,1}\Sigma$. They pull back to line bundles over $X$, denoted also by $L^+$ and $L^-$ respectively. Define $S^+\colonequals L^+\oplus L^-$ and so
\[
c_1(S^+)[\Sigma]=2(d-g(\Sigma)+1). 
\]

The spinor $\Phi$ decomposes as $(\Phi_+,\Phi_-)$ with $\Phi_\pm\in \Gamma(X,L^\pm)$. Let $z=t+is$ be the complex coordinate on $\C$. The Clifford multiplication $\rho=\rho_4: T^*X\to \Hom(S,S)$ is defined by setting
\begin{eqnarray*}
	\rho_4(dt)=\begin{pmatrix}
		0& -\Id \\
		\Id & 0\\
	\end{pmatrix},\ 
	\rho_4(ds)=\begin{pmatrix}
		0& \sigma_1\\
		\sigma_1 & 0\\
	\end{pmatrix} :\ S^+\oplus S^-\to S^+\oplus S^-,
\end{eqnarray*}
where $\sigma_1=\begin{pmatrix}
i & 0\\
0 & -i\\
\end{pmatrix}: S^+=L^+\oplus L^-\to L^+\oplus L^-$ is the first Pauli matrix. Moreover, for any $x\in \Sigma$ and $w\in T_x\Sigma$, we have 
\[
\rho_3(w)\colonequals\rho_4(dt)^{-1}\cdot\rho_4(w)= \begin{pmatrix}
0 & -\iota(\sqrt{2}w^{0,1})\ \cdot \\
\sqrt{2}w^{0,1}\otimes \cdot  & 0\\
\end{pmatrix}: S^+\to S^+.
\]

\begin{remark} We will frequently work with Clifford multiplications in dimension $2,3$ and $4$, denoted by $\rho_2$, $\rho_3$ and $\rho_4$ respectively. Identify $\C$ as $\R_t\times \R_s$, then they are related by 
	\[
	\rho_3(w)=\rho_4(dt)^{-1}\cdot\rho_4(w),\ \rho_2(v)=\rho_3(ds)^{-1}\cdot \rho_3(v): S^+\to S^+,
	\]
	for any $w\in T^*(\R_s\times \Sigma)$ and $v\in T^*\Sigma$. 
\end{remark}

The decomposition $S^+=L^+\oplus L^-$ is also parallel, so any \spinc connection $A$ splits as 
\[
\nabla_A=\begin{pmatrix}
\nabla_{A_+} & 0 \\
0 & \nabla_{A_-}\\
\end{pmatrix}.
\]
Let $\cB_0$ be any reference \spinc connection on $S^+\to \Sigma$. Then a reference \spinc connection $A_0$ on $S^+\to X$ is obtained by
\[
\nabla_{A_0}=\nabla_{\cB_0}+d_\C
\]
where $d_\C=dt\otimes \pt+ds\otimes \ps$ is the trivial connection on $\C$. Any \spinc connection $A$ on $X$ differs from $A_0$ by an imaginary-valued 1-form $a\in \Gamma(X,iT^*X)$. Their curvature tensors are related by
\[
F_A=F_{A_0}+da\otimes \Id_S. 
\]
Using the product structure on $X$,  the connection $\nabla_A=(\nabla^\C_A,\nabla_A^\Sigma)$ is decomposed into a $\C$-direction part and a $\Sigma$-direction part. The curvature tensor $F_A$ is decomposed accordingly as:
\[
F_A=F_A^\Sigma+F_A^\C+F^m_A
\]
where $F_A^m\in \Gamma(X,T^*\C\otimes T^*\Sigma\otimes\End(S))$ is the mixed term. A similar decomposition applies to the induced curvature form $F_{A^t}$ on $\Lambda^2 S^+=L^+\otimes L^-$: 
\begin{equation}\label{E7.2}
F_{A^t}=F_{A^t}^\Sigma dvol_\Sigma+ F_{A^t}^\C dvol_\C+F_{A^t}^m,
\end{equation}
with $F^m_{A^t}\in \Gamma(X,iT^*\C\otimes T^*\Sigma)$. Our description of $F_A$ shows that
\begin{equation}\label{22}
F_A^m=\half F_{A^t}^m\otimes \Id,
\end{equation}
and 
\begin{equation}\label{23}
F^\Sigma_A=\begin{pmatrix}
F_{A_+}^\Sigma & 0 \\
0 & F_{A_-}^\Sigma
\end{pmatrix}dvol_\Sigma=\begin{pmatrix}
\half F_{A^t}^\Sigma+\frac{i}{2} K_\Sigma& 0 \\
0 & \half F_{A^t}^\Sigma-\frac{i}{2} K_\Sigma
\end{pmatrix}dvol_\Sigma,
\end{equation}
where $K_\Sigma$ is the Gaussian curvature of $(\Sigma,g_\Sigma)$. 
\subsection{The Fundamental Gauged Landau-Ginzburg Model}\label{Subsec6.2} In this subsection, we provide a different perspective on the Seiberg-Witten equations on $\C\times \Sigma$ using gauged Landau-Ginzburg models and the gauged Witten equations. The fundamental gauged Landau-Ginzburg model 
$$\big(M(\Sigma,\s), W_\lambda, \CG(\Sigma)\big)$$
 defined below will allow us to apply results from Part \ref{Part1} to the Seiberg-Witten equations on $\C\times \Sigma$ and $\HH^2_+\times\Sigma$. 

\medskip

$\bullet$ \textit{The K\"{a}hler manifold} $M(\Sigma,\s)$ is the configuration space on $\Sigma$:
\[
M(\Sigma,\s)\colonequals(\cB_0,0)+\Omega^1(\Sigma, i\R)\oplus \Gamma(\Sigma, L^+\oplus L^-),
\]
where $\cB_0$ is an arbitrary reference \spinc connection on $S^+\to \Sigma$.  A configuration $\kappa\in M(\Sigma,\s)$ is a triple $(\cb,\cPsi_+,\cPsi_-)$, where the sum $\cB=\cB_0+\cb$ is a \spinc connection on $\Sigma$ and $\cPsi=(\cPsi_+,\cPsi_-)\in \Gamma(\Sigma, S^+)$ is a spinor. The complex structure of $M(\Sigma,\s)$ is defined by the bundle map
\[
J=\big(*_\Sigma, \rho_3(ds)\big)=\big(*_\Sigma,\sigma_1= \begin{pmatrix}
i & 0\\
0 & -i
\end{pmatrix}\big),
\]
while the Riemannian metric $g_M$ of $M(\Sigma,\s)$ is the flat $L^2$ metric:
\[
\langle (\cb_1,\cPsi_1), (\cb_2,\cPsi_2)\rangle=\int_{\Sigma}\langle \cb_1,\cb_2\rangle+\re \langle \cPsi_1,\cPsi_2\rangle. 
\]
Let $h_M$ be the Hermitian metric on $M(\Sigma,\s)$ induced from $J$ and $g_M$.

\medskip

$\bullet$ \textit{The gauge group} $\CG(\Sigma)=\Map(\Sigma,S^1)$ acts on $M(\Sigma,\s)$ by the standard formula:
\[
u(\cb,\cPsi)=(\cb-u^{-1}du, u\cPsi). 
\]

The Lie algebra of $\CG(\Sigma)$ is $\Lie(\CG)=\Gamma(\Sigma, i\R)$. It is also useful to consider a smaller group $\CG^e$, the identity component of $\CG$, which fits into a short exact sequence:
\[
0\to \CG^e\to \CG\xrightarrow{\pi} H^1(\Sigma;\Z)\to 0,\ \pi (u)=[ \frac{u^{-1}du}{2\pi i}].
\]


\medskip

$\bullet$ \textit{The moment map} $\mu$ is given by
\begin{align*}
\mu(\cb,\cPsi)&=-\half *_\Sigma F_{\cB^t}+\frac{i}{2}(|\cPsi_+|^2-|\cPsi_-|^2)\\
&=-*_\Sigma d\cb+\frac{i}{2}(|\cPsi_+|^2-|\cPsi_-|^2)-\half *_\Sigma F_{\cB^t_0}.
\end{align*}

 If $v=(\cvb,\cvpsi)$ is a tangent vector at $(\cb,\cPsi)$, then we have
\begin{align}
\langle \nabla\mu, v\rangle_{TM}&=-*_\Sigma d_\Sigma\cvb+i\re\langle i\cvpsi, \rho_3(ds)\cPsi\rangle\in\Lie(\CG),\label{F6.4}\\
\langle\nabla\mu, Jv\rangle_{TM}&=d^*_\Sigma\cvb+i\re\langle i\cvpsi,\cPsi\rangle\in \Lie(\CG). \nonumber
\end{align}

\medskip

$\bullet$ \textit{The superpotential} $W_\lambda$ is the perturbed Dirac functional. The Clifford multiplication on $\Sigma$
\[
\rho_2: T^*\Sigma\to \Hom(S^+, S^+),
\]
defines a Dirac operator for each \spinc connection $\cB$ on $S^+\to\Sigma$:
\[
D^\Sigma_{\cB}: \Gamma(\Sigma,S^+)\xrightarrow{\nabla_{\cB}} \Gamma(\Sigma,T^*\Sigma\otimes S^+)\xrightarrow{\rho_2}\Gamma(\Sigma,S^+).
\]
This operator is self-adjoint and switches the parity, i.e. 
\[
D^\Sigma_{\cB}=\begin{pmatrix}
0 & D^-_{\cB}\\
D^+_{\cB} & 0
\end{pmatrix}: \Gamma (\Sigma, L^+\oplus L^-)\to  \Gamma (\Sigma, L^+\oplus L^-).
\]

The superpotential $W_0$ is then defined as
\[
W_0(\cb,\cPsi_+,\cPsi_-)=\int_{\Sigma} \langle D^+_{\cB}\cPsi_+,\cPsi_-\rangle=\int_{\Sigma} \langle \cPsi_+,D^-_{\cB}\cPsi_-\rangle.
\]
The perturbation that we consider takes the form 
\[
W_\lambda(\cb,\cPsi)=W_0-\langle \cb, \lambda\rangle_{h_M},
\]
where $\lambda\in \Omega^1(\Sigma,i\R)$ and $h_M$ is the Hermitian inner product. 

\medskip

$\bullet$ \textit{The complex gauge group} $\CG_\C=\Map(\Sigma,\C^*)$ acts on $M$ by the formula:
\begin{equation}\label{GC-action}
e^{\alpha}u(\cb,\cPsi)=(\cb+i*_\Sigma d_\Sigma \alpha-u^{-1}du,e^{\alpha}u\cPsi_+,e^{-\alpha}u\cPsi_-),
\end{equation}
where $u\in \CG(\Sigma)$ and $\alpha: \Sigma\to \R$ is real. 
\begin{lemma}\label{L6.2} The superpotential $W_0: M(\Sigma)\to \C$ is invariant under $\CG_\C$.
\end{lemma}
\begin{proof}[Proof of Lemma \ref{L6.2}] It suffices to verify that for all $\alpha\in \Gamma(\Sigma, \R)$, $e^{-\alpha}	D_{\cB'}^+ (e^{\alpha}\ \cdot\ )=D_{\cB}^+(\cdot)$ if $\cB'=e^{\alpha}\cdot \cB$, or equivalently
	\[
	\rho_2(d\alpha)+\rho_2(i*_\Sigma d\alpha)=0: \Gamma(L^+)\to \Gamma(L^-). 
	\]
	Note that when restricted to $\Gamma(L^-)$, it becomes:
\[
	\rho_2(d\alpha)-\rho_2(i*_\Sigma d\alpha)=0: \Gamma(L^-)\to \Gamma(L^+). \qedhere
\]
\end{proof}

 As for the perturbed superpotential $W_\lambda$, 
\begin{itemize}
\item for $W_\lambda$ to be invariant under $\CG^e$, $\lambda$ has to be co-closed;

\item 	for $W_\lambda$ to be invariant under the identity component $\CG^e_\C$ of $\CG_\C$, $\lambda$ has to be harmonic;

\item for $W_\lambda$ to be invariant under $\CG_\C$, $\lambda$ has to zero.		
	\end{itemize}

Write $W_\lambda=L+i H$. Then 
\begin{equation}\label{F6.5}
	\nabla L(\cb,\cPsi)=\big(\rho_2^{-1}(\cPsi\cPsi^*)_\Pi-\lambda, D^-_{\cB}\cPsi_-, D^+_{\cB}\cPsi_+\big). 
\end{equation}
Here $(\cPsi\cPsi^*)_\Pi$ denotes the off-diagonal part of the endomorphism $\cPsi\cPsi^*: L^+\oplus L^-\to L^+\oplus L^-$.	The equation $\nabla L=0$ has solutions if and only if  $\lambda$ is a harmonic 1-form. If in addition its $(1,0)$-part $\lambda^{1,0}\in \Omega^1(\Sigma,\C)$ has $(2g(\Sigma)-2)$ simple zeros, we will  characterize the critical locus $\Crit(L)$ concretely in Proposition \ref{stableW}. Now consider a connection $\bar{A}$ on the trivial principal bundle $\C\times \CG(\Sigma)$ over $\C$:
	\[
	\bar{A}=d_\C+a_t(z)dt+a_s(z)ds
	\]
	with $a_t,a_s\in \Gamma\big(\C, \Lie \CG(\Sigma)\big)=\Gamma(\C\times \Sigma,i\R)$.

\begin{proposition}\label{P6.2} With the gauged Landau-Ginzburg model $(M(\Sigma,\s), W_\lambda, \CG(\Sigma))$ defined as above, the associated gauged Witten equations over $\C$ with $\vdelta=-*_\Sigma \nu\in \Lie(\CG)$ are equivalent to the Seiberg-Witten equations $(\ref{SWEQ})$ on $\C\times \Sigma$ with $\omega=\nu+ds\wedge \lambda$. Let
	\begin{align*}
P: \C&\to M(\Sigma,\s),\hspace{1em} z\mapsto \big(\cb(z),\cPsi(z)\big)
	\end{align*}
be a smooth map defined on $\C$. Then the  identification $
(A,\Phi)\leftrightarrow (\bar{A},P)
$ is made by taking
	\begin{align*}
	A-A_0&=(\bar{A}-d_\C)+(\cB(z)-\cB_0)=a_t(z)dt+a_s(z)ds+\cb(z),\\
	\Phi&=\cPsi(z) \text{ on } \{z\}\times\Sigma.
	\end{align*}
\end{proposition}
\begin{proof}
	The $J$-holomorphic equation in (\ref{generalized-vortex-equation}) in our case becomes
	\[
	\nabla^{\bar{A}}_{\partial_t}\begin{pmatrix}
	\cb\\
	\cPsi
	\end{pmatrix}+
	\begin{pmatrix}
	*_\Sigma & 0\\
	0 & \rho_3(ds)
	\end{pmatrix}\bigg(
	\nabla^{\bar{A}}_{\partial_s}\begin{pmatrix}
	\cb\\
	\cPsi
	\end{pmatrix}+\nabla L
	\bigg)=0.
	\]
	More concretely, it is 
	\begin{align}\label{6.4}
	(\pt \cb-d_\Sigma a_t)+*_\Sigma\big(\ps \cb-d_\Sigma a_s+\rho_2^{-1}(\cPsi\cPsi^*)_\Pi-\lambda\big)&=0,\\
	(\pt \cPsi+a_t\cPsi)+\rho_3(ds)(\ps\cPsi+a_s\cPsi+D^\Sigma_{\cB}\cPsi)&=0.\nonumber
	\end{align}
	The second equation gives rise to the Dirac operator $D_A^+\Phi=0$,
	while the first equation gives the off-diagonal part of the curvature equation:
	\[
	\half \rho_4(F_{A^t}^+)_\Pi-(\Phi\Phi^*)_\Pi-\rho_4(\omega^+)_\Pi=0
	\]
	with $\omega=\nu+ds\wedge \lambda$.
	The diagonal part comes from the moment map equation in (\ref{generalized-vortex-equation}): 
	\[
	-*_\C d_\C(a_tdt+a_s ds)-*_\Sigma d_\Sigma \cb+\frac{i}{2}(|\cPsi_+|^2-|\cPsi_-|^2)-\half *_\Sigma F_{\cB_0^t}=\vdelta.
	\]
	Indeed, $\half F_{A^t}^\Sigma=d_\Sigma\cb+\half F_{\cB_0^t}$ and $\half F_{A^t}^\C=d_\C (a_tdt+a_sds)$ in terms of the decomposition \eqref{E7.2}.  
\end{proof}

\subsection{Stability}
Now we examine the stability of the superpotential $W_\lambda$. Although $W_\lambda$ is not fully $\CG_\C$-invariant when $\lambda\neq 0$, $\nabla L=0$ is a $\CG_\C$-invariant equation on $M(\Sigma,\s)$. 

\begin{proposition}\label{stableW} Given any $H$-surface $\TSigma=(\Sigma,g_\Sigma,\lambda,\nu)$, consider the \spinc structure $\s$ on $\Sigma$ with $c_1(\s)[\Sigma]=2(d-g(\Sigma)+1)$ and the fundamental gauged Landau-Ginzburg model $(M(\Sigma,\s), W_\lambda, \CG(\Sigma))$ constructed above. Then $\Crit(L)$ consists of $\binom{2g-2}{d}$ free $\CG_\C(\Sigma)$-orbits, and each $\CG_\C$-orbit $\SO\subset \Crit(L)$ is closed under the smooth topology of $M(\Sigma,\s)$. Moreover, for any such $\SO$ and any $\vdelta\in \Lie(\CG)$, $\mu^{-1}(\vdelta)\cap \SO$ consists of a unique $\CG(\Sigma)$-orbit and $W_\lambda$ is a Morse-Bott function. 
\end{proposition}
\begin{remark}\label{R7.5} In the sense of Proposition \ref{stableW}, we say that $W_\lambda$ is stable and any $\vdelta\in \Lie(\CG)$ is $W_\lambda$-stable. The only difference from Definition \ref{W-stable} and \ref{stable} is that $W_\lambda$ is only invariant under the smaller group $\CG_\C^e$, while the finiteness of critical orbits is stated for the full $\CG_\C$-action. In particular, for this infinite dimensional example,  the superpotential $W_\lambda$ has infinitely many critical values. 
\end{remark}

\begin{proof}[Proof of Proposition \ref{stableW}] The verification that $W_\lambda$ is Morse-Bott is postponed to Proposition \ref{P7.7}, which concerns only the linearized operator at critical $\CG_\C$-orbits. We focus on the other statements. The equation $\nabla L=0$ reads: 
	\begin{equation}\label{6.6}
\big(\rho_2^{-1}(\cPsi\cPsi^*)_\Pi-\lambda, D^-_{\cB}\cPsi_-, D^+_{\cB}\cPsi_+\big)=0. 
	\end{equation}

	The last two equations in \eqref{6.6} imply that $\cPsi_+$ and $\cPsi_-^*$ are holomorphic with respect to some unitary connections on $L^+$ and $(L^-)^*$ respectively, while the first one says that 
	\begin{equation}\label{E7.9}
	\cPsi_+\otimes \cPsi_-^*=-\sqrt{2}\lambda^{1,0}.
	\end{equation}
	Since $\lambda^{1,0}$ is assumed to have only simple zeros, the zero loci $Z(\cPsi_+)$ and $Z(\cPsi_-^*)$ give rise to  a partition of $Z(\lambda^{1,0})$. Since $|Z(\cPsi_+)|=\deg L^+=d$, there are  
	$
	\binom{2g-2}{d}
	$
	such partitions in total, each of which corresponds to a free $\CG_\C$-orbit in $\Crit(L)$. 
	
	\medskip
	
	Fix a representative $\kappa=(\cB,\cPsi_+,\cPsi_-)$ in a free $\CG_\C$-orbit $\SO\subset \Crit(L)$. To show this $\CG_\C$-orbit $\SO$ is closed under the smooth topology of $M(\Sigma,\s)$, consider a sequence of configurations $\kappa_j= (\cB^j,\cPsi_+^j, \cPsi_-^j)\in \SO$ converging smoothly to some $\kappa_\infty=(\cB^\infty,\cPsi_+^\infty, \cPsi_-^\infty)\in M(\Sigma,\s)$. We have to show $\kappa_\infty\in\SO$. Since $Z(\cPsi_+^j)=Z(\cPsi_+)$ for all $j$ and the convergence is in smooth topology, $Z(\cPsi_+^\infty)=Z(\cPsi_+)$. Since $\Crit(L)$ is closed, this implies that $\kappa_\infty$ and $\kappa$ lie in the same $\CG_\C$-orbit.
	

	Take any smooth function $\vdelta\in \Gamma(\Sigma, i\R)$. To show that $\vdelta$ is $W_\lambda$-stable, we have to find the unique function $\alpha: \Sigma\to \R$ such that 
	\[
	e^{\alpha}\cdot \kappa=(\cB+i*_\Sigma d_\Sigma\alpha, e^{\alpha}\cPsi_+, e^{-\alpha}\cPsi_-),
	\]
lies in the fiber $\mu^{-1}(\vdelta)$, or equivalently, to solve the equation 
	\begin{equation}\label{criticalorbit}
	i\big(\Delta_\Sigma\alpha+\half (e^{2\alpha}|\cPsi_+|^2-e^{-2\alpha}|\cPsi_-|^2)+\frac{i}{2}*_\Sigma F_{\cB^t}\big)=\vdelta,
	\end{equation}
	and show that this solution is unique. The strategy is to use the variational principle to prove that the non-linear map:
	\begin{align*}
	\eta: L^2_{k}(\Sigma, \R)&\to L^2_{k-2}(\Sigma, \R),\\
	\alpha&\mapsto \eta(\alpha)=\Delta_\Sigma\alpha+\half (e^{2\alpha}-1)|\cPsi_+|^2+\half (1-e^{-2\alpha})|\cPsi_-|^2
	\end{align*}
	is a bijection for any $k\geq 2$. It suffices to prove the special case for $k=2$; the rest will follow by elliptic regularity. For any $g\in L^2(\Sigma, \R)$, define an energy functional as 
	\[
	\E_g(\alpha)=\|\eta(\alpha)-g\|^2_2.
	\]
	
	If $\alpha_0$ achieves the infimum $\inf_{\alpha\in L^2_2}\E_g(\alpha)$, let $f=\eta(\alpha_0)-g\in L^2(\Sigma)$. Then, for any tangent vector $v\in L^2_2(\Sigma, \R)$ at the minimizer $\alpha_0$, we have 
	\[
	0=\pt \E_g(\alpha_0+tv)=\langle f, \Delta_\Sigma v+(e^{2\alpha_0}|\cPsi_+|^2+e^{-2\alpha_0}|\cPsi_-|^2) v\rangle. 
	\]
	Since the linearized operator $\Delta_\Sigma +(e^{2\alpha_0}|\cPsi_+|^2+e^{-2\alpha_0}|\cPsi_-|^2)$ is positive and self-adjoint on $L^2_2$, $f=0$. To find such a minimizer $\alpha_0$, let $\{\alpha_n\}\subset L^2_2$ be a sequence that minimizes the energy $\E_g(\alpha)$, i.e.
	\[
	\lim_{n\to\infty} \E_g(\alpha_n)=\inf_{\alpha\in L^2_2}\E_g(\alpha). 
	\]
	
At this point, we need an a priori estimate: 
	\begin{lemma} \label{apriori}For any fixed $g\in L^2(\Sigma)$, there exists a function $\varphi: \R_+\to \R_+$ such that for any $\alpha\in L^2_2$, $\E_g(\alpha)<C$ implies that $\|\alpha\|_{L^2_2}\leq \varphi(C)$. 
	\end{lemma}

 Lemma \ref{apriori} allows us to find a weakly convergent subsequence among $\{\alpha_n\}$ in $L^2_2$. Let $\alpha_0$ be the limit. Then 
\[
\E_g(\alpha_0)\leq \lim_{n\to\infty}\E_g(\alpha_n)=\inf \E_g(\alpha),
\]
so the infimum is attained at $\alpha=\alpha_0$. This proves that $\eta: L^2_2\to L^2$ is surjective and so a solution to the equation \eqref{criticalorbit} exists. If $\eta(\alpha_1+\delta\alpha)=\eta(\alpha_1)$, then 
\[
\Delta_\Sigma\delta\alpha+\half (e^{2\delta\alpha}-1) e^{2\alpha_1}|\cPsi_+|^2+\half (1-e^{-2\delta\alpha})e^{-2\alpha_1}|\cPsi_-|^2=0. 
\]

 By the maximum principle, $\delta\alpha\equiv 0$ on $\Sigma$. This proves that $\eta$ is injective and so the solution to the equation \eqref{criticalorbit} is unique. It remains to verify Lemma \ref{apriori}. This estimate follows from Lemma \ref{L6.10}, in which we set $w_+\equiv |\cPsi_+|$ and $w_-=|\cPsi_-|$. Note that by \eqref{E7.9}, $w_\pm\not\equiv 0$. 
\end{proof}

\begin{lemma}\label{L6.10} Given smooth functions $w_+, w_-: \Sigma\to [0,\infty)$ such that $w_+\not\equiv 0$ and $w_-\not\equiv 0$,
consider the non-linear map
 \begin{align*}
\eta&: L^2_{2}(\Sigma, \R)\to L^2(\Sigma, \R),\hspace{1em} \alpha\mapsto\Delta_\Sigma\alpha+\half (e^{2\alpha}-1)w_+^2+\half (1-e^{-2\alpha})w_-^2.
	\end{align*}
	Then for any $g\in L^2(\Sigma)$, there exists a function $\varphi: \R_+\to \R_+$ such that for any $\alpha\in L^2_2$, $\|\eta(\alpha)-g\|_2^2<C$ implies that $\|\alpha\|_{L^2_2}\leq \varphi(C)$.
\end{lemma}

\begin{proof}[Proof of Lemma \ref{L6.10}] By the triangle inequality,
	\[
	\|\eta(\alpha)\|_2\leq \|\eta(\alpha)-g\|_2+\|g\|_2\leq \sqrt{C}+\|g\|_2.
	\]
	So it suffices to prove this lemma for $g=0$. Note that 
	\begin{equation}\label{E6.11}
	\|\eta(\alpha)\|_2\|\alpha\|_2 \geq\langle \eta(\alpha),\alpha\rangle_{L^2(\Sigma)}=\|d\alpha\|_2^2+\half \int_{\Sigma}  (e^{\alpha}-e^{-\alpha})\alpha (e^{\alpha}w_+^2+e^{-\alpha}w_-^2). 
	\end{equation}
	Let $\alpha_+\colonequals\half (\alpha+|\alpha|)$ be the positive part of $\alpha$. By \eqref{E6.11} and Kato's inequality, 
	\begin{align*}
	\|\eta(\alpha)\|_2\|\alpha\|_2&\geq \|d\alpha_+\|_2^2+\int_{\Sigma} \alpha_+^2w_+^2\geq c_+\|\alpha_+\|_{L^2_1(\Sigma)}^2
	\end{align*}
	for some $c_+>0$, because the weight function $w_+$ is somewhere positive. The negative part $\alpha_-\colonequals\half (\alpha-|\alpha|)$ of $\alpha$ is estimated in a similar way. As a result,
	\[
	\|\eta(\alpha)\|_2\|\alpha\|_2\geq 2c\big(\|\alpha_+\|_{L^2_1(\Sigma)}^2+\|\alpha_-\|_{L^2_1(\Sigma)}^2\big)\geq c\|\alpha\|_{L^1_2(\Sigma)}^2
	\]
	 for some $c>0$; so $	 \|\alpha\|_{L^2_1(\Sigma)}\leq c^{-1}\|\eta(\alpha)\|_2$. To estimate $\|\Delta\alpha\|_2^2$, note that the map
	 \[
	 \alpha\mapsto \eta(\alpha)-\Delta\alpha=\half (e^{2\alpha}-1)w_+^2+\half (1-e^{-2\alpha})w_-^2
	 \]
	 is continuous from $L^2_1(\Sigma)$ to $L^2(\Sigma)$ by Trudinger's inequality \cite[Chapter 13, Proposition 4.2]{PDEIII} or \cite[Proposition A.3]{W18}. As a result, 
	 \[
	 \|\Delta\alpha\|_2\leq \|\eta(\alpha)\|_2+\|\eta(\alpha)-\Delta\alpha\|_2\leq \|\eta(\alpha)\|_2+\varphi_1\big(\|\alpha\|_{L^2_1}\big)\leq \varphi\big(\|\eta(\alpha)\|_2\big)
	 \]
	 for some functions $\varphi, \varphi_1: \R_+\to \R_+$. This completes the proof of Lemma \ref{L6.10}.
\end{proof}


\subsection{The Morse-Bott Condition} In this subsection, we verify that $\re(W_\lambda)$ is a Morse-Bott function on $M(\Sigma,\s)$, which completes the proof of Proposition \ref{stableW}.

\begin{proposition}\label{P7.7} Given any $H$-surface $\TSigma=(\Sigma,g_\Sigma,\lambda,\nu)$, consider the \spinc structure $\s$ on $\Sigma$ with $c_1(\s)[\Sigma]=2(d-g(\Sigma)+1)$ and the fundamental gauged Landau-Ginzburg model $(M(\Sigma,\s), W_\lambda, \CG(\Sigma))$ constructed in Subsection \ref{Subsec6.2}. Then $L=\re(W_\lambda)$ is a Morse-Bott function. In particular, the Morse-Bott estimate \eqref{morse-bott} holds in our case. 
\end{proposition}
\begin{proof}  Since $M(\Sigma,\s)$ is a complex linear space, the tangent space at any $\kappa=(\cB,\cPsi)\in M(\Sigma,\s)$ is identified with 
	\[
	\SH\colonequals \Omega^1(\Sigma, i\R)\oplus \Gamma(\Sigma, L^+\oplus L^-).
	\]
	Let $\SH_k$ be the completion of $\SH$ with respect to the $L^2_{k,\cB}$ norm:
	\[
	\|(\cvb,\cvpsi)\|^2_{L^2_{k,\cB}}=\sum_{0\leq j\leq k} \int_{\Sigma} |\nabla^k \cvb|^2+|\nabla_{\cB}^k\cvpsi|^2. 
	\]
	
	This family of norms on the tangent bundle of $M(\Sigma,\s)$ is equivariant under the gauge action of $\CG(\Sigma)$. The Lie algebra $\g$ of $\CG(\Sigma)$ is $\Gamma(\Sigma, i\R)$ and let $\g_k=L^2_k(\Sigma, i\R)$ be its $L^2_k$-completion for $k=0,1$. To prove Proposition \ref{P7.7}, by Lemma \ref{L4.6}, it suffices to verify that the extended operator
	\[
	\widehat{D}_{\kappa}=
	\begin{pmatrix}
	0 & 0 & \langle \nabla\mu, \cdot \rangle_{T M}\\
	0 & 0 &  \langle J\nabla\mu, \cdot \rangle_{TM}\\
	\langle \nabla\mu, \cdot \rangle_\g &  \langle J\nabla\mu, \cdot \rangle_\g & \Hess L
	\end{pmatrix}: \g_1\oplus \g_1\oplus \SH_1\to \g_0\oplus \g_0\oplus \SH_0,
	\]
	is invertible for any $\kappa=(\cB,\cPsi)\in \Crit(L)$. Since $\widehat{D}_{\kappa}$ is $L^2$-self-adjoint and Fredholm, it suffices to show that $\widehat{D}_{\kappa}$ is injective. Notice that the images of 
	\[
\langle \nabla\mu, \cdot\rangle_\g,\ \langle J\nabla\mu, \cdot\rangle_\g,\ D_{\kappa}\colonequals \big(\langle \nabla\mu, \cdot \rangle_{TM}, \langle J\nabla\mu, \cdot \rangle_{TM}, \Hess L\big)
	\]
	are pairwise $L^2$-orthogonal in $\g_0\oplus \g_0\oplus \SH_0$. The first two are injective, because the $\CG_\C$-action is free at $\kappa$. We focus on the last operator 
	\[
	D_{\kappa}: \SH_1\to \g_0\oplus \g_0\oplus \SH_0. 
	\]
	
	Suppose $v= (\delta \cb, \delta\cPsi)\in \ker D_{\kappa}$, then the tangent vector $v$ solves the following equations by \eqref{F6.4} and \eqref{6.6}:
	\begin{align}
	-*_\Sigma d_\Sigma\cvb+i\re\langle i\cvpsi, \rho_3(ds)\cPsi\rangle&=0,\label{E9.3}\\
	d^*_\Sigma\cvb+i\re\langle i\cvpsi,\cPsi\rangle&=0,\label{E9.4}\\
	(\delta\cPsi \cPsi^*+\cPsi\delta\cPsi^*)_\Pi&=0,\label{E9.5}\\
	D^\Sigma_{\cB}\delta\cPsi+\rho_2(\delta \cb)\cPsi&=0.\label{E9.6}
	\end{align}
	
	The key observation is that the third equation \eqref{E9.5} is an algebraic constraint on the spinor $\delta\cPsi$. Since $\cPsi_+$ and $\cPsi_-$ do not have common zeros by \eqref{E7.9}, the spinor $\cPsi=(\cPsi_+, \cPsi_-)$ is nowhere vanishing on $\Sigma$ and so by \eqref{E9.5},
	\begin{equation}\label{E9.7}
	\delta \cPsi=(h\cPsi_+,-\bar{h}\cPsi_-)=i\beta \cPsi+(\alpha \cPsi_+,-\alpha \cPsi_-),
	\end{equation}
	for a complex valued function $h=\alpha+i\beta: \Sigma\to \C$. By \eqref{E9.6}\eqref{E9.7} and the fact that $D_{\cB}^\Sigma\cPsi=0$, we have 
	\[
	\rho_2(dh+\delta\cb)\cPsi_+=0,	\rho_2(-d\bar{h}+\delta\cb)\cPsi_-=0.
	\]
By the non-vanishing property of $\cPsi$, this means
	\[
	\delta \cb= i(*_\Sigma d\alpha-d\beta).
	\]
	In other words, $(\delta\cb,\delta\cPsi)$ is generated by the linearized action of $\CG_\C$ at $\kappa$. By the gauge fixing condition \eqref{E9.4}, $\beta\equiv 0$. By $\eqref{E9.3}$, $\alpha\equiv 0$.
\end{proof}

\subsection{A Remark on the 1-Form $\lambda$} For any $H$-surface $\TSigma=(\Sigma, g_\Sigma,\lambda,\mu)$, the holomorphic 1-form $\lambda^{1,0}$ is assumed to have $(2g(\Sigma)-2)$ simple zeros. As we shall see in this subsection, Proposition \ref{P7.7} does not hold if this assumption is dropped. 

 To this end, let $\SL_i, 1\leq i\leq 3$ be holomorphic line bundles over $\Sigma$ and $\sigma_i\neq 0\in H^0(\Sigma, \SL_i)$ be holomorphic sections. Suppose that $\deg \SL_3>0$ and $\sigma_1, \sigma_2$ do not have common zeros. Consider a solution $\kappa=(\cB,\cPsi_+,\cPsi_-)$ to the equation $\nabla L=0$ such that 
	\[
	-\sqrt{2}\lambda^{1,0}=\sigma_1\sigma_2\sigma_3^2,\ \cPsi_+=\sigma_1\sigma_3 \text{ and } \cPsi_-=(\sigma_2\sigma_3)^*.
	\]

	The goal is to find a non-zero vector $v=(\delta b,\delta \cPsi)\in \ker D_\kappa\subset  T_\kappa M(\Sigma,\s)$, i.e. a solution to the equations \eqref{E9.3}--\eqref{E9.6}. As in the proof of Proposition \ref{P7.7}, the algebraic equation \eqref{E9.5} implies that 
	\[
	(	\delta\Psi_+,	\delta\Psi_-)=\big(\sigma_1\sigma_4, -(\sigma_2\sigma_4)^*\big),
	\]
	for some smooth section $\sigma_4\in C^\infty(\Sigma, \SL_3)$. If there is a harmonic 1-form $\delta \cb^h$ on $\Sigma$ such that 
	\begin{equation}\label{E7.18}
	\bpartial_{\SL_3}\sigma_4+(\delta \cb^h)^{0,1}\otimes \sigma_3=0,
	\end{equation}
	then the tangent vector $v_0\colonequals \big(\delta \cb^h, \sigma_1\sigma_4, -(\sigma_2\sigma_4)^*\big)$ solves \eqref{E9.5} and \eqref{E9.6}. To fulfill \eqref{E9.3} and \eqref{E9.4}, we replace $v_0$ by 
	\[
	v_0+\big(i(*_\Sigma d\alpha-d\beta),(\alpha+i\beta)\cPsi_+,(-\alpha+i\beta)\cPsi_-\big)
	\]
	for some suitable functions $\alpha,\beta:\Sigma\to \R$. Thus it suffices to construct the pair $(\delta \cb^h, \sigma_4)$ satisfying the equation \eqref{E7.18}. To this end, we find a harmonic form $\delta \cb^h\neq 0$ such that 
	$(\delta \cb^h)^{0,1}\otimes \sigma_3$ is $L^2$-orthogonal to $\im (\bpartial_{\SL_3})^\perp.$ Alternatively, we show that the composition 
	\begin{equation}\label{EC.2}
	\begin{array}{rcl}
	\Omega_h^1(\Sigma, i\R)\to& C^\infty(\Sigma,\Lambda^{1,0}\Sigma\otimes \SL_3^*)&\to H^0(\Sigma,\Lambda^{1,0}\Sigma\otimes \SL_3^*)\\
	\delta b^h\mapsto& (\delta b^h)^{1,0}\otimes \sigma_3^*&\mapsto\Pi_3((\delta b^h)^{1,0}\otimes \sigma_3^*),
	\end{array}
	\end{equation}
	is not injective, where $\Pi_3$ denotes the $L^2$-orthogonal projection onto $H^0(\Sigma,\Lambda^{1,0}\Sigma\otimes \SL_3^*)$. By the Riemann-Roch theorem and Serre duality, 
	\begin{align*}
	\dim_\C H^0(\Sigma,\Lambda^{1,0}\Sigma\otimes \SL_3^*)&=g-1-\deg \SL_3+\dim_\C H^0(\Sigma, \SL_3)\\
	&\leq g=\dim_\C \Omega_h^1(\Sigma, i\R).
	\end{align*}
	The equality is only achieved when $\dim_\C H^0(\Sigma, \SL_3)=\deg \SL_3+1$. However, this can only happen if $\deg \SL_3=0$ or $g(\Sigma)=0$. As a result, the composition \eqref{EC.2} is not injective. It suffices to take $\delta b^h$ that lies in the kernel of \eqref{EC.2}. 

\section{Point-like Solutions Are Trivial}\label{Sec7}

With all machineries developed so far, we are ready to study the monopole equations on $\C\times \Sigma$. By \cite{W18}, finite energy solutions to the unperturbed equations on $\C\times\Sigma$ are non-trivial in general and can be classified algebraically. In our case, we show the other extreme: 
\begin{theorem}\label{pointlike} Given any $H$-surface $\TSigma=(\Sigma,g_\Sigma,\lambda,\nu)$, consider the \spinc structure $\s$ on $\Sigma$ with $c_1(\s)[\Sigma]=2(d-g(\Sigma)+1)$. For  the fundamental Landau-Ginzburg model $(M(\Sigma,\s),W_\lambda,\CG(\Sigma))$ constructed in Subsection \ref{Subsec6.2}, let $(\bar{A}, P)$ be a solution to the gauged Witten equations on the complex plane $\C$ with $\vdelta=-*_\Sigma \nu$: 
	\begin{equation}\label{SWEQ2}
	\left\{ \begin{array}{r}
		-*F_{\bar{A}}+\mu=\vdelta,\\
	\nabla^{\bar{A}}_{\partial_t}P +J\nabla^{\bar{A}}_{\partial_s}P+\nabla H=0,
	\end{array}
	\right.
	\end{equation}
	where $P: \C\to M(\Sigma,\s)$ is a smooth map and $\bar{A}$ is a smooth connection on the trivial $\CG(\Sigma)$-bundle over $\C$. If the analytic energy 
	\begin{equation}\label{7.2}
	\E_{an}(\bar{A},P; \C)=\int_{\C} |\nabla_{\bar{A}} P|^2+|\nabla H|^2+|F_{\bar{A}}|^2+|\vdelta-\mu|^2
	\end{equation}
	is finite, then $(\bar{A},P)$ is gauge equivalent to a constant configuration and so $\E_{an}(\bar{A},P; \C)=0$. 	
\end{theorem}

Recall from Proposition \ref{P6.2} that each $(\bar{A},P)$ is identified with a Seiberg-Witten configuration $(A,\Phi)$ on $\C\times\Sigma$, where $A$ is a \spinc connection and $\Phi$ is a spinor. Then the total energy \eqref{7.2} coincides with the classical analytic energy of $(A,\Phi)$ for the monopole equations. The actual statement is slightly stronger:

 \begin{lemma}[{\cite[Lemma 2.1]{W18}}]\label{D7.3} For any region $\Omega\subset \C$ and any configuration $\gamma=(\bar{A}, P)=(A,\Phi)$, define
	\begin{align*}
	\E_{an}(\bar{A}, P; \Omega)&\colonequals \int_{\Omega} |\nabla_{\bar{A}} P|^2+|\nabla H|^2+|F_{\bar{A}}|^2+|\vdelta-\mu|^2, \\
	\E_{an}(A,\Phi; \Omega)&\colonequals\int_{\Omega}\int_{\Sigma} \frac{1}{4}|F_{A^t}|^2+|\nabla_A\Phi|^2+|(\Phi\Phi^*)_0+\rho_4(\omega^+)|^2+\frac{K_\Sigma}{2}|\Phi|^2+\re\langle F^\Sigma_{A^t}, \vdelta\rangle, \qedhere
	\end{align*}
	where $K_\Sigma$ is the Gaussian curvature of $\Sigma$ and  $\omega=ds\wedge\lambda-\vdelta dvol_\Sigma$.
	Then 
	\[
	\E_{an}(\bar{A}, P; \Omega)=\E_{an}(A,\Phi; \Omega).
	\]
 We refer to either of them as the local energy functional of $\gamma$ over $\Omega$. 
\end{lemma}
\begin{proof}[Proof of Theorem \ref{pointlike}] We follow the proof of Theorem \ref{point-like-solutions}. Let $\SO_1,\cdots, \SO_k$ be the free $\CG_\C$-orbits in $\Crit(L)$ and $\kappa_j=(\cb^j,\cPsi^j_+,\cPsi^j_-)$ be a representative in $\SO_j\cap \mu^{-1}(\vdelta)$ for each $1\leq j\leq k$. Define a family of metrics on the quotient configuration space $M(\Sigma,\s)/\CG(\Sigma)$ using $L^2_l$-Sobolev norms:
	\[
	d_l([\kappa_1'],[\kappa_2'])\colonequals\inf_{g\in \CG} \|\kappa_1'-g\cdot \kappa_2'\|_{L^2_l}, \forall \kappa_1',\kappa_2'\in M(\Sigma,\s), l\geq0.
	\]

	We first verify the condition of Theorem \ref{point-like-solutions} by showing that	for some $1\leq j\leq k$
	\begin{equation}\label{7.3}
	d_l([P(t,s)], [\kappa_j])\to 0 \text{ as }z=t+is\to\infty,
	\end{equation}
and the convergence holds for all $l\geq 0$. Let $n=(n_1,n_2)\in \Z\times\Z\subset \C$ and $B(0,R)\subset \C$ be the disk of radius $R$. Define 
	\[
	(A_n, \Phi_n)(z,x)\colonequals (A,\Phi)(z+n, x), \forall z\in \Omega\colonequals \overline{B(0,10)} \text{ and }x\in\Sigma.
	\]
	Since the total energy $\E_{an}(A,\Phi; \C)=\E_{an}(\bar{A},P;\C)$ is finite, $\{(A_n,\Phi_n)\}$ is a family of solutions on $\Omega\times\Sigma$ with $\E_{an}(A_n,\Phi_n;\Omega)\to 0$ as $|n|\colonequals |n_1|+|n_2|\to\infty$. Here $\E_{an}(A_n,\Phi_n; \Omega)$ is the local energy functional defined in Lemma \ref{D7.3}.

\smallskip

Let $U_j$ be an $L^2_l$-neighborhood of $[\kappa_j]\in M(\Sigma, \s)/\CG(\Sigma)$ such that $U_j\cap U_{j'}=\emptyset$ if $j\neq j'$. By the standard compactness theorem \cite[Theorem 5.1.1]{Bible}, up to gauge transformations, any subsequence of $\{(A_n,\Phi_n)\}$ contains a further subsequence converging in $\SC^\infty$-topology in the interior. Let $(A_\infty, \Phi_\infty)$ be such a limit. Since $\E_{an}(A_\infty,\Phi_\infty; \Omega)=0$, $(A_\infty,\Phi_\infty)$ is gauge equivalent to a constant family of $\kappa_j$ for some $1\leq j\leq k$ with $\bar{A}=d_\C$. This argument also shows that the image of $P|_\Omega$ lies in some $U_j$ when $|n|\gg 1$. Since $\C\setminus B(0,R)$ is connected for any $R>0$, this subscript $j$ is independent of $n$. This proves (\ref{7.3}).

By Proposition \ref{stableW}, the superpotential $W_\lambda$ is stable. By Lemma \ref{vanishing-lemma}, if $(\bar{A},P)$ solves the gauged Witten equation (\ref{SWEQ2}), then $
\nabla L(P(z))\equiv 0.$
The stability of $W_\lambda$ now implies that $P(z)$ lies in the $\CG_\C$-orbit of $\kappa_j$, so $P(z)=e^{\alpha(z)}u(z)\cdot \kappa_j$ for some smooth functions $\alpha\in \Gamma(\C\times \Sigma, \R)$ and $u\in \Gamma(\C\times \Sigma, S^1)$. We set $u\equiv 0$ by a gauge transformation and so
\begin{equation}\label{E8.4}
P(z)=e^{\alpha(z)}\cdot \kappa_j=(\cb^j+i*_\Sigma d_\Sigma\alpha(z), e^{\alpha(z)}\cPsi^j_+, e^{-\alpha(z)}\cPsi^j_-).
\end{equation}

Write $\bar{A}=d_\C+a_t dt+a_s ds$. The second equation of $(\ref{SWEQ2})$ then implies (comparing (\ref{6.4})):
\begin{align*}
(\pt \cb-d_\Sigma a_t)+*_\Sigma(\ps \cb-d_\Sigma a_s)&=0,\\
(\pt \cPsi+a_t\cPsi)+\rho_3(ds)(\ps\cPsi+a_s\cPsi)&=0,
\end{align*}
so $a_t=-i\ps \alpha,\ a_s=i\pt \alpha.$ Since $\kappa_j\in \mu^{-1}(\vdelta)$, we have 
	\begin{equation}\label{momentmap}
	-*_\Sigma d_\Sigma \cb^j+\frac{i}{2}(|\cPsi_+^j|^2-|\cPsi_-^j|^2)-\half *_\Sigma F_{\cB^t_j}=\vdelta.
	\end{equation}

 Combined with \eqref{E8.4} and (\ref{momentmap}), the moment map equation in \eqref{SWEQ2} then becomes
\begin{equation}\label{7.5}
(\Delta_\C+\Delta_\Sigma)\alpha+\half (e^{2\alpha}-1)|\cPsi_+^j|^2+\half (1-e^{-2\alpha})|\cPsi_-^j|^2=0.
\end{equation}

By the boundary condition (\ref{7.3}), $\|\alpha(z)\|_\infty\to 0$ as $z\to\infty$. The maximum principle then implies that $\alpha\equiv 0$, so $(\bar{A},P)$ is gauge equivalent to the constant configuration $(P\equiv \kappa_j, \bar{A}\equiv d_\C)$. 
\end{proof}
 Theorem \ref{pointlike} will play an important role in the proof of compactness theorem in the second paper \cite{Wang20}. In practice, it is convenient to work with a weaker condition than the finiteness of total energy: 
 \[
 \E_{an}(\bar{A}, P;\C)<\infty.
 \]

 To state the next result, define $I_n=[n-2,n+2]_t\subset \R_t$. Choose a compact domain $\Omega_0\subset I_0\times [0,\infty)_s$ with smooth boundary such that 
 \begin{equation}\label{E11.1}
 I_0\times [1,3]_s\subset \Omega_0\subset I_0\times [0,4]_s.
 \end{equation}
  For any  $n\in\Z$ and $R\in \R$, define $ \Omega_{n,R}$ to be the translated domain:
 \begin{equation}\label{E11.2}
 \Omega_{n,R}\colonequals \{(t,s):(t-n,s-R)\in \Omega_0\}\subset  I_n\times \R_s.
 \end{equation}

\begin{proposition}\label{P7.4} There exists a constant $\epsilon_*>0$ with the following significance. Under the assumptions of Theorem \ref{pointlike}, suppose instead that the local energy functional 
	\[
	\E_{an}(\bar{A}, P;  \Omega_{n,R})<\epsilon_*
	\]
for all $|n|+|R|\gg 1$, then $(\bar{A},P)$ is gauge equivalent to the constant configuration. 
\end{proposition}

Apparently, Proposition \ref{P7.4} implies Theorem \ref{pointlike}. 

\begin{proof}[Proof of Proposition \ref{P7.4}] There are two ways to proceed. In the first approach,  we apply Theorem \ref{T9.1} below to show the total analytic energy $\E_{an}(\bar{A}, P; \C)$ is actually finite, since the local energy functional $\E_{an}(\bar{A}, P;  \Omega_{n,R})$ has exponential decay as $|n|+|R|\to \infty$. Then Proposition \ref{P7.4} follows from Theorem \ref{pointlike}. In the second approach, we adapt the proof of Theorem \ref{pointlike} to our situation. There are three major modifications:

\medskip

\Step 1. If $\epsilon_*$ is small enough, then the Morse-Bott estimate (\ref{morse-bott}) in the proof of Lemma \ref{vanishing-lemma} still holds for any $P(z)$ when $|z|\gg 1$. This step requires the compactness theorem \cite[Theorem 5.1.1]{Bible}. 

\medskip

\Step 2. In the proof of Lemma \ref{vanishing-lemma}, we concluded from (\ref{4.6}) that if $E(r_0)>0$, then 
\[
E(r)=\int_{B(0,r)} |\nabla H|^2
\]
blows up exponentially as $r\to\infty$. In our case, since $\E_{an}(\bar{A}, P;  \Omega_{n,R})$ is uniformly bounded for all $n\in \Z$ and $R\in \R$, $E(r)$ can grow at most at the rate $r^2$. We still arrive at a contradiction, so $\nabla H\equiv 0$. 

\medskip 

\Step 3. Finally, using the stability of the superpotential $W$, we have to show the equation (\ref{7.5}) can only have the trivial solution $\alpha\equiv 0$. At this point, we only know $\alpha$ is uniformly bounded on $\C\times \Sigma$ and we argue as follows. If $\alpha: \C\times\Sigma\to \R$ is a solution of (\ref{7.5}), then 
\begin{align}\label{E8.8}
\half (\Delta_{\C}+\Delta_\Sigma)\alpha^2&\leq \big\langle(\Delta_{\C}+\Delta_\Sigma)\alpha,\alpha\big\rangle \\
&=-\half \alpha(e^{2\alpha}-1)|\cPsi_+^j|^2-\half\alpha (1-e^{-2\alpha})|\cPsi_-^j|^2\leq 0.\nonumber
\end{align}
Set $V(z)\colonequals\int_{\{z\}\times \Sigma}\alpha^2$; then $V(z)$ is a bounded subharmonic function on $\C$ and so is constant. The inequality \eqref{E8.8} then implies that 
\[
\int_{\{z\}\times \Sigma} \alpha(e^{2\alpha}-1)|\cPsi_+^j|^2+\alpha (1-e^{-2\alpha})|\cPsi_-^j|^2=0.
\]
Since $\cPsi^j=(\cPsi_+^j,\cPsi_-^j)$ is nowhere vanishing, $\alpha(z)\equiv 0$. 
\end{proof}
\section{Proof of Theorem \ref{T1.3}}

In this section, we prove Theorem \ref{T1.3} by generalizing ideas from Section \ref{1.5}. 

\begin{theorem}\label{T9.1}  Given any $H$-surface $\TSigma=(\Sigma,g_\Sigma,\lambda,\nu)$, consider the \spinc structure $\s$ on $\Sigma$ with $c_1(\s)[\Sigma]=2(d-g(\Sigma)+1)$. Then there exist constants $\epsilon(\TSigma),\zeta(\TSigma)>0$ with the following significance. For the fundamental gauged Landau-Ginzburg model $(M(\Sigma,\s), W_\lambda,\CG(\Sigma))$ constructed in Subsection \ref{Subsec6.2}, let $\gamma=(\bar{A},P)$ be a solution to the gauged Witten equations \eqref{SWEQ2} on $\HH^2_+$ with $\E_{an}(\gamma;  \Omega_{n,R})<\epsilon$ for any $n\in \Z$ and $R\geq 0$. Then 
	\[
	\E_{an}(\gamma; \Omega_{n,R})<e^{-\zeta R}. 
	\] 
	Here the subset $ \Omega_{n,R}\subset \HH^2_+$ is defined as in \eqref{E11.2}.
\end{theorem}

\begin{proof} We adapt the proof of Theorem \ref{exponential-decay} and follow the notations from Proposition \ref{P7.7}. For any $\kappa=(\cB,\cPsi)\in M(\Sigma,\s)$, recall that $\SH_k$ is the completion of the tangent space $T_\kappa M$ with respect to the $L^2_{k,\cB}$ norm for any $k\geq 0$:
	\[
	\|(\cvb,\cvpsi)\|^2_{L^2_{k,\cB}}=\sum_{0\leq j\leq k} \int_{\Sigma} |\nabla^k \cvb|^2+|\nabla_{\cB}^k\cvpsi|^2.
	\]
  We claim that the trilinear tensors defined in the proof of Theorem \ref{exponential-decay}:
	\begin{align}\label{F11.1}
	\langle\nabla_\cdot \Hess H(\cdot),\cdot\rangle&:   \SH_1\otimes \SH_1\otimes \SH_1\to \R,\\
	\langle \Hess\mu(\cdot), \cdot\otimes \cdot\rangle&:\SH_1\otimes \SH_1\otimes \g_1\to \R, \nonumber
	\end{align}
	are bounded operators. Indeed, take tangent vectors $v_i=(\cvb_i, \cvpsi_i)\in T_\kappa M$ for $i=1,2$. Using (\ref{F6.4}) and (\ref{F6.5}), we compute that:
	\begin{align*}
	\Hess L(v_1)&=\big(\rho_2^{-1}(\cPsi\delta\cPsi_1^*+\delta\cPsi_1\cPsi^*)_0, D_B^\Sigma \delta\cPsi_1+\rho_2(\cvb_1)\cPsi\big),\\
	(\nabla_{v_2}\Hess L)(v_1)&=\big(\rho_2^{-1}(\delta\cPsi_2\delta\cPsi_1^*+\delta\cPsi_1\delta \cPsi_2^*)_0, \rho_2(\cvb_2) \delta\cPsi_1+\rho_2(\cvb_1)\delta\cPsi_2\big), \\
	\langle \Hess\mu(v_1), v_2\rangle&=i\re \langle i\delta\cPsi_1, \rho_3(ds)\delta\cPsi_2\rangle .
	\end{align*}
	
	Hence, the tensors in (\ref{F11.1}) are independent of $\kappa\in M$ and involve only pointwise multiplications of sections. Since $L^2_{1,\cB}\embed L^3$ in dimension $2$ (with a uniform Sobolev constant independent of $\cB$), and the multiplication map $L^3\times L^3\times L^3\to L^1$ is bounded. Our claim follows.
	
	Now we come to analyze the differential operators
	\begin{align*}
	D_\kappa: \SH_1&\to \g_0\oplus \g_0\oplus  \SH_0,\\
	v=(\cvb,\cvpsi)&\mapsto \big(\langle \nabla \mu, v\rangle_{TM}, \langle J\nabla\mu, v\rangle_{TM},\Hess L(v)\big),
	\end{align*}
	and $J\langle \nabla \mu, \cdot \rangle_\g : \g_1\to \SH_0,\
	\xi\mapsto  (-d\xi, \xi\cPsi). $
	
	\begin{lemma}\label{L12.4} Suppose $\kappa_*=(\cB_*, \cPsi_*)\in M$ is a reference configuration in $ \mu^{-1}(\vdelta)\cap\Crit(L)$. Then for any $r>0$, we can find an $L^r$-neighborhood $U_r(\kappa_*)$ of $\kappa_*$ and some $c>0$ such that for any $\kappa=(\cB,\cPsi)\in U_r(\kappa_*)$, $v\in T_\kappa M$ and $\xi\in \g$, we have
		\begin{equation}\label{F11.2}
		\|D_\kappa(v)\|_{L^2(\Sigma)}\geq c\|v\|_{L^2_{1,\cB}}\text{ and }\|J\langle \nabla\mu, \xi\rangle_\g \|_2\geq c\|\xi\|_{L^2_1}.
		\end{equation}
	\end{lemma} 
	\begin{proof}[Proof of Lemma] If $\kappa=\kappa_*$ is the reference connection, then the estimates \eqref{F11.2} follow from the injectivity of the extended operator $\widehat{D}_\kappa$ in the proof of Proposition \ref{P7.7}. In general, let $\delta\kappa=\kappa-\kappa_*$. Then $\widehat{D}_\kappa(v)=\widehat{D}_{\kappa_*} (v)+I(\delta\kappa,v)$ for a bilinear operator $I(\cdot, \cdot )$ involving only pointwise multiplication, so 
		\[
		\|I(\delta\kappa,v)\|_2\leq \|\delta\kappa\|_p\|v\|_q\leq C(q)\|\delta\kappa\|_p\|v\|_{L^2_{1,\cB}}
		\]
		for any positive $(p,q)$ with $1/p+1/q=1/2$. The constant $C(q)$ comes from the Sobolev embedding $L^2_{1,\cB}\embed L^q$. Similarly, we have 
		\[
		\|v\|_{L^2_{1,\cB_*}}\geq \|v\|_{L^2_{1,\cB}}-C(q)\|\delta\kappa\|_p\|v\|_{L^2_{1,\cB}}. 
		\]
		Thus the estimates (\ref{F11.2}) hold when $\|\delta\kappa\|_p\ll 1$.
	\end{proof}
	
	Back to the proof of Theorem \ref{T9.1} and continue our adaptation of the proof of Theorem \ref{exponential-decay}. Following the notations therein, define
	\begin{align*}
	u(z)&=\|\nabla^A P\|^2_{L^2(\{z\}\times\Sigma)}+\|F_{\bar{A}}\|^2_{L^2(\{z\}\times\Sigma)},\\
	 w(z)&=\|\nabla^A P\|^2_{\SH_1(\{z\}\times\Sigma)}+\|F_{\bar{A}}\|^2_{L^2_1(\{z\}\times\Sigma)}, 
	\end{align*}
	for all $z\in \HH^2_+$. For any $\eta>0$, by the compactness theorem \cite[Theorem 5.2.1]{Bible}, there exists a constant $\epsilon(\eta)>0$ such that for any configuration $\gamma=(\bar{A}, P)$ with 
\[
\E_{an}(\gamma, \Omega_0)=\int_{\Omega_0 } u(z)dz<\epsilon(\eta),
\]
we have the pointwise estimate 
	\[
	0\leq u(z)\leq w(z)\leq \eta, \forall z\in \Omega_0'
	\]
	for a smaller domain $\Omega_0'\subset \Omega_0$. By Proposition \ref{stableW}, $\mu^{-1}(\delta)\cap \Crit(L)$ consists of $\binom{2g-2}{d}$ free $\CG(\Sigma)$-orbits. Let them be $\SO_1,\cdots, \SO_k$ and $\kappa_j\in \SO_j$ be a reference configuration for each $1\leq j\leq k\colonequals\binom{2g-2}{d}$. By taking $\eta\ll 1$, we deduce that there exists some $1\leq j\leq k$ such that for all $z\in \Omega_0'$, $P(z)\in U_4(\kappa_j)$ (after a gauge transformation). Here $U_4(\kappa_j)$ is the neighborhood of $\kappa_j$ obtained in Lemma \ref{L12.4} with $p=4$.  By shrinking the size of $U_4(\kappa_j)$, we ensure that $U_4(\kappa_j)\cap U_4(\kappa_{j'})=\emptyset$ for all $j\neq j'$.
	
	\medskip
	
	Now replace $\Omega_0$ by $ \Omega_{n,R}$ for any $n\in \Z$ and $R\geq 0$. This implies that when $0<\eta\ll 1$ and 
	\[
	\E_{an}(\gamma;\Omega_{n,R})<\epsilon(\eta),
	\]
$P(z)$ lies in some $U_4(\kappa_j)$ (after a suitable gauge transformation) for all $z$ in a slightly smaller subdomain of $\Omega_{n,R}$ and this subscript $j$ is independent of $(n, R)$, since $\HH^2_+$ is connected. Hence, for some $1\leq j\leq k$, $P(z)\in U_4(\kappa_j)$ for all $z\in \R_t\times [1,+\infty)_s$. Using the Bochner-type formula from Lemma \ref{Bochner}, we deduce that 
	\[
	0\geq \half \Delta_{\HH^2_+}u+\zeta^2 w-Cw^{3/2}\geq \half \Delta_{\HH^2_+}u+\frac{\zeta^2}{2} w\geq \half (\Delta_{\HH^2_+}+\zeta^2)u,
	\]
	for some $\zeta>0$. Now we use Lemma \ref{maximum2} to conclude.
\end{proof}

\section{Finite Energy Solutions on $\R_s\times \Sigma$}\label{Sec8}

 Fix an $H$-surface $\TSigma=(\Sigma, g_\Sigma, \lambda,\nu)$. In this section, we study the 3-dimensional Seiberg-Witten equations:
\begin{equation}\label{3DDSWEQ}
\left\{
\begin{array}{r}
\half \rho_3(F_{B^t})-(\Psi\Psi^*)_0-\rho_3(\omega)=0,\\
D_B\Psi=0.
\end{array}
\right.
\end{equation}
on the cylinder $\R_s\times \Sigma$ with $\omega=\nu+ds\wedge \lambda$. In terms of gauged Landau-Ginzburg models, the equations (\ref{3DDSWEQ}) gives rise to a downward gradient flowline of $L=\re W_\lambda$:
\begin{equation}\label{8.2}
\ds p(s)+\nabla L\circ p=0 \text{ and } \mu(p(s))=\vdelta,
\end{equation}
where $p(s)=(\cb(s),\cPsi(s)): \R_s\to M(\Sigma,\s)$ is a smooth path in the K\"{a}hler manifold $M(\Sigma,\s)$. The relation of \eqref{8.2} with (\ref{3DDSWEQ}) can be seen by setting
\[
B=\ds+\cB_0+\cb(s) \text{ and } \Psi=\cPsi(s) \text{ on } \{s\} \times \Sigma,
\]
for a reference \spinc connection $\cB_0$. We require the path $p$ to have finite analytic energy:
\begin{equation}\label{3Denergy}
\E_{an}(p)\colonequals \int_{\R_s} |\ds p|^2+|\nabla L|^2<\infty. 
\end{equation}

By (\ref{3Denergy}) and the stability of $W_\lambda$, the path $p$ is of finite length and the limit
$
q_\pm=\lim_{s\to\pm\infty} p(s)
$
lies in $\mu^{-1}(\vdelta)\cap\Crit(L)$. By the Cauchy-Riemann equation (\ref{CRequation}), $p(s)$ follows the Hamiltonian flow of $H=\im W_\lambda$. Hence, a flowline connecting $q_-$ and $q_+$ exists only if 
\begin{equation}\label{constraint}
L(q_-)\geq L(q_+) \text{ and } H(q_-)=H(q_+). 
\end{equation}

This motivates the following definition:
\begin{definition} An $H$-surface $\TSigma=(\Sigma, g_\Sigma,\lambda,\nu)$ is called \textit{good} if for any critical values $a_1\neq a_2$ of $W_\lambda$, $\im a_1\neq \im a_2$. 
\end{definition}

One can always make $\TSigma$ good by replacing $\lambda$ by $e^{-i\theta}\lambda$ for some $e^{i\theta}\in S^1$, since the set of critical values of $W_\lambda$ is at most countable. Indeed, the relation
\begin{equation}\label{E10.5}
e^{i\theta}W_\lambda(\cb,\cPsi_+,\cPsi_-)=W_0(\cb,\cPsi_+, e^{-i\theta}\cPsi_-)-\langle \cb, e^{-i\theta}\lambda\rangle_{h_M}=W_{e^{-i\theta}\lambda}(\cb,\cPsi_+, e^{-i\theta}\cPsi_-),
\end{equation}
implies that $(\cb,\cPsi_+,\cPsi_-)\in \Crit(\re(W_\lambda))$ if and only if $(\cb,\cPsi_+,e^{-i\theta}\cPsi_-)\in \Crit(\re(W_{{e^{-i\theta}\lambda}}))$, but now critical values are rotated by the angle $e^{i\theta}$. 

\begin{corollary} If an $H$-surface $\TSigma$ is good, then finite energy solutions to \eqref{3DDSWEQ} are necessarily $\R_s$-invariant. 
\end{corollary}

When $g(\Sigma)=1$, we can understand good $H$-surfaces more concretely:

\begin{proposition}\label{8.1} If $g(\Sigma)=1$ and $\lambda \in \Omega^1_h(\Sigma)\cong H^1(\Sigma, i\R)$ is not a real multiple of an integral class, then $\TSigma$ is good. In particular, any finite energy solution of $(\ref{3DDSWEQ})$ has to be $\R_s$-invariant, i.e. $p(s)\equiv q_-\equiv q_+$.
\end{proposition}
\begin{proof} By Proposition \ref{stableW}, $\mu^{-1}(\vdelta)\cap \Crit(L)$ consists of a single $\CG(\Sigma)$-orbit, so $q_+=u\cdot q_-$ for some $u:\Sigma\to S^1$. Hence, 
	\[
	W_\lambda(q_-)-W_\lambda(q_+)=-\int_\Sigma \langle u^{-1}du, \lambda\rangle_{h_M}. 
	\]
	
	In particular, $H(q_-)-H(q_+)=4\pi ^2\big([\frac{u^{-1}du}{2\pi i}]\cup [\frac{\lambda}{2\pi i}]\big)[\Sigma]$. If $\lambda$ is proportional to an integral class over $\R$, then this pairing is non-zero unless $[\frac{u^{-1}du}{2\pi i}]=0\in H^1(\Sigma, \Z)$. This implies that
	\[
	\E_{an}(p)=2\big(L(q_-)-L(q_+)\big)=0,
	\]
so the path $p$ has to be $\R_s$-translation invariant. 
\end{proof}

\begin{remark} A solution of (\ref{3DDSWEQ}) can be viewed as an $S^1$-invariant solution of the 4-dimensional equations (\ref{SWEQ}) on $S^1\times \R_s\times \Sigma$. When $g_\Sigma$ is flat, Proposition \ref{8.1} follows from a theorem of Taubes  \cite[Proposition 4.4]{Taubes01}.
\end{remark}

Taubes' theorems provide another simple condition that precludes non-trivial solutions. 

\begin{proposition}[{\cite[Proposition 4.7]{Taubes01}}]\label{moment-perturbation} Let $(\Sigma, g_\Sigma)$ be a flat 2-torus and consider the $H$-surface $\TSigma=(\Sigma,g_\Sigma,\lambda,\nu)$ with $\nu$ harmonic. If $\langle \nu, [\Sigma]\rangle\neq 0$, then any finite energy solution of $(\ref{3DDSWEQ})$ has to be $\R_s$-invariant, i.e. $p(s)\equiv q_-\equiv q_+$.
\end{proposition}

\begin{proof}  The proof is borrowed from \cite[P. 486-487]{Taubes01}. Since $g_\Sigma$ is flat, the closed 2-form $
	\omega=ds\wedge\lambda+\nu
	$ in (\ref{3DDSWEQ}) 
 is parallel on $\R_s\times\Sigma$. The spin bundle $S^+$ splits as
 \begin{equation}\label{8.4}
L^+_\omega\oplus L^-_\omega
 \end{equation}  with $\rho_3(\omega)$ acting on by a diagonal matrix
 \[
m \begin{pmatrix}
-1 & 0\\
0 & 1
 \end{pmatrix}
 \]
where $m=\sqrt{|\nu|^2+|\lambda|^2}$ is a positive number. The splitting (\ref{8.4}) is parallel. Let $p(s)=(b(s),\Psi(s))$ be a solution of (\ref{8.2}) on $\R_s\times \Sigma$. Write $\Psi(s)=\sqrt{2m}(\alpha(s),\beta(s))$ with respect to the decomposition (\ref{8.4}). By Witten's  vanishing spinor argument \cite[Section 4]{Witten94}, $\beta\equiv 0$. The first equation of (\ref{3DDSWEQ}) becomes
\[
\half F_{B^t}=(1-|\alpha|^2)\omega. 
\]

The curvature form  $F_{B^t}$ is closed, so $d(1-|\alpha|^2)\wedge\omega=0$, which is also $\langle d |\alpha|^2, *_3 \omega\rangle=0$. The dual tangent vector of $i*_3\omega$ generates a flow on $\R_s\times \Sigma$ along which $|\alpha|^2$ stays constant. Since $\nu\neq 0$,  this flow translates the spatial coordinate $s$ as time varies. Since $|\alpha|\to 1$ as $s\to\pm \infty$, $|\alpha|\equiv 1$. This completes the proof. 
\end{proof}

When $(\Sigma, g_\Sigma)$ is flat, $\nu=0$ and $\lambda$ is a multiple of an integral class, there is a non-trivial moduli space of flowlines for any pair $(q_-,q_+)$ subject to (\ref{constraint}). They are pulled back from vortices on the cylinder $\R_s\times S^1$. These moduli spaces are not regular; their expected dimensions are always zero by the index computation. For more details, see \cite[Section 4(d)(e)]{Taubes01}. Here is an immediate corollary of Proposition \ref{moment-perturbation}.

\begin{corollary}\label{C9.4} Let $(\Sigma, g_\Sigma)$ be a flat 2-torus and consider the $H$-surface $\TSigma=(\Sigma,g_\Sigma,\lambda,\nu)$ with $\nu$ harmonic. If $\langle \nu, [\Sigma]\rangle\neq 0$, then for any $e^{i\theta}\in S^1$, any downward gradient flowline of the functional $\re(e^{i\theta}W_\lambda)$:
\[
\ps p(s)+\nabla \re(e^{i\theta}W_\lambda)=0,
\]
has to be a constant path.
\end{corollary}
\begin{proof} This corollary follows from \eqref{E10.5} and Proposition \ref{moment-perturbation}.
\end{proof}

%% file: maximumprinciple.tex
\section{The Maximum Principle}\label{max}

In this appendix, we prove a version of maximum principle. The author is greatly indebted to Ao Sun for teaching him the proof of Lemma \ref{maximumprinciple}. The Laplacian operator is always assumed to have a non-negative spectrum. In particular, over the complex plane,
\[
\Delta_{\C}\colonequals -\partial_t^2-\partial_s^2.
\]

This sign convention is adopted throughout this paper. 
\begin{proposition}\label{maximumprinciple} Take $\Lm>0$. Suppose $u: \HH^2_+=\R_t\times [0,\infty)_s\to \R$ is  a bounded $\SC^2$-function on the upper half plane such that 
	\begin{enumerate}
\item $(\Delta_\C+\Lm^2)u\leq 0$, and
\item $u(t,0)\leq 0$ for any $t\in \R_t$. 
	\end{enumerate}
Then $u(t,s)\leq 0$ for any $(t,s)\in \HH^2_+$. 
\end{proposition}
\begin{proof} Choose a smooth cut-off function $\psi: [0,\infty)\to [0,\infty)$ such that $0\leq \psi\leq 1$, $\psi \equiv 1$ on $[0,1]$ and $\equiv 0$ on $[2,\infty)$. Define $\phi_R(z)\colonequals \psi(|z|/R)$ for $z\in \C$. Then 
\begin{itemize}
\item $\phi_R\equiv 1$ when $|z|<R$ and $\equiv 0$ when $|z|>2R$;
\item $\nabla\phi_R=0$ and $\Delta_{\C}\phi_R\equiv 0$ when $|z|<R$ or $|z|>2R$;
\item for some $C>0$, $|\nabla \phi_R|<\frac{C}{R}$ and $|\Delta_{\C}\phi_R|<\frac{C}{R^2}$.
\end{itemize}

Only the last property requires some explanation. In general, we have 
\begin{align*}
\nabla\phi_R(z) &=\frac{1}{R} \phi' (\frac{|z|}{R})\partial_r, \\
\Delta_{\C}\phi_R(z)&=-(\partial_r^2\phi_R+\frac{1}{|z|}\partial_r\phi_R)=-\frac{1}{R^2}\psi''(\frac{|z|}{R})-\frac{1}{|z|R}\psi'(\frac{|z|}{R}). 
\end{align*}

Suppose $u(z_0)>0$ at some $z_0\in \HH^2_+$. Consider the function $u_R(z)\colonequals u(z)\cdot\phi_R(z-z_0) $. Then $u_R(z)\equiv 0$ when $|z-z_0|>2R$ and $
u_R(t,0)\leq 0$ for all $t\in \R_t$. Hence, $\max u_R$ is attained at some $z_1$ in the interior of $\HH^2_+$. Let $N=\|u\|_\infty$, so
\begin{equation}\label{C.1}
0<u(z_0)=u_R(z_0)\leq u_R(z_1)\leq N \phi_R(z_1-z_0).
\end{equation}

At $z_1\in \HH^2_+$, we have 
\[
0=(\nabla u_R)(z_1)=\nabla u(z_1)\cdot \phi_R(z_1-z_0)+u(z_1)\cdot\nabla\phi_R(z_1-z_0),
\]
so $\nabla u(z_1)=-u(z_1)\nabla\phi_R (z_1-z_0)/\phi_R(z_1-z_0)$. The relation $\Delta_{\C} u\leq -\Lm^2u$ then implies:
\begin{align*}
0\leq (\Delta_{\C} u_R)(z_1)&=\phi_R(z_1-z_0)\cdot\Delta_{\C} u(z_1)+u(z_1)\cdot\Delta_{\C}\phi_R(z_1-z_0)-2\nabla\phi_R(z_1-z_0)\cdot \nabla u(z_1),\\
&\leq u(z_1)\cdot\bigg(-\Lm^2\phi_R+\Delta_{\C}\phi_R+\frac{2|\nabla\phi_R|^2}{\phi_R}\bigg)(z_1-z_0).
\end{align*}

However, this inequality can not hold when $R\gg 1$. Indeed, by (\ref{C.1}),
\[
\bigg|\Delta_{\C}\phi_R+\frac{2|\nabla\phi_R|^2}{\phi_R}\bigg|(z_1-z_0)\leq \frac{C}{R^2}+\frac{2C^2}{R^2}\cdot \frac{N}{u(z_0)}<\frac{\Lm^2 u(z_0)}{N}\leq \Lm^2\phi_R(z_1-z_0),
\]
when $R\gg 1$. In the meantime, $u(z_1)>0$. A contradiction. This completes the proof of Proposition \ref{maximumprinciple}.
\end{proof}

\begin{corollary} \label{coro-max} Take $\Lm>0$. Suppose $u: \HH^2_+=\R_t\times [0,\infty)_s\to \R$ is  a bounded $\SC^2$-function on the upper half plane such that 
	\begin{enumerate}
		\item $(\Delta_\C+\Lm^2)u\leq 0$, and
		\item for some $K>0$, $u(t,0)\leq K$ for all $t\in \R_t$. 
	\end{enumerate}
	Then $u(t,s)\leq Ke^{-\Lm s}$ for any $(t,s)\in \HH^2_+$. 
\end{corollary}
\begin{proof} Let $v(t,s)=Ke^{-\Lm s}$. Then $(\Delta_{\C}+\Lm^2)v=0$ and $v(t,0)=K$ for any $t\in \R_t$. Apply Proposition \ref{maximumprinciple} to $u-v$ to conclude.
\end{proof}

There are analogous statements for a strip of finite length. Their proofs are similar and omitted here. 
\begin{proposition}
	Take $\Lm>0$. Suppose $u: \R_t\times [0,2R]_s\to \R$ is  a bounded $\SC^2$-function such that 
	\begin{enumerate}
		\item $(\Delta_\C+\Lm^2)u\leq 0$, and
		\item $u(t,s)\leq 0$ for all $t\in \R_t$ and $s\in \{0, 2R\}$. 
	\end{enumerate}
	Then $u(t,s)\leq 0$ for all $(t,s)\in \R_t\times [0,2R]_s$. 
\end{proposition}

\begin{corollary}\label{CA.4}  	Take $\Lm>0$. Suppose $u: \R_t\times [0,2R]_s\to \R$ is  a bounded $\SC^2$-function such that 
	\begin{enumerate}
		\item $(\Delta_\C+\Lm^2)u\leq 0$, and
		\item for some $K>0$, $u(t,s)\leq K$ for all $t\in \R_t$ and $s\in \{0, 2R\}$. 
	\end{enumerate}
	Then $u(t,s)\leq K\cdot \frac{\cosh(\Lm(s-R))}{\cosh(\Lm R)}$ for all $(t,s)\in \R_t\times [0,2R]_s$. 
\end{corollary}

%% file: BFormula.tex
\section{A Bochner-Type Formula}\label{B}

The purpose of this appendix is to summarize some formulae for a gauged Landau-Ginzburg model $(M,W,G,\rho)$ from Riemannian geometry. In particular, we will prove a Bochner-type formula for a solution $(A,P)$ to the gauged Witten equations \eqref{generalized-vortex-equation} on $\HH^2_+$. For our primary application, we will take $M$ to be a complex linear space, in which case most formulae will become simpler.

\subsection{Some Useful Formulae} Recall that $(M,\omega, J, g)$ is a K\"{a}hler manifold and $(G,\rho)$ is a compact abelian Lie group acting on $M$ preserving the K\"{a}hler structure. $(G,\rho)$ admits a moment map $\mu: M\to \g$ which is $G$-invariant. $W=L+i H$ is a $G_\C$-invariant holomorphic function on $M$, called the superpotential. 

For any $\xi\in \g$, let $\tilde{\xi}$ be the vector field on $M$ induced from the action $(G,\rho)$. The convention of the moment map used in our paper is that 
\[
\iota(\tilde{\xi})\omega=-d\langle \mu, \xi\rangle_\g,
\]
or equivalently,
\begin{equation}\label{convention1}
\tilde{\xi}=J\langle \nabla\mu, \xi\rangle_\g,
\end{equation}
since $\omega(\cdot,\cdot)=g(J\cdot, \cdot)$. Here $\nabla\mu\in \Gamma(M, TM\otimes \g)$ is viewed as a $\g$-valued vector field on $M$. 

\begin{lemma}\label{identities} For a gauged Landau-Ginzburg model $(M,W, G,\rho)$ defined as in Definition \ref{LGmodels},  we have the following identities:
	\begin{enumerate}
		\item\label{1} $\nabla L+J\nabla H=0. $
		\item\label{2} $\Hess L+ J\circ\Hess H=0$.
		\item\label{3} $J\circ \Hess H+\Hess H \circ J=0$, i.e. the Hessian $\Hess H$ of $H$ anti-commutes with $J$.
		\item\label{4} $J\circ \Hess \mu=\Hess \mu \circ J$, i.e. the Hessian $\Hess \mu$ commutes with $J$. 
		\item\label{5} $\langle \nabla\mu, \nabla H\rangle_{TM}=\langle J\nabla\mu, \nabla H\rangle_{TM}=0$.
		\item \label{6} $\langle \nabla \mu, \tilde{\xi}\rangle_{TM}=0$ for any $\xi\in\g$.
	\end{enumerate}
\end{lemma}
\begin{proof} The first identity (\ref{1}) is the Cauchy-Riemann equation \eqref{CRequation}. Since $M$ is K\"{a}hler, the complex structure $J$ is parallel, i.e. $
	\nabla J=0$, so (\ref{2}) follows from (\ref{1}) by taking the covariant derivative.
	
	Both $\Hess L$ and $\Hess H$ are symmetric operators with respect to the metric $g$, so by (\ref{2}), we have
	\[
	J\circ \Hess H=(J\circ \Hess H)^T= (\Hess H)^T\circ J^T=-\Hess H\circ J. 
	\]
This proves \eqref{3}.	Since the metric $g$ is $G$-invariant, for any $\xi\in \g$, the Lie derivative $\CL_{\tilde{\xi}} g$ is zero. This implies that for any vector fields 
		$v,w$ on $M$,
		\begin{equation}\label{EB.2}
\langle \nabla_v \tilde{\xi}, w\rangle+\langle \nabla_w\tilde{\xi}, v\rangle=0.
		\end{equation}
	Using the defining equation (\ref{convention1}), we conclude that $\langle J\Hess \mu, \xi\rangle_\g$ is an anti-symmetric operator, so $\Hess \mu$ commutes with $J$. This proves (\ref{4})
	
	Finally, since $H$ is $G$-invariant, $\langle \tilde{\xi}, \nabla H\rangle=0$. By (\ref{convention1}), $\langle J\nabla \mu, \nabla H\rangle_{TM}=0$. The other identity in (\ref{5}) follows from the $G$-invariance of $L$ and the first identity (\ref{1}).
	
	(\ref{6}) follows from the fact that $\mu: M\to \g$ is $G$-invariant. 
\end{proof}

A smooth connection $A=d+a$ in the trivial principal bundle $\HH^2_+\times G\to \HH^2_+$ allows us to take covariant derivatives of a map $P: \HH^2_+\to M$. It is also important to know the covariant derivative of a vector field $v$ along $P$, i.e. a smooth map $v: \HH^2_+\to TM$ with $v(z)\in T_{P(z)}M$. Recall that for any tangent vector $(z,V)\in T\HH^2_+$, the covariant derivative $\nabla^A_VP\in T_{P(z)}M$ is defined by the property:
\begin{equation}\label{covariant2}
\nabla^A_VP=\frac{d}{dt} \rho(e^{t\cdot a(V)}) P\big(\gamma(t)\big)\bigg|_{t=0}
\end{equation}
where $\gamma: [0,1]\to \HH^2_+$ is a path with $\gamma(0)=z$ and $\gamma'(0)=V$. The action of $G$ extends to the tangent bundle $TM$ of $M$:
\[
\rho_{TM}(g)(p,v)\mapsto \big(\rho(g)p, \rho(g)_*v\big). 
\]
If $v$ is a vector field along $P\circ \gamma(t)$, it is reasonable to define its derivative $\nabla^A_Vv\in T_{P(z)}M$ as:
\begin{equation}\label{covariant3}
\nabla^A_Vv\colonequals \nabla_{\nabla^A_VP}\rho(e^{t\cdot a(V)})_*v. 
\end{equation}
\begin{remark} If the connection 1-form $a\equiv 0$ is trivial, the formula \eqref{covariant3} reduces to $\nabla_{P_*V}v$ and should be understood as follows: pull back the bundle $TM$ along with the Levi-Civita connection via $P\circ \gamma$ and differentiate the pullback section $(P\circ\gamma)^*v$ at $t=0$. In other words, $\nabla_{P_*V}v\colonequals \nabla_V (P^*v)$, which might be non-zero even if $P\circ \gamma(t)\equiv P(z)$ is a constant path. In terms of local coordinates $\{y^i\}_{1\leq i\leq n}$, 
	\begin{equation}
	\nabla_{P_*V} v\colonequals \pt v^i(t)\partial_i+\frac{\partial y^j\big(P\circ\gamma(t)\big)}{\partial t} v^i(t)\Gamma_{ij}^k\partial_k,
	\end{equation}
	where $\Gamma_{ij}^k$ are Christoffel symbols and $v=v^i\partial_i: [0,1]_t\to TM$ is the smooth section lying over $P\circ \gamma$. If $a$ is non-trivial,  first apply the transformation $\{\rho(e^{t\cdot a(V)})\}_{t\in [0,1]}$ to kill $a$ at $t=0$ and then compute the derivative. The formula  \eqref{covariant3} is the practical formula only when $\nabla^A_V P\neq 0$. 
\end{remark}

It is enlightening to find a concrete formula of $\nabla^A_Vv$ without using the group action. By the defining property of the moment map (\ref{convention1}), we have 
\begin{equation}\label{covariant1}
\nabla^A_VP=V\cdot P+\tilde{a}(V)=P_*(V)+J\langle \nabla\mu, a(V)\rangle_\g. 
\end{equation}
\begin{lemma}\label{covariant4} The covariant derivative of a vector field $v$ equals:
	\[
	\nabla^A_Vv=\nabla_{P_*V}v+ J\langle \Hess \mu(v), a(V)\rangle_\g. 
	\]
	where $\nabla_{P_*V}v$ denotes the covariant derivative with respect to the Levi-Civita connection. 
\end{lemma}

\begin{remark}
	The correction term $J\langle \Hess \mu(v), a(V)\rangle$ reflects the dependence on the connection 1-form $a$. It is linear in $a$, $v$ and $V$ as expected. 
\end{remark}

\begin{proof}
	The formal proof is to linearize the equation (\ref{covariant1}) along the tangent vector $v(z)\in T_{P(z)}M$. Let us make this intuition precise. Consider a variation of $P\circ \gamma$ along the vector field $v$:
	\[
	Q(r, t,s)=\rho\big(e^{r\cdot a(V)}\big)\exp_{P\circ \gamma(t)}\big(sv(t)\big).
	\]
	When $r\equiv 0$, $Q(0,t,s)$ is a variation of the path $P\circ\gamma(t)$. Indeed, $Q(0,t,0)\equiv P\circ \gamma(t)$. When $s\equiv 0$, the covariant derivative of $P$ is defined as (comparing (\ref{covariant2})):
	\begin{align*}
	\frac{d}{dt} Q(t,t,0)\bigg|_{t=0}=\nabla^A_VP&=\frac{d}{dt} Q(0,t,0)+\frac{d}{dt} Q(t,0,0)\bigg|_{t=0}\\
	&=P_*(V)+J\langle \nabla\mu, a(V)\rangle_\g. 
	\end{align*}
	
	Let $U_1=\frac{d}{dt} Q(t,t,s)$ and $U_2=\frac{d}{ds} Q(t,t,s)$. Then $U_1=U_3+U_4$ with
	\[
	U_3=(\partial_2Q)(t,t,s), U_4=(\partial_1Q)(t,t,s).
	\] 
	$U_4=J\langle \nabla\mu, a(V)\rangle_\g$. When $t=s=0$, $U_3=P_*V$.  Moreover, $[U_2,U_4]=0$.
	By (\ref{covariant3}), we have 
	\begin{align*}
	\nabla^A_Vv=\nabla_{U_1}U_2\bigg |_{t=s=0}&=\nabla_{U_3}U_2+\nabla_{U_4}U_2\bigg |_{t=s=0}=\nabla_{P_*V}v+\nabla_{U_2}U_4\bigg |_{t=s=0}\\
	&=\nabla_{P_*V}v+ J\langle \Hess \mu(v), a(V)\rangle_\g. \qedhere
	\end{align*}
\end{proof}

The next lemma concerns the curvature tensor of $\nabla^A$. Since we are merely interested in the manifold $\HH^2_+$, it suffices to work with vector fields $\partial_t$ and $\partial_s$.
\begin{lemma}\label{identities2} Write $T=\nabla^A_{\partial_t} P$ and $S=\nabla^A_{\partial_s}P$ for short. The following properties hold for any configuration $(A,P)$ and any vector field $v$ along $P$:
	\begin{enumerate}
		\item The connection $\nabla^A$ is equivariant under the gauge transformation $u(A,P)=(A-u^{-1}du, u\cdot P)$, i.e.
		\[
		u_*(\nabla^A_VP)=\nabla^{u(A)}_V u(P),\ u_*(\nabla^A_Vv)=\nabla^{u(A)}_Vu_*v,
		\]
		where $u_*v$ is the vector field along $u(P)$. 
		\item If $v$ is the pull-back of a $G$-invariant vector field on $M$, then $\nabla^A_{\partial_s} v=\nabla_Sv$. 
		\item $(\nabla^A_{\partial_t}\nabla^A_{\partial_s}-\nabla^A_{\partial_s}\nabla^A_{\partial_t})P=J\langle \nabla\mu, F_A(\partial_t,\partial_s)\rangle_\g=-\tilde{F}$ where $F=-*_2F_A$.
		\item For any vector fields $v,w$ on $\im P\subset M$, 
		\[
		\partial_s\langle v,w\rangle=\langle \nabla^A_{\partial_s} v, w\rangle +\langle v, \nabla^A_{\partial_s} w\rangle,
		\]
		i.e. the connection $\nabla^A$ is unitary. 
		\item\label{Id5} The curvature tensor of $\nabla^A$ is given by $$(\nabla^A_{\partial_t}\nabla^A_{\partial_s}-\nabla^A_{\partial_s}\nabla^A_{\partial_t})v=R_M(T,S)v+J\langle \Hess \mu(v), F_A(\partial_t,\partial_s)\rangle_\g,$$
		where $R_M(\cdot,\cdot)\cdot $ denotes the Riemannian curvature tensor on $M$.  
	\end{enumerate}
\end{lemma}
\begin{proof}
	
	The property (1) follows from the defining property (\ref{covariant2}) and  (\ref{covariant3}) of $\nabla_A$. 
	
	If $v$ is induced from a $G$-variant vector field on $M$, then for any $g\in G$, $\rho(g)_*v=v$. By (\ref{covariant3}),
	\[
	\nabla^A_Vv\colonequals \nabla_{\nabla^A_VP}\rho(e^{t\cdot a(V)})_*v= \nabla_{\nabla^A_VP}v. 
	\]
	
	This proves (2). For (3), if $F_A\equiv 0$ near a point $x\in \HH^2_+$, then we apply a gauge transformation $u$ so that the connection 1-form $a\equiv 0$ near $x$. Thus,
	\[
	u_*(\nabla^A_{\partial_t}\nabla^A_{\partial_s}-\nabla^A_{\partial_s}\nabla^A_{\partial_t})P=\nabla_{\partial_t}\partial_s u( P)-\nabla_{\partial_s}\partial_t u (P)=0.
	\]
	This shows the commutator is at least proportional to $*F_A$. To work out the general case, we apply (\ref{covariant1}) and Lemma \ref{covariant4}:
	\begin{align*}
	\nabla^A_{\partial_t}\nabla^A_{\partial_s}P&=\nabla_{\partial_t P} S+J\langle \Hess \mu(S),a(\partial_t)\rangle_\g \\
	&=\nabla_{\partial_t P}\partial_sP+\nabla_{\partial_tP} J\langle \nabla\mu, a(\partial_s)\rangle _\g+J\langle \Hess \mu(S),a(\partial_t)\rangle_\g \\
	&=\nabla_{\partial_t P}\partial_sP+J\langle \nabla\mu, \partial_t a(\partial_s)\rangle_\g-\langle \Hess \mu(\langle \nabla\mu, a(\partial_s)\rangle_\g), a(\partial_t)\rangle_\g\\
	&\qquad +J\langle \Hess \mu(\partial_t P), a(\partial_s)\rangle_\g+J\langle \Hess \mu(\partial_s P),a(\partial_t)\rangle_\g.
	\end{align*}
	
	At this point, we need the following fact. For any $\xi,\eta\in \g$, 
	\[
	\big\langle \Hess \mu\big(\langle \nabla\mu, \xi\rangle_\g \big), \eta\big\rangle_\g= \big\langle \Hess \mu\big(\langle \nabla\mu, \eta\rangle_\g \big), \xi \big\rangle_\g.
	\]
This follows from the relation $\nabla_{\tilde{\xi}}\tilde{\eta}-\nabla_{\tilde{\xi}}\tilde{\eta}=[\tilde{\xi},\tilde{\eta}]=\widetilde{[\xi,\eta]}=0$ (since $G$ is abelian). Now compare the formula of $	\nabla^A_{\partial_t}\nabla^A_{\partial_s}P$ with that of $	\nabla^A_{\partial_s}\nabla^A_{\partial_t}P$. This proves (3).
	
	As for (4), we use the gauge invariance of $\nabla^A$. Alternatively, one uses Lemma \ref{covariant4} and the fact that 
	\begin{equation}\label{J-Hess}
	\langle J\Hess \mu(v),w\rangle_{TM}+\langle J\Hess \mu(w),v\rangle_{TM}=0,
	\end{equation}
which follows from \eqref{EB.2}.
	
	The expression of the curvature tensor $(\ref{Id5})$ requires some work. Again, if $F_A\equiv 0$, we use the gauge invariance of $\nabla_A$, and $(\ref{Id5})$ follows from the definition of $R_M$. The actually proof is not very tidy. We follow the strategy of (3):
	\begin{align*}
	\nabla^A_{\partial_t}\nabla^A_{\partial_s}v&=\nabla_{\partial_t P} \nabla^A_{\partial_s}v+J\langle \Hess \mu(\nabla^A_{\partial_s}v), a(\partial_t)\rangle_\g\\
	&=\nabla_{\partial_t P} \nabla_{\partial_s P}v+J\nabla_{\partial_t P}\langle \Hess \mu(v), a(\partial_s)\rangle_\g+J\langle \Hess \mu(\nabla^A_{\partial_s}v), a(\partial_t)\rangle_\g\\
	&=\nabla_{\partial_t P} \nabla_{\partial_s P}v+J\langle \Hess \mu(v), \partial_t a(\partial_s)\rangle_\g+J\langle (\nabla_{\partial_t P}\Hess \mu)(v), a(\partial_s)\rangle_\g\\
	&\qquad +J\langle \Hess \mu(\nabla_{\partial_tP}v), a(\partial_s)\rangle_\g+J\langle \Hess \mu(\nabla_{\partial_sP}v), a(\partial_t)\rangle_\g\\
	&\qquad + J\big\langle \Hess \mu\big(\langle J\Hess \mu(v),a(\partial_s)\rangle\big), a(\partial_t)\big\rangle_\g.
	\end{align*} 
	There are six terms in the expression. The fourth and fifth ones will also occur in that of $\nabla^A_{\partial_s}\nabla^A_{\partial_t}v$. Hence, $\nabla^A_{\partial_t}\nabla^A_{\partial_s}v-\nabla^A_{\partial_s}\nabla^A_{\partial_t}v$ can be written as $N_1+N_2+N_3+N_6$ with 
	\begin{align*}
N_1&\colonequals\nabla_{\partial_t P} \nabla_{\partial_s P}v-\nabla_{\partial_s P} \nabla_{\partial_t P}v=R_M(\partial_tP,\partial_s P)v,\\
N_2&\colonequals J\langle \Hess \mu(v), \partial_t a(\partial_s)\rangle_\g-J\langle \Hess \mu(v), \partial_s a(\partial_t)\rangle_\g=	J\langle \Hess \mu(v), F_A(\partial_t,\partial_s)\rangle_\g,\\
	N_3&\colonequals J\langle (\nabla_{\partial_t P}\Hess \mu)(v), a(\partial_s)\rangle_\g-J\langle (\nabla_{\partial_s P}\Hess \mu)(v), a(\partial_t)\rangle_\g,\\
	N_6&\colonequals J\big\langle \Hess \mu\big(\langle J\Hess \mu(v),a(\partial_s)\rangle_\g\big), a(\partial_t)\big\rangle_\g-J\big\langle \Hess \mu\big(\langle J\Hess \mu(v),a(\partial_t)\rangle_\g\big), a(\partial_s)\big\rangle_\g. 
	\end{align*}
On the other hand, we have 
	\[
	R_M(T,S)-R_M(\partial_tP,\partial_s P)=R_M\big(\partial_tP,\tilde{a}(\partial_s)\big)+R_M\big(\tilde{a}(\partial_t), \partial_sP\big)+R_M\big(\tilde{a}(\partial_t),\tilde{a}(\partial_s)\big). 
	\]
	To complete the proof of \eqref{Id5}, it remains to verify that
\begin{align*}
N_3&= R_M\big(\partial_tP,\tilde{a}(\partial_s))v-R_M( \partial_sP,\tilde{a}(\partial_t)\big)v,\\
N_6&=R_M\big(\tilde{a}(\partial_t),\tilde{a}(\partial_s)\big)v,
\end{align*}
which is the content of Lemma \ref{BL.6} and \ref{BL.7}.
\end{proof}
\begin{lemma}\label{BL.6}For any $\xi\in \g$ and vector fields $v,w$ on $M$, we have 
	$$R_M(w,\tilde{\xi})v=\langle \nabla_w (J \Hess \mu)(v),\xi\rangle_\g.$$
\end{lemma}
\begin{proof}[Proof of Lemma] Differentiating (\ref{J-Hess}) along a vector field $u$ on $M$ yields
	\[
	\langle \nabla_{u}(J\Hess \mu)(v),w\rangle_{TM}+\langle \nabla_{u}(J\Hess \mu)(w),v\rangle_{TM}=0.
	\]
	
	The key observation is that for any vector fields $u,v,w$, we have 
	\[
	\langle R_M(u,v)w,J\nabla\mu\rangle_{TM}=-\langle \nabla_u(J\Hess \mu)(v),w\rangle_{TM}-\langle \nabla_v(J\Hess \mu)(w),u\rangle_{TM}. 
	\]
	Indeed, we use the symmetry of curvature tensor to compute:
	\begin{align*}
	\langle R_M(u,v)w,J\nabla\mu\rangle_{TM}&=-\langle R_M(u,v)J\nabla\mu,w\rangle_{TM}\\
	&=-\langle \nabla_u(J\Hess \mu)(v),w\rangle_{TM}+\langle \nabla_v(J\Hess \mu)(u),w\rangle_{TM}\\
	&=-\langle \nabla_u(J\Hess \mu)(v),w\rangle_{TM}-\langle \nabla_v(J\Hess \mu)(w),u\rangle_{TM}.
	\end{align*}
	This expression is unchanged if we permute $(u,v,w)$. Using the symmetry
	\[
	R_M(u,v)w+R_M(v,w)u+R_M(w,u)v=0
	\]
	from Riemannian geometry, we conclude that 
	\[
	\langle \nabla_u(J\Hess \mu)(v),w\rangle_{TM}+\langle \nabla_v(J\Hess \mu)(w),u\rangle_{TM}+\langle \nabla_w(J\Hess \mu)(u),v\rangle_{TM}=0.
	\]
	In particular, this implies $\langle R_M(u,v)w,J\nabla\mu\rangle_{TM}=\langle \nabla_w(J\Hess \mu)(u),v\rangle_{TM}.$ Finally, note that $\tilde{\xi}=\langle \nabla\mu, \xi\rangle_\g$ and 
	\[
	\langle R_M(w,\tilde{\xi})u,v\rangle =\langle R_M(u,v)w,\tilde{\xi}\rangle=\langle \nabla_w (J \Hess \mu)(u),v\otimes \xi\rangle_{TM\otimes \g}.\qedhere
	\]
\end{proof}
\begin{lemma}\label{BL.7}
	For any $\xi,\eta\in\g$ and any vector field $v$ on $M$, 
	\[
	R_M(\tilde{\xi},\tilde{\eta})v=J\big\langle \Hess \mu\big(\langle J\Hess \mu(v),\eta\rangle_\g\big), \xi\big\rangle_\g-J\big\langle \Hess \mu\big(\langle J\Hess \mu(v),\xi\rangle_\g\big), \eta\big\rangle_\g.
	\]
\end{lemma}
\begin{proof}[Proof of Lemma] This identity is equivalent to that
	\begin{equation}\label{B.6}
	\langle R_M(\tX,\tY)v,w\rangle=-\langle \nabla_v\tY,\nabla_w\tX\rangle+\langle \nabla_v\tX,\nabla_w\tY\rangle,
	\end{equation}
	for any vector field $w\in \Gamma(M, TM)$. We first rewrite the right hand side of (\ref{B.6}):
	\begin{align*}
	I&\colonequals-\langle \nabla_v\tY,\nabla_w\tX\rangle+\langle \nabla_v\tX,\nabla_w\tY\rangle\\
	&=-v\cdot \langle \tY, \nabla_w\tX\rangle+\langle \tY,\nabla_v\nabla_w\tX\rangle +w\cdot \langle \tY,\nabla_v\tX\rangle-\langle \tY, \nabla_w\nabla_v\tX\rangle.
	\end{align*}
	By \eqref{EB.2} and the relation $\nabla_{\tX}\tY=\nabla_{\tY}\tX$, we have 
	\[
	\langle \nabla_v\tX,\tY\rangle=-\langle \nabla_{\tY}\tX,v\rangle=-\langle \nabla_{\tX}\tY,v\rangle=\langle \nabla_v\tY,\tX\rangle. 
	\]
	Therefore, $v\cdot\langle \tY,\tX\rangle=2\langle\nabla_v \tX,\tY\rangle$. Moreover,
	\begin{align*}
	-v\cdot \langle \tY, \nabla_w\tX\rangle+w\cdot \langle \tY,\nabla_v\tX\rangle&=-\half v\cdot w\langle \tY,\tX\rangle+\half w\cdot v\langle \tY,\tX\rangle,\\
	&=-\half [v,w]\langle \tY,\tX\rangle =-\langle \tY,\nabla_{[v,w]} \tX\rangle. 
	\end{align*}
	Finally, we conclude that $I=\langle R_M(v,w)\tX,\tY\rangle=\langle R_M(\tX,\tY)v,w\rangle$.
\end{proof}
\subsection{A Bochner-Type Formula} By abuse of notations, we denote the trivial bundle $\g\times M\to M$ also by $\g$. Define a bundle map $D: TM\to \g\oplus\g\oplus TM$ as follows: for any tangent vector $v\in T_xM$, define its image as:
\[
D_x(v)\colonequals \big( \langle \nabla \mu, v\rangle_{TM}, \langle \nabla\mu, Jv\rangle_{TM},\Hess_x L(v)\big)\in \g\oplus \g\oplus T_xM. 
\]

Bochner's formula \cite[P. 334]{GTM171} was originally stated for a harmonic function $u: N\to \R$ on a Riemannian manifold $N$. It computes the Laplacian of $|\nabla u|^2$:
\[
0=\half \Delta_N |\nabla u|^2+|\Hess u|^2+\Ric(\nabla u,\nabla u). 
\]

We provide a formula in the same spirit for a solution $(A,P)$ to the gauged Witten equation \eqref{generalized-vortex-equation} on $\HH^2_+$, with the bundle map $D$ playing the role of $\Ric(\cdot, \cdot )$. The Laplacian operator $\Delta_N$ or $\Delta_{\HH^2_+}$ is always assumed to have a non-negative spectrum. In particular, 
\[
\Delta_{\HH^2_+}=-(\partial_t^2+\partial_s^2). 
\]

\begin{theorem}\label{BochnerFormula} Write $T=\Dt P$, $S=\Ds P$ and $F=-*_2 F_A$ for short. For a solution $(A,P)$ to the gauged Witten equation $(\ref{generalized-vortex-equation})$ on $\HH^2_+$, we have identities:
	\begin{enumerate}
		\item  The Laplacian $\half (-\Delta_{\HH^2_+})|T|^2=\half (\partial_s^2+\partial_t^2)|T|^2$ of $|T|^2$ is equal to
		\begin{align*}
		&|\Ds T|^2+|\Dt T|^2+|D_P(T)|^2+\langle R_M(S,T)S,T\rangle\\
		&\qquad+\langle(\nabla_T\Hess H)(\nabla H), T\rangle+\langle \Hess \mu(2JS-T), T\otimes F\rangle_{TM\otimes\g}.
		\end{align*}
		\item  Similarly, $\half (-\Delta_{\HH^2_+})|S|^2=\half (\partial_s^2+\partial_t^2)|S|^2$ is equal to
		\begin{align*}
		&|\Dt S|^2+|\Ds S|^2+|D_P(S)|^2+\langle R_M(T,S)T,S\rangle\\
		&\qquad+\langle(\nabla_S\Hess H)(\nabla H), S\rangle+\langle \Hess \mu(-2JT-S), S\otimes F\rangle_{TM\otimes\g}.
		\end{align*}
		\item The Laplacian $\half (-\Delta_{\HH^2_+})|F|_\g^2$ of $|F|_\g^2$ is equal to
		\begin{align*}
		&|\partial_sF|^2_\g+|\partial_tF|^2_\g+|\langle \nabla\mu, F\rangle_\g |^2+2\langle \Hess \mu(JS),T\otimes F\rangle_{TM\otimes\g}. 
		\end{align*}
	\end{enumerate}
\end{theorem}

Define $\nabla^AP\colonequals dt\otimes T+ds\otimes S$ and $\Hess_A P\colonequals\sum_{1\leq i,j\leq 2}\omega_i\otimes \omega_j\otimes \nabla^A_{e_i}\nabla^A_{e_j}P$ with $(e_1,e_2)=(\partial_t,\partial_s)$ and $(\omega_1,\omega_2)=(dt,ds)$. 

\begin{corollary} \label{B.8}We have the following identity for $\Delta_{\HH^2_+}\big(|\nabla^A P|^2+|F|^2_\g\big)$:
	\begin{align*}
	0&=\half \Delta_{\HH^2_+}\big(|\nabla^A P|^2+|F|^2_\g\big)+|\Hess_A P|^2+|\nabla F|^2_\g+|D_P(\nabla^A P)|^2+|\langle \nabla\mu, F\rangle_\g |^2\\
	&\qquad+ 2\langle R(S,T)S,T\rangle+ \langle(\nabla_T\Hess H)(\nabla H), T\rangle+\langle(\nabla_S\Hess H)(\nabla H), S\rangle\\
	&\qquad+ 6\langle \Hess \mu(JS),T\otimes F\rangle_{TM\otimes\g}-\langle \Hess \mu(T),T\otimes F\rangle_{TM\otimes\g}-\langle \Hess \mu(S),S\otimes F\rangle_{TM\otimes\g}.
	\end{align*}
	
\end{corollary}
\begin{proof}[Proof of Corollary \ref{B.8}] It follows by adding up the three identities from Theorem \ref{BochnerFormula}. 
\end{proof}

Define the connection Laplacian of $P$ as:
\[
\Delta_A P\colonequals -\nabla^A_{\partial_t}\nabla^A_{\partial_t}P-\nabla^A_{\partial_s}\nabla^A_{\partial_s}P=-(\nabla^A_{\partial_t}T+\nabla^A_{\partial_s}S).
\]

To prove Theorem \ref{BochnerFormula}, we start with a useful formula of  $\Delta_A P$.
\begin{lemma}\label{connectionlaplacian}
	$-\Delta_A P=J\tilde{F}+\Hess H(\nabla H)$. 
\end{lemma}
\begin{proof}
	Apply the operator $-\nabla^A_{\partial_t}+J\nabla^A_{\partial_s}$ to the equation (\ref{firsteq}). Using the fact that $\Hess H$ anti-commutes with $J$ (Lemma \ref{identities} (3)) and Lemma \ref{identities2} (2)(3), we deduce that
	\begin{align*}
	0&=(-\nabla^A_{\partial_t}+J\nabla^A_{\partial_s})(T+JS+\nabla H)\\
	&=-(\nabla^A_{\partial_t}T+\nabla^A_{\partial_s}S)-J(\nabla^A_{\partial_t}S-\nabla^A_{\partial_s}T)-\Hess H(T+JS)\\
	&=\Delta_A P+J(\tilde{F})+\Hess H(\nabla H).\qedhere
	\end{align*}
	\begin{remark}\label{B.10} It is enlightening to work out Lemma \ref{connectionlaplacian} and Theorem \ref{BochnerFormula} concretely in the special case of Example \ref{example1}, with $M=\C, W\equiv 0$ and $\mu=\frac{i}{2}|x|^2$. The metric on $\C$ is flat, and $T_xM$ is identified with $\C$ for each $x\in M$. Hence, 
		\[
		\nabla\mu(x)=x\otimes i \text{ and } \Hess \mu=1_\C\otimes i. 
		\]
		Take $\delta=\frac{i}{2}\in \g=i\R$. In this case, the equation (\ref{generalized-vortex-equation}) recovers the vortex equation on $\HH^2_+\subset\C$: 
		\begin{equation*}
		\left\{\begin{array}{rl}
		\bpartial_A P&=0\\
		-*F_A+\mu&=\frac{i}{2}
		\end{array}
		\right.
		\end{equation*}
		where $A$ is a unitary connection and $P: \HH^2_+\to \C$ is a smooth function. Then Lemma \ref{connectionlaplacian} says
		\[
		\Delta_A P=\langle F, i\rangle_\g P,
		\]
		which follows from the Weitzenb\"{o}ck formula. Since $\nabla H\equiv 0$, $T=-JS$. The first two identities in Theorem \ref{BochnerFormula} yield:
		\[
		0=\half \Delta_{\HH^2_+}|\nabla_A P|^2+|\Hess_A P|^2+|D(\nabla_A P)|^2-3\langle F, i\rangle_\g  |\nabla_A P|^2,
		\]
		where $D(v)=\big(\langle \nabla\mu ,v\rangle, \langle \nabla\mu ,Jv\rangle\big)$ and so $|D(v)|^2_\g=|P|^2|v|^2$. In the meanwhile, we have 
		\[
		0=\half \Delta_{\HH^2_+}|F|^2+|\nabla F|^2_\g+|P|^2|F|^2_\g-|\nabla_AP|^2\langle F,i\rangle_\g. 
		\]
		
		These formulae were first proved in \cite[Proposition 6.1]{JT}.
	\end{remark}
	
\end{proof}
\begin{proof}[Proof of Theorem \ref{BochnerFormula}] Let us start with $\partial_s^2|T|^2$. By Lemma \ref{identities2} (4), 
	\begin{align*}
	\half \partial_s^2|T|^2&=\partial_s \langle \nabla^A_{\partial_s}T,T\rangle=| \nabla^A_{\partial_s}T|^2+\langle T,\nabla^A_{\partial_s}\nabla^A_{\partial_s}T\rangle.
	\end{align*}
	By Lemma \ref{identities2} (3)(5), we have 
	\begin{align*}
	\nabla^A_{\partial_s}\nabla^A_{\partial_s}T&=\nabla^A_{\partial_s}(\nabla^A_{\partial_t}S+\tilde{F})\\
	&=\nabla^A_{\partial_t}\nabla^A_{\partial_s}S+R_M(S,T)S+\langle J\Hess \mu(S), F\rangle_\g+\nabla^A_{\partial_s}\tilde{F},
	\end{align*}
which implies that 
	\begin{align}\label{EB.9}
	\half (\partial_t^2+\partial_s^2)|T|^2&=|\Ds T|^2+|\Dt T|^2+\langle \Dt(-\Delta_AP), T\rangle+\langle R_M(S,T)S,T\rangle\\
	&\qquad+ \langle J\Hess \mu(S), T\otimes F\rangle_{TM\otimes\g}+\langle \nabla^A_{\partial_s}\tilde{F},T\rangle.\nonumber 
	\end{align}
	
	By using the equation (\ref{secondeq}), we extract some positivity from the last term:
	\begin{align*}
	\langle \nabla^A_{\partial_s}\tilde{F},T\rangle&= \langle \Ds \langle J\nabla \mu, F\rangle_\g,T\rangle \\
	&=\langle J\Hess \mu(S), T\otimes F\rangle_{TM\otimes\g}+\big\langle\langle J\nabla \mu, \partial_sF\rangle_{\g},T\big\rangle. \\
	&=\langle J\Hess \mu(S), T\otimes F \rangle_{TM\otimes\g}+\langle \langle \nabla\mu, JT\rangle_{TM}, \langle \nabla\mu, S\rangle_{TM}\rangle_\g.
	\end{align*}
	
	By Lemma \ref{identities} (5), $\langle \nabla\mu,J\nabla H\rangle_{TM}=0$. Using (\ref{firsteq}), we have 
	$$ \langle \nabla\mu, JT\rangle_{TM}=\langle \nabla\mu, S-J\nabla H\rangle_{TM} =\langle \nabla\mu, S\rangle_{TM} .$$
	Hence, $\langle \nabla^A_{\partial_s}\tilde{F},T\rangle_{TM}=\langle J\Hess \mu(S), T\otimes F \rangle_{TM\otimes\g}+|\langle \nabla\mu, JT\rangle_{TM} |_\g^2$.
	
	Now we deal with the term involving $\Delta_A P$ in \eqref{EB.9}, using Lemma \ref{connectionlaplacian}. We exploit the fact that $\Hess H(\nabla H)$ is a $G$-invariant vector field on $M$ and Lemma \ref{identities2} (2): 
	\begin{align*}
	\Dt(-\Delta_AP)&=\Dt\big(J\tilde{F}+\Hess H(\nabla H)\big)\\
	&=-\Dt\langle \nabla\mu, F\rangle_\g +(\nabla_T\Hess H)(\nabla H)+\Hess H^2(T),\\
	&=-\langle \Hess \mu(T), F\rangle_\g -\langle \nabla\mu, \partial_t F\rangle_\g+(\nabla_T\Hess H)(\nabla H)+\Hess H^2(T).
	\end{align*}
	Note that by \eqref{secondeq}, $-\big\langle \langle \nabla\mu, \partial_t F\rangle_\g, T\big\rangle =\big\langle \langle \nabla\mu, \langle \nabla\mu, T\rangle_{TM}\rangle_{\g}, T\big\rangle =|\langle \nabla\mu,T\rangle_{TM}|_\g^2$. 
	
	Combining all these together, we obtain that
	\begin{align*}
	\half (\partial_t^2+\partial_s^2)|T|^2&=|\Ds T|^2+|\Dt T|^2+\langle R_M(S,T)S,T\rangle\\
	&\qquad+ |\langle \nabla\mu,T\rangle_{TM}|_\g^2+|\langle \nabla\mu,JT\rangle_{TM}|_\g^2+|\Hess H(T)|^2\\
	&\qquad+\langle(\nabla_T\Hess H)(\nabla H), T\rangle+\langle \Hess \mu(2JS-T),  T\otimes\g\rangle_{TM\otimes\g}.
	\end{align*}
	The formula of $\half (\partial_t^2+\partial_s^2)|S|^2$ is proved in a similar may. Finally, we deal with the Laplacian of $|F|^2_\g$. By (\ref{secondeq}), we have 
	\begin{align*}
	\half \partial_s^2|F|^2_\g&=|\partial_sF|^2_\g+\langle F,\partial_s^2(-\mu)\rangle_\g,\\
	&=|\partial_sF|^2_\g-\langle\Hess \mu(S),S\otimes F\rangle_{TM\otimes\g}-\langle \nabla\mu, \Ds S\otimes F\rangle_{TM\otimes\g}. 
	\end{align*}
	By Lemma \ref{connectionlaplacian}, we have 
	\begin{align*}
	\langle -\nabla\mu, (\Dt T+\Ds S)\otimes F\rangle_{TM\otimes\g} &= \langle-\nabla\mu, J\tilde{F}\otimes F\rangle_{TM\otimes\g}+\langle -\nabla\mu, \Hess H(\nabla H)\otimes F\rangle_{TM\otimes\g} \\
	&=|\langle \mu,F\rangle_\g|^2+\langle \Hess \mu(\nabla H), \nabla H\otimes F\rangle_{TM\otimes\g}.
	\end{align*}
	At the last step, we used the identity:
	\begin{equation}\label{123}
	\langle \Hess \mu (X),\nabla H\rangle_{TM}+\langle \nabla\mu,\Hess H(X)\rangle_{TM}=0
	\end{equation}
	with $X=\nabla H$. Note that by Lemma \ref{identities} (5), $\langle \nabla\mu, \nabla H\rangle_{TM}\equiv 0$. Expanding the expression $X\cdot \langle \nabla\mu, \nabla H\rangle_{TM}\equiv 0$ yields (\ref{123}). This completes the proof of Theorem \ref{BochnerFormula}.
\end{proof}

%% file: higher_1.0.tex
\section{Surfaces with Cylindrical Ends}\label{A:higher}

In this appendix, we generalize Proposition \ref{P6.2} for oriented surfaces with cylindrical ends. To start, we fix a conformal structure on a closed oriented surface $\Sigma_g$ of genus $g$ and consider the punctured surface 
 \[
 \Sigma_{g,n}\colonequals \Sigma_g\setminus\{p_1,\cdots,p_n\},
 \]
 where $p_j,\ 1\leq j\leq n$ are distinct points on $\Sigma_g$ and let $\D\colonequals \sum p_j$ be the associated positive divisor. Later, we will specify the metric on $\Sigma_{g,n}$ compatible with this complex structure. We require that 
 \[
 \chi(\Sigma_{g,n})=2-2g-n\leq 0.
 \]
 Thus one may take $g=0$ if $n\geq 2$. The harmonic 1-form $\lambda$ is allowed to have poles at each puncture $p_j, 1\leq j\leq n$, but further conditions are required:
 \begin{assumpt}\label{AC.5}  Take $\lambda\in \Omega^1(\Sigma_{g,n}, i\R)$ and assume that its $(1,0)$-part $\lambda^{1,0}\in H^0(\Sigma_g, \K(\D))$ extends to a holomorphic section of the canonical bundle $\K\colonequals \Lambda^{1,0}\Sigma_g \to \Sigma_g$ twisted by $\D$. We require that
 	\begin{itemize}
 		 	\item  the residue of $\lambda^{1,0}$ at each puncture $p_j\in \D$ is non-zero, and
 	\item  $\lambda^{1,0}$ has $2g-2+n$ simple zeros on $\Sigma_{g,n}$.
 	\end{itemize}
For such a 1-form $\lambda$ to exist, we must have $n\geq 2$. 
 \end{assumpt}

In particular, for some holomorphic coordinate $z_j$ near $p_j\in \Sigma_g$ and $a_j\neq 0\in \C$, 
\[
\lambda^{1,0}=-\frac{a_j dz_j}{z_j},
\]
with $z_j(p_j)=0$. By rescaling the coordinate function $z_j$ if necessary, we assume that the neighborhood $z_j^{-1}(B(0, 1)), 1\leq j\leq n$ are pairwise disjoint in $\Sigma_g$. In terms of the polar coordinate $z_j=\exp(-(s_j+i\theta_j))$ with  $s_i\in [0,\infty)$ and $\theta_j\in \R/2\pi\Z$, we have 
\[
\lambda^{1,0}=a_j (ds_j+id\theta_j). 
\]
 Pick a Riemannian metric of $\Sigma_{g,n}$ which restricts to the product metric on each cylindrical end  
\[
U_j\colonequals [0,\infty)_{s_j}\times (\R/2\pi \Z)_{\theta_j}\subset \Sigma_{g,n}, 1\leq j\leq n,
\]
and is compatible with the given complex structure. Then on each $U_j$, the 1-form $\lambda=2i(\im a_j d s_j+\re a_j d\theta_j)$  is covariantly constant.

\medskip

Because the surface $\Sigma_{g,n}$ is not compact, we have to work instead with a relative \spinc structure $\bs=(\s,\varphi)$ on $\Sigma_{g,n}$, where $\s=(S^+=L^+\oplus L^-, \rho_2)$ is a \spinc structure on $\Sigma_{g,n}$ and 
\[
\varphi: L^+|_{\coprod U_j}\cong \C
\]
is a trivialization of $L^+$ over the cylindrical ends $\coprod U_j$. As in Section \ref{Sec6}, we still have $L^-=L^+\otimes \Lambda^{0,1}\Sigma_{g,n}$. In this case, the relative Chern classes and the relative degrees of $\bs, L^+$ and $\Lambda^{0,1}\Sigma_{g,n}$ are well-defined. In particular, $\deg \Lambda^{0,1}\Sigma_{g,n}=-(2g-2+n)\leq 0$ and
\[
 c_1(\bs)[\Sigma_{g,n},\partial \Sigma_{g,n}]=2 \deg L^++\deg \Lambda^{0,1}\Sigma_{g,n}.
\]
We are interested in the case when $ d\colonequals \deg L^+\in [0,2g-2+n]$.

\medskip

When $g=0$ and $n=2$, we simply take $\Sigma_{0,2}=\R_s\times (\R/2\pi\Z)_\theta$ to be the product manifold. For the standard relative \spinc structure $\bs_{std}=(S^+_{std},\rho_2)$, we have 
\[
S^+_{std}=\C\oplus \Lambda^{0,1} \Sigma_{0,2}.
\]

We take $(d,\nabla^{LC})$ as the reference \spinc connection. Then for any $\delta\in i\R$ and $\lambda=2i((\im a) ds+(\re a) d\theta)$, there is a unique covariantly constant spinor $(\cPsi_+,\cPsi_-)\in \Gamma(\Sigma_{0,2}, S^+_{std})$ (up to the action of $S^1$) that solves the equations
\begin{equation}\label{EC.3}
\left\{
\begin{array}{rl}
\cPsi_+\otimes\cPsi_-^*&=-\sqrt{2}\lambda^{1,0},\\
\frac{i}{2}\big(|\cPsi_+|^2-|\cPsi_-|^2\big)&\equiv \delta. 
\end{array}
\right.
\end{equation}

\medskip

Back to the case of $\Sigma_{g,n}$, the gauged Landau-Ginzburg model 
\[
(M(\Sigma_{g,n},\bs;\vdelta_*), W_{\lambda},\CG(\Sigma_{g,n}))
\]
now relies on an auxiliary function $\vdelta_*$ and is constructed as follows:
\begin{itemize}
\item $\vdelta_*: \Sigma_{g,n}\to i\R$ is a smooth function such that $\vdelta_*\equiv \delta_j$ for some constant $\delta_j\in i\R$ on each end $U_j\subset \Sigma_{g,n}$;
\item the reference \spinc connection $\cB_*$ is identified with $
(d,\nabla^{LC})
$ under the trivialization $\varphi$ on each cylindrical end $U_j$;
\item the reference spinor $\cPsi_*$ is identified  with the solution $(\cPsi_{+,j},\cPsi_{-,j})$ of \eqref{EC.3} associated to the data $(\delta_j, a_j)$ on each cylindrical end $U_j$ under the trivialization $\varphi$. 

\medskip

\item The K\"{a}hler manifold $M(\Sigma,\bs;\vdelta_*)$ is the configuration space on $\Sigma$:
\[
(\cB_*,\cPsi_*)+L^2_k(\Sigma_{g,n}, iT^*\Sigma_{g,n}\oplus S^+),
\]
for some $k\geq 1$, where $\kappa_*\colonequals (\cB_*,\cPsi_*)$ is the reference configuration defined above.  
\item The complex gauge group acting on $M$ is $\CG_\C(\Sigma_{g,n})=\{u:\Sigma_{g,n}\to \C^*: u-1\in L^2_{k+1}(\Sigma_{g,n})\}$ whose Lie algebra is $L^2_k(\Sigma_{g,n}, \C)$. 
\item The moment map $\mu: M(\Sigma_{g,n},\bs;\vdelta_*)\to L^2_{k-1}(\Sigma_{g,n},i\R)$ is defined by the formula
\begin{align*}
\mu(\cB,\cPsi)&=-\half *_\Sigma F_{\cB^t}+\frac{i}{2}(|\cPsi_+|^2-|\cPsi_-|^2)-\vdelta_*.
\end{align*}
\item the superpotential $W_\lambda$ is the Dirac functional perturbed by the harmonic 1-form $\lambda$, cf. Subsection \ref{Subsec6.2}.
\end{itemize}

This setup will allow us to generalize Propositions \ref{stableW}.

\begin{proposition}\label{stableW3} Suppose that the harmonic 1-form $\lambda \neq 0\in \Omega^1_h(\Sigma_{g,n}, i\R)$ is chosen as above and Assumption \ref{AC.5} holds. Then for the gauged Landau-Ginzburg model 
	$$
	\big(M(\Sigma_{g,n},\bs; \vdelta_*), W_\lambda, \CG(\Sigma_{g,n})\big)
	$$
defined above, the critical locus $\Crit(L)$ consists of 
	$
	\binom{2g-2+n}{d}
	$
	closed free $\CG_\C$-orbits with $d=\deg L^+$. For any $\vdelta\in L^2_{k-1}(\Sigma_{g,n}, i\R)$ and any free $\CG_\C$-orbit $\SO\subset\Crit(L)$, $\mu^{-1}(\vdelta)\cap \SO$ consists of a unique $\CG(\Sigma)$-orbit. Moreover, $W_\lambda$ is a Morse-Bott function. In this sense, we say that $W_\lambda$ is stable and any $\vdelta\in \Lie(\CG)$ is $W_\lambda$-stable. 
\end{proposition}
\begin{proof} The proof follows from the same line of arguments as in Proposition \ref{stableW}, and we point out where the proof is modified. To construct a configuration in $\mu^{-1}(\vdelta)\cap \SO$, we have applied an a priori estimate from Lemma \ref{L6.10} with 
	\[
	w_+=|\cPsi_+| \text{ and } w_-=|\cPsi_-|.
	\]
	Here $\kappa=(\cB, \cPsi_+,\cPsi_-)$ is some representative in a free $\CG_\C$-orbit of $\Crit(L)$. However, Lemma \ref{L6.10} is stated only for closed surfaces. We must adapt its proof to the non-compact surface $\Sigma_{g,n}$. To start, we choose $\kappa$ (within the $\CG_\C$-orbit) such that 
	\[
\kappa|_{U_j}=(\cB_*,\cPsi_*),
	\]
	i.e. it agrees with the reference configuration on each end $U_j\subset \Sigma_{g,n}$. As a result, for some $c>0$, we have  
$
w_+,w_->c
$
on the union $\coprod U_j$. This allows us to derive the estimate:
\[
\|d\alpha_\pm\|_2^2+\int_{\Sigma_{g,n}} \alpha_\pm^2w_\pm^2\geq c_\pm \|\alpha_\pm\|_{L^2_1}^2,
\]
as required in the proof of Lemma \ref{L6.10}. Finally, Trudinger's inequality also holds for non-compact spaces; cf. \cite[Chapter 13, Proposition 4.2]{PDEIII} or \cite[Proposition A.3]{W18}. Now the proof of Proposition \ref{stableW} can proceed with no difficulty.
\end{proof}

\begin{remark} Proposition \ref{stableW3} will allow us to compute the monopole Floer homology of the product manifold $\Sigma_{g,n}\times S^1$ in the second paper \cite{Wang20}. Indeed, the group
	\[
	\HM_*(\Sigma_{g,n}\times S^1,\omega;\bs)
	\]
	will have rank $\binom{2g-2+n}{d}$ for a suitable closed 2-form $\omega$  on $\Sigma_{g,n}\times S^1$ and for the relative \spinc structure $\bs$ with 
	\[
	c_1(\bs)=(2d-2g+2-n)\cdot k\in H^2(\Sigma_{g,n}\times S^1,\partial \Sigma_{g,n}\times S^1;\Z),
	\]
	where $k$ is the Poincar\'{e} dual of $\{pt\}\times S^1$. 
\end{remark}
\begin{remark} The idea of \cite{Wang20} is to complete a 3-manifold with toroidal boundary into a manifold with cylindrical ends. Likewise, one may complete a balanced sutured manifold, which is a 3-manifold with corners, into a manifold with planar ends. Then Proposition \ref{stableW3} would be the replacement of Proposition \ref{stableW}, if one attempts to  construct the sutured Floer homology analytically.
\end{remark}

Finally, we explain why Assumption \ref{AC.5} is a generic condition. Let $\K\colonequals \Lambda^{1,0}\Sigma_g$ be the canonical bundle of $\Sigma_g$. Then we have a long exact sequence:
\begin{equation}\label{EC.1}
\begin{array}{rcccl}
0\to H^0(\Sigma_g, \K)\embed &H^0(\Sigma_g, \K(\D))&\xrightarrow{\Res} &\C^n&\xrightarrow{\sigma} \C\to 0,\\
&\eta &\mapsto &(\Res_{p_j}(\eta))_{1\leq j\leq n}&
\end{array}
\end{equation}
where $\eta$ is a holomorphic 1-form with poles possibly along $\D$ and $\Res_{p_j}(\eta)$ is the residue of $\eta$ at $p_j\in \D$. The penultimate map $\sigma$ takes the sum of all residues. This sequence is exact, because $\dim_\C H^0(\Sigma_g, \K(\D))=g+n-1$ by the Riemann-Roch theorem.
\begin{lemma}\label{LC.9} If $g\geq 1$, then for any vector $v=(b_1,\cdots, b_n)\in \C^n$ with $b_j\neq 0$, $1\leq j\leq n$ and $\sigma(v)=0$, there is a dense open subset of $\Res^{-1}(v)$ such that any $\eta$ in this subset has $2g-2+n$ simple zeros.
\end{lemma}
\begin{proof}[Proof of Lemma] Let $V$ be the subset of $\Res^{-1}(v)$ consisting of sections with simples zeros. $V$ is clearly open. To show that $V$ is dense, consider a form $\eta\in \Res^{-1}(v)$ with zeros at $x_1,\cdots,x_m\subset \Sigma_g\setminus \D$ (possibly with multiplicities). It suffices to find some $\eta'\in H^0(\Sigma_g, \K)$ such that $\eta'(x_j)\neq 0$ for any $1\leq j\leq m$; then the sum $
	\eta+\epsilon\eta'$ lies in $V$
	for any sufficiently small $\epsilon$. Since $g\geq 1$, the canonical bundle $\K$ is base-point free. Indeed, by the Riemann-Roch theorem and Serre duality, for any $q\in \Sigma_g$,
	\[
		\dim_\C H^0(\Sigma_g, \K(-q))=g-1<\dim H^0(\Sigma_g, \K).
	\]
	Now we set $\eta'=\sum \epsilon_j \eta_j'$ for some $\epsilon_j\in \C$ and $\eta_j'\in H^0(\Sigma_g, \K)$ with $\eta_j'(q_j)\neq 0$, $1\leq j\leq m$. 
\end{proof}

The case when $g=0$ is more rigid, since the map $\Res$ in \eqref{EC.1} is injective. For any $v\in \C^n$, one may vary the positive divisor $\D$ instead to achieve Assumption \ref{AC.5}:
\begin{lemma} If $g=0$, then for any vector $v=(b_1,\cdots, b_n)\in \C^n$ with $b_j\neq 0$, $1\leq j\leq n$ and $\sigma(v)=0$, the section $\eta=\Res^{-1}(v)$ has $n-2$ simple zeros in $\Sigma_{0,n}=\CP^1\setminus \D$ for a generic choice of $\D$.
\end{lemma}
\begin{proof}[Proof of Lemma] By an automorphism of $\CP^1$, we set $p_n=\infty$. Then the holomorphic 1-form $\eta=\Res^{-1}(v)$ can be constructed explicitly as 
	\[
	\eta=dz\cdot \sum_{j=1}^{n-1}\frac{b_j}{z-p_j}.
	\]
As we vary $p_1$ in a small neighborhood, the difference
\[
\frac{1}{z-p_1}-\frac{1}{z-p_1'}=\frac{p_1-p_1'}{(z-p_1)(z-p_1')}, p_1'\in \C,
\]
is non-vanishing on $\C$. Now one can argue as in the proof of Lemma \ref{LC.9}.
\end{proof}